\documentclass[10pt, a4paper, notitlepage]{article}
\pdfoutput=1

\usepackage{pgfplots}
\usepackage{biblatex}
\usepackage{bbm}

\usepackage{etoolbox}
\usepackage{tikz}
\usetikzlibrary{calc}
\usetikzlibrary{cd}
\usetikzlibrary{decorations.markings}
\usetikzlibrary{decorations.pathreplacing}
\usetikzlibrary{decorations.pathmorphing}
\usetikzlibrary{decorations.text}
\usetikzlibrary{arrows.meta}
\usetikzlibrary{arrows}
\usetikzlibrary{positioning}

\usepackage{geometry} 
\usepackage{tocloft} 
\usepackage{fancyhdr} 
\usepackage{longtable}
\usepackage{subcaption}
\usepackage{enumitem}
\usepackage{amssymb}
\usepackage{amsmath}
\usepackage{stackrel} 
\usepackage{uniinput}
\usepackage{amsthm}
\usepackage{stmaryrd}
\usepackage{mathtools}
\usepackage{fouridx}
\usepackage{xparse}
\usepackage{bm}
\usepackage{aliascnt}
\usepackage{import}

\usepackage{hyperref}
\theoremstyle{definition}
\newtheorem{theorem}{Theorem}
\let\emph\textbf

\newaliascnt{remark}{theorem}
\newaliascnt{definition}{theorem}
\newaliascnt{proposition}{theorem}
\newaliascnt{lemma}{theorem}
\newaliascnt{corollary}{theorem}
\newaliascnt{example}{theorem}
\newaliascnt{convention}{theorem}
\newaliascnt{todo}{theorem}

\newtheorem{remark}[remark]{Remark}
\newtheorem{definition}[definition]{Definition}
\newtheorem{proposition}[proposition]{Proposition}
\newtheorem{lemma}[lemma]{Lemma}

\newtheorem{example}[example]{Example}
\newtheorem{convention}[convention]{Convention}
\newtheorem{todo}[todo]{Todo}

\aliascntresetthe{remark}
\aliascntresetthe{definition}
\aliascntresetthe{proposition}
\aliascntresetthe{lemma}
\aliascntresetthe{corollary}
\aliascntresetthe{example}
\aliascntresetthe{convention}
\aliascntresetthe{todo}

\numberwithin{equation}{section}
\numberwithin{figure}{section}
\numberwithin{table}{section}
\makeatletter\let\c@table\c@figure\makeatother

\numberwithin{theorem}{section}
\numberwithin{remark}{section}
\numberwithin{definition}{section}
\numberwithin{proposition}{section}
\numberwithin{lemma}{section}
\numberwithin{corollary}{section}
\numberwithin{example}{section}
\numberwithin{convention}{section}
\numberwithin{todo}{section}

\newcommand{\Id}{\operatorname{Id}}
\newcommand{\id}{\mathrm{id}}

\newcommand{\Hom}{\operatorname{Hom}}
\newcommand{\Ker}{\operatorname{Ker}}

\newcommand{\End}{\operatorname{End}}
\newcommand{\Ob}{\operatorname{Ob}}
\newcommand{\cat}[1]{\mathcal{#1}}

\newcommand{\tensor}{\otimes}
\newcommand{\htensor}{\widehat{\otimes}}

\newcommand{\Tw}{\operatorname{Tw}}

\newcommand{\Add}{\operatorname{Add}}
\newcommand{\MC}{\operatorname{MC}}

\def\H{\operatorname{H}}

\newcommand{\HH}{\operatorname{HH}}
\newcommand{\HC}{\operatorname{HC}}

\newcommand{\pmat}[1]{\begin{pmatrix} #1 \end{pmatrix}}
\newcommand{\vspan}{\operatorname{span}}

\newcommand{\isoto}{\xrightarrow{\sim}}

\newcommand{\embeds}{\hookrightarrow}
\def\Im{\operatorname{Im}}

\newcommand{\running}{~|~}

\newcommand{\cone}{\operatorname{cone}}
\newcommand{\Perf}{\operatorname{Perf}}

\newcommand{\M}{\mathcal{M}}

\newcommand{\Sym}{\operatorname{Sym}}

\newcommand{\Spec}{\operatorname{Spec}}

\newcommand{\CF}{\operatorname{CF}}

\newcommand{\restr}{|}

\newcommand{\Disk}{\operatorname{Disk}}
\renewcommand{\Re}{\operatorname{Re}}

\newcommand{\Hilbhor}{\operatorname{Hilb}^{\operatorname{hor}}}
\newcommand{\NH}{\operatorname{NH}}

\newcommand{\HW}{\operatorname{HW}}
\newcommand{\HF}{\operatorname{HF}}

\newcommand{\Comp}{\operatorname{Comp}}

\newcommand{\Map}{\operatorname{Map}}

\newcommand{\landau}{\operatorname{O}}
\newcommand{\pathsplitL}{\operatorname{split}^{\mathrm{L}}}
\newcommand{\pathsplitR}{\operatorname{split}^{\mathrm{R}}}
\newcommand{\pathsplit}{\operatorname{split}}
\newcommand{\Irr}{\operatorname{Irr}}
\newcommand{\Bimod}{\operatorname{Bimod}}
\newcommand{\Barcon}{\operatorname{Bar}}
\newcommand{\sgn}{\operatorname{sign}}
\newcommand{\KLRW}{\operatorname{KLRW}}

\newcommand{\ordFuk}{\operatorname{Fuk}^{\mathcal{O}}}
\newcommand{\ordrelFuk}{\operatorname{relFuk}^{\mathcal{O}}}
\newcommand{\Fuk}{\operatorname{Fuk}}

\newcommand{\relFuk}{\operatorname{relFuk}}

\title{Fukaya categories of Coulomb branches as unique deformations}
\author{Jasper van de Kreeke}
\date{}

\begin{document}

\maketitle

\begin{abstract}
The symplectic geometry of Coulomb branches is complicated and it is particularly difficult to determine their Fukaya categories. Relative Fukaya categories present an approach to circumvent these difficulties by first computing the Fukaya category of the complement of a divisor and then solving a deformation problem. In this paper, we apply this approach to the specific case of horizontal Hilbert schemes by removing their matter divisor and narrowing down the set of possible deformations through an additional $ ℤ^2 $-grading. We utilize an existing description of the Fukaya category after removal of the matter divisor, in particular we use a specific generating Lagrangian and the identification between its endomorphism algebra and the NilHecke algebra. The core of this paper consists of solving the deformation problem, after which we recover the result of Aganagic et al.
\end{abstract}

\tableofcontents

\section{Introduction}
Coulomb branches are a class of symplectic varieties rigorously defined by Braverman-Finkelberg-Nakajima \cite{BFN-II}. Their Fukaya categories were computed in the case of multiplicative Coulomb branches of framed quiver settings via a generating system of Lagrangians and explicit disk-counting by Aganagic et al.~\cite{ADLSZ}. They turn out to be equivalent to the twisted completions of KLRW algebras:
\begin{equation*}
\Tw\KLRW(Γ, d) \embeds \relFuk(\mathcal{M}^× (Γ, d), W).
\end{equation*}
In this paper, we offer an alternative proof in the special case of planar horizontal Hilbert schemes $ Y = \Hilbhor_k (ℂ^* × ℂ) $ where the specific KLRW algebras are the classical NilHecke algebras $ \NH_k $. The idea is to exhibit the relative Fukaya category $ \relFuk(Y_W)_L $ and the algebra $ \NH_k $ as deformations of the common base algebra $ \NH_{k, ħ=0} $. By establishing an additional $ ℤ^2 $-grading on $ \relFuk(Y_W)_L $ and $ \NH_k $ and proving that $ \NH_{k, ħ=0} $ admits only one non-trivial $ ℤ^2 $-graded deformation over $ ℂ⟦ħ⟧_{(0, 1)} $, we conclude that both deformations are automatically equivalent:
\begin{equation*}
\begin{tikzcd}
\relFuk(Y_W)_L \arrow[d, dash, "ℤ^2\text{-defo}"] \arrow[r, "\sim"] & \NH_k \arrow[d, dash, "ℤ^2\text{-defo}"] \\
\Fuk(Y_{O, W})_L \arrow[r, "\sim"] & \NH_{k, ħ=0}.
\end{tikzcd}
\end{equation*}

\paragraph*{Fukaya categories of Coulomb branches}
Coulomb branches originally gained importance as mirror pairs for Higgs branches in 3d mirror symmetry \cite{Intriligator-Seiberg}. Their mathematical formulation was finalized in \cite{BFN-II} and their coherent sheaf categories were investigated by Webster \cite{Webster-I, Webster-II}. As part of the program to categorify knot invariants via mirror symmetry, their Fukaya categories were investigated towards categorification of knot invariants \cite{Aganagic} and the case of $ \Fuk(\mathcal{M}^× (Γ, d), W) $ was worked out explicitly by Aganagic, Danilenko, Li, Shende and Zhou \cite{ADLSZ}.

Their procedure starts with analyzing the definition of the Coulomb branch, which is defined as the spectrum of a K-theoretical algebra. This analysis yields a set of algebra generators and relations, thus establishing a coordinate system on the Coulomb branch itself. With this coordinate set, the Coulomb branch is embedded into a high-dimensional complex space and obtains a Kähler Liouville structure. A specific potential $ W $ dictates the location of stops to be introduced, turning the Liouville manifold into a Liouville sector in the sense of Ganatra-Pardon-Shende \cite{GPS-I}. In particular, this bundle of data defines a wrapped Fukaya category.

There are two important divisors in the Coulomb branch, the root and matter divisors. In the complement of these divisors, there is a coordinate transformation which brings the potential $ W $ into a simpler form. In terms of these new coordinates, for every red and black dot configuration $ θ ∈ \KLRW(Γ, d) $ on the circle $ S^1 $ there is an explicit associated Lagrangians $ T_θ ⊂ \mathcal{M}^× (Γ, d) $. It is verified that the strand diagrams between two configurations are in one-to-one correspondence with the intersection points of the associated Lagrangians in the Fukaya category. This is the basis of the $ A_∞ $-inclusion $ \Tw\KLRW \embeds \Fuk(\M^{×} (Γ, d), W) $.

\paragraph*{Relative Fukaya categories}
A technique to compute a Fukaya category $ \Fuk(X) $ via deformations was introduced by Seidel in 2001 \cite{Seidel-relative}. First, one removes a divisor $ D ⊂ X $ and forms the relative Fukaya category $ \relFuk(X, D) $. In its standard use case, $ \relFuk(X, D) $ is a deformation of $ \Fuk(X \setminus D) $, linear over the deformation base $ ℂ⟦ħ⟧ $ with one deformation parameter $ ħ $. Second, one examines the simpler category $ \Fuk(X \setminus D) $ by means of symplectic procedures and generating systems. Third, one determines the set of deformations of $ \Fuk(X \setminus D) $ up to gauge equivalence and proves that the specific deformation $ \relFuk(X, D) $ falls within a narrow subclass of these deformations. In the ideal case, this procedure makes it possible to identify $ \relFuk(X, D) $ as being equivalent to an a priori determined candidate deformation $ \cat C_{ħ} $. In consequence $ \relFuk(X, D) ≅ \cat C_{ħ} $ and specialization to $ ħ=1 $ yields a complete description of $ \Fuk(X) $.

In the present paper, we apply this technique to the case of horizontal Hilbert schemes $ \Hilbhor_k (ℂ^* × ℂ) $. The idea is to remove the matter divisor, equivalently the big diagonal where at least two of the $ y_i $ coordinates agree. We show that via a standard topological procedure, the relative Fukaya category obtains an additional $ ℤ^2 $-grading. Through an explicit computation of the relevant graded parts of Hochschild cohomology, we observe that the $ ℤ^2 $-grading drastically narrows down the deformation theory. In our process we borrow the description of the Fukaya category after removal of the big diagonal from \cite[section 7]{ADLSZ}. However we circumvent the remaining disk-counting carried out in \cite[section 8]{ADLSZ}. This allows us to recover the inclusion of $ \Tw\NH_k $ into $ \relFuk(\Hilbhor_k (ℂ^* × ℂ), W) $.

\paragraph*{Structure of the paper}
In \autoref{sec:prelimliou}, we formulate a graded version of $ A_∞ $-deformation theory and recall constructions of Liouville sectors and Fukaya categories. In \autoref{sec:prelimhilbhor}, we recall the construction of the horizontal Hilbert scheme $ \Hilbhor_k (ℂ^* × ℂ) $ and its associated Fukaya categories. In \autoref{sec:unique}, we introduce the relative Fukaya category of the horizontal Hilbert scheme as a $ ℤ^2 $-graded deformation over $ ℂ⟦ħ⟧_{(0, 1)} $ and examine the $ ℤ^2 $-graded deformation theory of $ \NH_{k, ħ=0} $. In \autoref{th:unique-main-th}, we conclude that the relative Fukaya category necessarily agrees with $ \NH_k $.

\paragraph*{Acknowledgements}
The author would like to thank Mina Aganagic and Peng Zhou for supervision. The author was supported by NWO Rubicon grant 019.232EN.029.

\section{Preliminaries on deformation theory and Fukaya categories}
\label{sec:prelimliou}
In this section, we recall standard deformation theory of $ A_∞ $-categories and algebras. We recall the formalism of $ L_∞ $-algebras and Hochschild cohomology. We also recall the deformation theory of quiver algebras with reduction system as studied by Barmeier and Wang \cite{Barmeier-Wang}.

In this section, we recall the notions of Liouville manifolds, Liouville domains, Liouville sectors and Kähler Liouville structures. We also recall several standard Liouville structures on low-dimensional spaces for the convenience of the reader.

\subsection{Deformations of $ A_∞ $-categories}
\label{sec:prelim-liou-defo}
In this section we recall $ A_∞ $-categories and their deformation theory. We start by recalling $ A_∞ $-categories, their twisted completion and minimal models. Then we recall completed tensor products and curved $ A_∞ $-deformations. We comment on uncurving procedures, twisted completion and minimal models in the context of $ A_∞ $-deformations. We follow \cite{Paper-IIA} where also more detail can be found.

\paragraph*{$ A_∞ $-categories} We start by recalling that a ($ ℤ $- or $ ℤ/2ℤ $-graded, strictly unital) \emph{$ A_∞ $-category} $ \cat C $ consists of a collection of objects together with $ ℤ $- or $ ℤ/2ℤ $-graded hom spaces $ \Hom(X, Y) $, distinguished identity morphisms $ \id_X ∈ \Hom^0(X, X) $ for all $ X ∈ \cat C $, together with multilinear higher products
\begin{equation*}
μ^k: \Hom(X_k, X_{k+1}) ¤ … ¤ \Hom(X_1, X_2) → \Hom(X_1, X_{k+1}), \quad k ≥ 1
\end{equation*}
of degree $ 2-k $ such that the $ A_∞ $-relations and strict unitality axioms hold. The \emph{twisted completion} $ \Tw\cat C $ is the $ A_∞ $-category $ \cat C $ of virtual chain complexes of objects of $ \cat C $. A \emph{functor} $ F: \cat C → \cat D $ between two $ A_∞ $-categories is a mapping $ F: \Ob(\cat C) → \Ob(\cat D) $ together together with higher components $ F^k $ for $ k ≥ 1 $ which satisfy the $ A_∞ $-functor relations. A functor $ F $ is an \emph{isomorphism} if $ F: \Ob(\cat C) → \Ob(\cat D) $ is a bijection and $ F^1: \Hom_{\cat C} (X, Y) → \Hom_{\cat D} (FX, FY) $ is an isomorphism for all $ X, Y ∈ \cat C $. The functor $ F $ is a \emph{quasi-isomorphism} if $ F: \Ob(\cat C) → \Ob(\cat D) $ is a bijection and $ F^1: \Hom_{\cat C} (X, Y) → \Hom_{\cat D} (FX, FY) $ is a quasi-isomorphism of complexes for every $ X, Y ∈ \cat C $. An $ A_∞ $-category $ \cat C $ is \emph{minimal} if $ μ^1_{\cat C} = 0 $. A \emph{minimal model} of $ \cat C $ is any minimal $ A_∞ $-category $ \cat D $ together with a quasi-isomorphism $ F: \cat D → \cat C $. A minimal model of $ \cat C $ is generically denoted $ \H\cat C $. By the famous Kadeishvili theorem, every $ A_∞ $-category has a minimal model. In fact, a minimal model can be constructed semi-explicitly by sums over trees.

\paragraph*{Completed tensor products} A \emph{deformation base} is a complete local Noetherian unital $ ℂ $-algebra $ B $ with residue field $ B/\mathfrak{m} = ℂ $. By the Cohen structure theorem, every deformation base is of the form $ ℂ⟦x_1, …, x_n⟧ / I $ with $ I $ denoting some ideal. If $ X $ is a vector spaces, then $ B \htensor X = \lim (B/\mathfrak{m}^k \tensor X) $ denotes the completed tensor product over $ ℂ $. For simplicity, we write $ \mathfrak{m}^k X $ to denote the infinitesimal part $ \mathfrak{m}^k X = \mathfrak{m}^k \htensor X ⊂ B \htensor X $. Recall that $ B \htensor X $ is a $ B $-module and comes with the $ \mathfrak{m} $-adic topology, which turns $ B \htensor X $ into a sequential Hausdorff space. For convenience, we may from time to time use expressions like $ x = \landau(\mathfrak{m}^k) $ to indicate $ x ∈ \mathfrak{m}^k X $. Every element in $ B \htensor X $ can be written as a series $ \sum_{i = 0}^∞ m_i x_i $. Here $ m_i $ is a sequence of elements $ m_i ∈ \mathfrak{m}^{→∞} $ and $ x_i $ is a sequence of elements $ x_i ∈ X $. We use the notation $ m_i ∈ \mathfrak{m}^{→∞} $ to indicate that $ m_i ∈ \mathfrak{m}^{k_i} $ for some sequence $ (k_i) ⊂ ℕ $ with $ k_i → ∞ $. A map $ φ: B \htensor X → B \htensor Y $ is \emph{continuous} if it is continuous with respect to the $ \mathfrak{m} $-adic topologies. A map $ φ: (B \htensor X_k) ¤ … ¤ (B \htensor X_1) → B \htensor Y $ is \emph{continuous} if for every $ 1 ≤ i ≤ k $ and every sequence of elements $ x_1, …, \hat x_i, …, x_k $ the map $ μ(x_k, …, -, …, x_1): B \htensor X_i → B \htensor Y $ is continuous. Every $ B $-linear map $ B \htensor X → B \htensor Y $ and every every $ B $-multilinear map $ (B \htensor X_k) ¤ … ¤ (B \htensor X_1) → B \htensor Y $ is automatically continuous. Linear maps $ X → B \htensor Y $ can be uniquely extended to $ B $-linear maps $ B \htensor X → B \htensor Y $ and multilinear maps $ X_k ¤ … ¤ X_1 → B \htensor Y $ can be uniquely extended to $ B $-multilinear maps $ (B \htensor X_k) ¤ … ¤ (B \htensor X_1) → B \htensor Y $. The \emph{leading term} of a $ B $-linear map $ φ: B \htensor X → B \htensor Y $ is the map $ φ_0: X → Y $ given by the composition $ φ_0 = π φ \restr_{X} $, where $ π: B \htensor Y → Y $ denotes the standard projection. If the leading term $ φ_0 $ is injective or surjective, then $ φ $ is injective or surjective itself.

\paragraph*{$ A_∞ $-deformations} Let $ \cat C $ be an $ A_∞ $ category with products $ μ^k $ and let $ B $ a deformation base. An \emph{$ A_\infty $-deformation} of $ \cat C_q $ of $ \cat C $ consists of the same objects as $ \cat C $, hom spaces $ \Hom_{\cat C_q} (X, Y) = B \htensor \Hom_{\cat C} (X, Y) $ for $ X, Y ∈ \cat C $, and $ B $-multilinear products $ μ_q^k $ of degree $ 2 - k $ for $ k ≥ 1 $, and curvature $ μ_{q, X}^0 ∈ \mathfrak{m} \Hom_{\cat C_q}^2 (X, X) $ of degree $ 2 $ for every object $ X ∈ \cat C $, such that $ μ_q $ reduces to $ μ $ once the maximal ideal $ \mathfrak{m} $ is divided out, and $ μ_q $ satisfies the curved $ A_∞ $ ($ cA_∞ $) relations. The deformation is \emph{unital} if the deformed higher products still satisfy the unitality axioms. A \emph{functor of deformed $ A_∞ $-categories} $ F_q: \cat C_q → \cat D_q $ consists of a map $ F_q: \Ob(\cat C) → \Ob(\cat D) $ together with for every $ k ≥ 1 $ a $ B $-multilinear degree $ 1-k $ map $ F^k_q $ and infinitesimal curvature $ F^0_{q, X} ∈ \mathfrak{m} \Hom^1_{\cat D} (F_q X, F_q X) $ for every $ X ∈ \cat C $, such that the curved $ A_∞ $-functor relations hold. If $ \cat C_q $ and $ \cat D_q $ are strictly unital, then $ F_q $ is strictly unital if $ F_q^1 (\id_X) = \id_{F_q X} $ for every $ X \in \cat C $ and $ F_q^{≥2} (…, \id_X, …) = 0 $. Note that $ F_q $ itself is allowed to have curvature. If $ F_q: \cat C_q \to \cat D_q $ is a functor of $ A_\infty $-deformations, then its leading term $ F: \cat C \to \cat D $ is automatically a functor of $ A_\infty $-categories. The functor $ F_q $ is a \emph{quasi-isomorphism} if its leading term $ F: \cat C → \cat D $ is a quasi-isomorphism of $ A_∞ $-categories.

\paragraph*{Uncurving} It is sometimes possible to remove curvature from a deformation $ \cat C_q $ by means of gauging. Uncurving refers to finding a functor $ F_q: \cat C_q' → \cat C_q $ of $ A_∞ $-deformations where $ \cat C_q' $ is another deformation of $ \cat C $ with possibly less curvature and the leading term of $ F_q $ is the identity functor. Let $ r = \{r_X\}_{X ∈ \cat C} $ consist of $ r_X ∈ \mathfrak{m} \End^1_{\cat C} (X) $ for every $ X ∈ \cat C $, then the \emph{uncurving} $ \cat C_q' $ of $ \cat C_q $ by $ r $ is the unique $ A_∞ $-deformation of $ \cat C $ such that $ F_q: \cat C_q' \isoto \cat C_q $ given by $ F_q^0 = r $ and $ F_q^1 = \Id $ and $ F_q^{≥2} = 0 $ is an isomorphism of $ A_∞ $-deformations. The curvature of $ \cat C_q' $ is then $ μ^0_{\cat C_q', X} = \sum_{i ≥ 0}^∞ μ^i_{\cat C_q} (r_X, …, r_X) $. An object $ X ∈ \cat C $ is \emph{uncurvable} if there exists an $ r_X ∈ \mathfrak{m} \End^1(X) $ such that $ X $ has vanishing curvature in the uncurving $ \cat C_q' $ defined by $ r_X $. If $ X, Y ∈ \cat C $ are quasi-isomorphic, then $ X $ is uncurvable if and only if $ Y $ is uncurvable. If $ F_q: \cat C_q → \cat D_q $ is a functor of $ A_∞ $-deformations and $ X ∈ \cat C $ is uncurvable, then $ F_q (X) $ is uncurvable.

\paragraph*{Deformed twisted completion and minimal models}
Let $ \cat C_q $ be a deformation of $ \cat C $. The twisted completion $ \Tw\cat C_q $ is a deformation of $ \Tw\cat C $. Its objects are defined in terms of twisted differentials, but the twisted differentials do not satisfy the Maurer-Cartan equation with respect to the deformed product $ μ_{\cat C_q} $. The failure to satisfy the Maurer-Cartan equation is captured in the object's curvature. It is possible to define a variant $ \Tw'\cat C_q $ of the twisted completion of $ A_∞ $-deformations by allowing additional infinitesimal entries anywhere in the $ δ $-matrix. The objects of $ \Tw'\cat C_q $ are pairs $ (X, δ = δ_0 + δ') $ where $ X ∈ \Add\cat C $ and $ δ_0 ∈ \Hom_{\cat C}^1 (X, X) $ and $ δ' ∈ \mathfrak{m} \Hom_{\cat C}^1 (X, X) $ such that the leading part $ δ_0 $ is upper triangular and satisfies the Maurer-Cartan equation with respect to $ μ_{\cat C} $. The category $ \Tw'\cat C_q $ is an object-cloning deformation of $ \Tw\cat C_q $. The minimal model $ \H\cat C_q $ of $ \cat C_q $ is a deformation of $ \H\cat C $. There exists a quasi-isomorphism $ π_q: \cat C_q → \H\cat C_q $ whose leading term is the standard projection $ π: \cat C → \H\cat C $. The minimal model $ \H\cat C_q $ may have nonzero curvature and differential.

\subsection{Perfect modules of deformed algebras}
In this section, we investigate perfect modules over deformed algebras. We start from an associative algebra $ A $ and a deformation $ A_ħ $ over a deformation base $ B $. The goal is to show that the perfect modules of $ A $ lift to perfect modules of $ A_ħ $.

\begin{definition}
Let $ A $ be an algebra and $ A_ħ $ a deformation. Let $ M $ be an $ A $-module and $ M_ħ $ an $ A_ħ $-module. Then $ M_ħ $ is a \emph{flat deformation} of $ M $ if there is a continuous $ B $-linear isomorphism $ φ: M_ħ \isoto B \htensor M $ which preserves the $ A $-action on zeroth order in the sense that $ φ(am) - a φ(m) ∈ \mathfrak{m} B $ for $ a ∈ A $ and $ m ∈ M_ħ $.
\end{definition}

\begin{lemma}
\label{th:perfect-generation}
Let $ M_ħ $ be a flat deformation of $ M $ by means of $ φ: M_ħ \isoto B \htensor M $. If $ v_1, …, v_k ∈ M $ are generators of $ M $, then their preimages $ φ^{-1} (v_1), …, φ^{-1} (v_k) $ generate $ M_ħ $.
\end{lemma}

\begin{proof}
As a first step, we show that every $ x ∈ \mathfrak{m}^k M_ħ $ can be written in the form $ x = \sum a_i φ^{-1} (v_i) + y $ with $ a_i ∈ \mathfrak{m}^k $ and $ y ∈ \mathfrak{m}^{k+1} M_ħ $. Indeed, we have $ φ(x) ∈ \mathfrak{m}^k \htensor M $. Since $ v_1, …, v_k $ generate $ M $, we can write $ φ(x) = \sum a_i v_i + y $ where $ a_i ∈ \mathfrak{m}^k $ and $ y ∈ \mathfrak{m}^{k+1} \htensor M $.

As a second step, start with arbitrary $ z_0 ∈ M_ħ $. Applying the first step, write $ z_0 = \sum a_i φ^{-1} (v_i) + z_1 $ with $ a_i ∈ \mathfrak{m}^k $ and $ z_1 ∈ \mathfrak{m}^1 M_ħ $. Continuing this way, we obtain a presentation of $ z_0 $ as a linear series combination of the $ φ^{-1} (v_i) $.
\end{proof}

Let us now investigate whether the kernel and image of deformed maps are flat deformations. As it turns out, in some cases they are deformations of the kernel and image of the original maps, while in other cases this does not hold.

\begin{example}
Regard the deformation base $ B = ℂ⟦ħ⟧ $. The map $ f_ħ: A_ħ → A_ħ $ given by $ f_ħ (x) = ħ x $ is a deformation of $ f = 0 $. The kernel of $ f_ħ $ is zero, while the kernel of $ f $ is $ A $. The image of $ f_ħ $ is $ ħA_ħ $ while the image of $ f $ is zero.
\end{example}

\begin{remark}
Let us recall the notion of flat deformations of $ B $-modules. Let $ V $ be a vector space and $ W ⊂ V $ a subspace. Then a completed $ B $-submodule $ W_ħ ⊂ B \htensor V $ is a \emph{flat deformation} of $ W $ if $ W_ħ ∩ \mathfrak{m} \htensor V ⊂ \mathfrak{m} W_ħ $.
\end{remark}

\begin{lemma}
\label{th:perfect-free-im-ker}
Let $ f: A^{⊕k} → A^{⊕l} $ be a map of free $ A $-modules and $ f_ħ: A_ħ^{⊕k} → A_ħ^{⊕l} $ a deformation. If the image of $ f_ħ $ is a flat deformation of the image of $ f $, then so is the kernel.
\end{lemma}

\begin{proof}
We divide the proof into two parts. In the first part, we show that for every $ x ∈ \Ker(f) $ there exists a deformation $ x + x' ∈ \Ker(f_ħ) $ with $ x' ∈ \mathfrak{m} A_ħ^{⊕k} $. In the second part, we show that $ \Ker(f_ħ) ∩ \mathfrak{m} A^{⊕k} ⊂ \mathfrak{m} \Ker(f_ħ) $. We start by writing $ f_ħ (e_i) = a_i + b_i $ with $ a_i ∈ A^{⊕l} $ and $ b_i ∈ \mathfrak{m} A^{⊕l} $.

For the first part, let $ x = (x_1, …, x_l) ∈ \Ker(f) $. Since $ 0 = f(x) = \sum x_i a_i $, we have
\begin{equation*}
f_ħ (x) = \sum x_i (a_i + b_i) = \sum x_i b_i ∈ \mathfrak{m} A^{⊕l} ∩ \Im(f_ħ).
\end{equation*}
By assumption we conclude $ f_ħ (x) ∈ \mathfrak{m} \Im(f_ħ) = f_ħ (\mathfrak{m} A^{⊕l}) $. Therefore there exists $ x' ∈ \mathfrak{m} A^{⊕l} $ such that $ f_ħ (x + x') = 0 $.

For the second part, let $ x ∈ \Ker(f_ħ) ∩ \mathfrak{m} A^{⊕k} $. We assume for simplicity that the deformation base is $ B = ℂ⟦ħ⟧ $. The case of general $ B $ is more complicated and can be solved by means of techniques similar to those in \cite[section 2.3]{Paper-III}. We shall therefore assume that $ B = ℂ⟦ħ⟧ $ and thus $ x ∈ \Ker(f_ħ $ is of the form $ x = ħx' $, thus $ 0 = f_ħ (x) = ħ f_ħ (x') $. Since $ A_ħ^{⊕l} $ is flat, we conclude $ f_ħ (x') = 0 $. This shows $ x = ħx' ∈ ħ\Ker(f_ħ) $ and finishes the proof.
\end{proof}

Let us now treat finitely generated projective modules of $ A_ħ $. They are direct summands of $ A_ħ^{⊕k} $ for some $ k $ and we shall inspect their properties.

\begin{lemma}
Let $ A_ħ^{⊕k} = P_ħ ⊕ P'_ħ $. Then $ \mathfrak{m} A_ħ^{⊕k} = \mathfrak{m} P_ħ ⊕ \mathfrak{m} P'_ħ $. In particular $ P_ħ ∩ \mathfrak{m} A_ħ^{⊕k} ⊂ \mathfrak{m} P_ħ $.
\end{lemma}

\begin{proof}
For simplicity we assume $ B = ℂ⟦ħ⟧ $. For the first part, let $ x ∈ A_ħ^{⊕k} ∩ (ħ) $. Write $ x = ħ y $. Then $ x = π_{P_ħ} (x) + π_{P'_ħ} (x) = ħ π_{P_ħ} (y) + ħ π_{P'_ħ} (y) $. For the second part, let $ x ∈ P_ħ ∩ (ħ) $. Using the first part, we conclude $ x ∈ (ħ P_ħ ⊕ ħ P'_ħ) ∩ P_ħ = ħ P_ħ $. This finishes the proof.
\end{proof}

Let us recall the characterization of projective $ A_ħ $-modules. If $ P $ is any finitely generated projective $ A $-module, we can present $ P $ as a direct summand in $ A^{⊕k} = P ⊕ Q $. In consequence we have the projection $ π_P: A^{⊕k} → A^{⊕k} $ which is an idempotent in the sense that $ π_P^2 = π_P $. In \cite[section 7.4]{Paper-III} it was shown that $ π_P $ can be deformed to an $ A_ħ $-linear map $ π_{P_ħ}: A_ħ^{⊕k} → A_ħ^{⊕k} $ which satisfies $ π_{P_ħ}^2 = π_{P_ħ} $. In consequence, the image $ P_ħ = π_{P_ħ} (A_ħ^{⊕k}) $ is a projective $ A_ħ $-module. It is naturally a deformation of $ P $, and we shall explain explicitly here that it is flat.

\begin{lemma}
\label{th:perfect-perfect-flat}
Let $ P $ be a projective $ A $-module and $ A^{⊕k} = P ⊕ Q $. Let $ π_P: A^{⊕k} → P $ be the projection onto $ P $. Let $ A_ħ^{⊕k} = P_ħ ⊕ Q_ħ $ be a decomposition of $ A_ħ $-modules. Then $ P_ħ $ is a flat deformation of $ P $.
\end{lemma}

\begin{proof}
For simplicity we assume $ B = ℂ⟦ħ⟧ $. Regard any $ p_ħ ∈ P_ħ ∩ ħ A_ħ^{⊕k} $. Then we can write $ p_ħ = ħ p'_ħ $ and thus $ p_ħ = π_{P_ħ} (p_ħ) = ħ π_{P_ħ} (p'_ħ) ∈ ħ P_ħ $. We conclude $ P_ħ ∩ ħ A_ħ^{⊕k} ⊂ ħ P_ħ $ as desired.
\end{proof}

\begin{lemma}
\label{th:perfect-subtraction}
Let $ A_ħ^{⊕k} = P_ħ ⊕ P'_ħ $. Let $ V_ħ ⊂ P_ħ $ and $ V ⊂ P $. If $ V_ħ ⊕ P'_ħ ⊂ A_ħ^{⊕k} $ is a flat deformation of $ V ⊕ P' ⊂ A^{⊕k} $, then $ V_ħ ⊂ A_ħ^{⊕k} $ is a flat deformation of $ V ⊂ A_ħ^{⊕k} $.
\end{lemma}

\begin{proof}
For simplicity we assume $ B = ℂ⟦ħ⟧ $. Let $ v ∈ V $ and regard $ (v, 0) ∈ V ⊕ P' $. There exists $ (v_ħ, p'_ħ) ∈ V_ħ ⊕ P'_ħ $ such that $ (v_ħ, p'_ħ) - (v, 0) ∈ A_ħ^{⊕k} ∩ (ħ) = ħ P_ħ ⊕ ħ P'_ħ $. Therefore $ v - v_ħ ∈ ħ P_ħ ⊂ ħ A_ħ^{⊕k} $.

Second, let $ v ∈ V_ħ ∩ (ħ) $. Regard $ (v, 0) ∈ (V_ħ ⊕ P'_ħ) ∩ (ħ) $. By assumption, we can write $ (v, 0) = ħ (v', p') $ with $ v' ∈ V_ħ $ and $ p' ∈ P'_ħ $. This implies $ v = ħv' ∈ ħ V_ħ $ and finishes the proof.
\end{proof}

\begin{lemma}
Let $ f_ħ: P_ħ → Q_ħ $ be a deformation of $ f: P → Q $. If $ \Im(f_ħ) $ is a flat deformation of $ \Im(f) $, then $ \Ker(f_ħ) $ is a flat deformation of $ \Ker(f) $.
\end{lemma}

\begin{proof}
Let $ A_ħ^{⊕k} = P_ħ ⊕ P'_ħ $ and $ A_ħ^{⊕l} = Q_ħ ⊕ Q'_ħ $. We regard the map $ \bar{f} = f ∘ π_P: A^{⊕k} → A^{⊕l} $ and $ \bar{f}_ħ = f_ħ ∘ π_{P_ħ} : A_ħ^{⊕k} → A_ħ^{⊕l} $.
The image of $ \bar{f}_ħ $ is equal to $ \Im(f_ħ) $ and thus a flat deformation of $ \Im(f) $, equivalently the image of $ \bar{f} $. By \autoref{th:perfect-free-im-ker}, we have $ \Ker(\bar{f_ħ}) $ being a flat deformation of $ \Ker(\bar{f}) $. Finally, our goal is to relate these kernels to $ \Ker(f_ħ) $ and $ \Ker(f) $, respectively. We have $ \Ker(\bar{f_ħ}) = \Ker(f_ħ) ⊕ P'_ħ $ and $ \Ker(\bar{f}) = \Ker(f) ⊕ P' $. By \autoref{th:perfect-subtraction}, we conclude that $ \Ker(f_ħ) $ is a flat deformation of $ \Ker(f) $. This proves the claim.
\end{proof}

\begin{lemma}
\label{th:perfect-ker-flat}
Let $ P $ be a finitely generated projective $ A $-module and $ P_ħ $ a deformation. Let $ M $ be an $ A $-module and $ M_ħ $ a flat deformation. If $ f_ħ: P_ħ → M_ħ $ is surjective, then $ \Ker(f_ħ) $ is a flat deformation of $ \Ker(f) $.
\end{lemma}

\begin{proof}
For simplicity we assume $ B = ℂ⟦ħ⟧ $. We start by showing that $ \Ker(f_ħ) $ is a deformation of $ \Ker(f) $ as submodules of $ P_ħ $. Let $ x ∈ \Ker(f) $. Then $ f_ħ (x) = ħ y $. Since $ y ∈ \Im(f_ħ) $, we can write $ y = f_ħ (x') $. Therefore $ f_ħ (x - ħx') = ħ y - ħ y = 0 $. This shows that $ \Ker(f_ħ) $ is a deformation of $ \Ker(f) $ as submodules of $ P_ħ $.

We continue by showing that $ \Ker(f_ħ) $ is a flat deformation. Let $ x ∈ \Ker(f_ħ) ∩ ħP_ħ $. We can write $ x = ħ x' $ with $ x' ∈ P_ħ $. We have $ 0 = f_ħ (x) = ħ f_ħ (x') $ and since $ M_ħ $ is flat we conclude $ f_ħ (x') = 0 $. Therefore $ x ∈ ħ \Ker(f_ħ) $. We conclude $ \Ker(f_ħ) ∩ ħ P_ħ ⊂ ħ \Ker(f_ħ) $. This finishes the proof.
\end{proof}

\begin{lemma}
A module morphism $ f_ħ: M_ħ \to N_ħ $ which only takes values in $ \mathfrak{m}^k N_ħ $ lies in $ \mathfrak{m}^k \Hom(M_ħ, N_ħ) $.
\end{lemma}

\begin{proof}
In case $ B = ℂ⟦ħ⟧ $ this is obvious. In the other cases, this is classic but a bit more tedious to spell out.
\end{proof}

\begin{lemma}
\label{th:perfect-lift-surjective}
Let $ P $ be a finitely generated projective $ A $-module and $ P_ħ $ a flat deformation. Let $ M $ be an $ A $-module and $ M_ħ $ a flat deformation. Then $ \Hom_{A_ħ} (P_ħ, M_ħ) ≅ B \htensor \Hom_A (P, M) $. An element $ f_ħ \in  \Hom_{A_ħ} (P_ħ, M_ħ) $ is a lift of $ f \in \Hom_A (P, M) $ if and only if its leading term with respect to this isomorphism is $ f $. For any lift $ f_ħ $, we have that $ f_ħ $ is surjective if and only if $ f $ is.
\end{lemma}

\begin{proof}
Let us write $ A^{⊕k} = P ⊕ Q $ and $ A_ħ^{⊕k} = P_ħ ⊕ Q_ħ $. Now define the map $ ψ: B \htensor \Hom_A (P, M) \to \Hom_{A_ħ} (P_ħ, M_ħ) $ as the continuous $ B $-linear extension of the following assignment: An element $ f \in Hom_A (P, M) $ is sent to $ ψ(f)(e_i) = f(e_i) \in M $. Note that the resulting $ ψ(f) $ is automatically an $ A_ħ $-module morphism.

We claim that $ ψ $ is surjective. It suffices to show that any morphism $ f_ħ \in \Hom_{A_ħ} (P_ħ, M_ħ) $ agrees with an element in the image of $ ψ $ up to pointwise terms in $ \mathfrak{m} $. Let $ π: M_ħ \to M $ be the standard projection and define $ f(e_i) = π(f_ħ (e_i)) $ for $ i = 1, …, k $. We then have that $ ψ(f)(e_i) - f_ħ (e_i) = π(f_ħ (e_i)) - f_ħ (e_i) \in \mathfrak{m} \htensor M $. By an adoption of \autoref{th:perfect-generation}, we conclude that $ ψ $ is surjective.

We claim that $ ψ $ is injective. For simplicity we assume $ B = ℂ⟦ħ⟧ $. Regard an element of the form $ f = \sum b_i ħ^i f_i ∈ B \htensor \Hom_A (P, M) $ and assume $ ψ(f) = 0 $. We show that $ b_0 f_0 = 0 $. Pick any $ p ∈ P $ and a $ p_ħ ∈ P_ħ $ such that $ p - p_ħ ∈ ħ A_ħ^{⊕k} $. Since $ ψ(\sum b_i ħ^i f_i)(p_ħ) = 0 $, we conclude $ \sum b_i ħ^i f_i (p) ∈ ħ M_ħ $. Since $ p $ is arbitrary, we conclude $ b_0 f_0 = 0 $. Therefore we can write $ f = ħ f' $ and $ 0 = ψ(f) = ħ ψ(f') $, thus $ ψ(f') = 0 $. Continuing this way we conclude $ f = 0 $. Thus $ ψ $ is injective and we finish the proof.
\end{proof}

\begin{example}
Consider the case that $ M_ħ = M $, where $ ħ $ acts by zero. This $ A_ħ $-module is not a flat deformation of $ M $ and therefore \autoref{th:perfect-lift-surjective} does not apply. In fact, $ ψ $ is not injective, as for instance $ ψ(ħf) = 0 $.
\end{example}

\begin{lemma}
\label{th:perfect-image-extension}
Let $ A_ħ^{⊕k} = P_ħ ⊕ P'_ħ $ and $ A_ħ^{⊕l} = Q_ħ ⊕ Q'_ħ $. Let $ I_ħ ⊂ Q_ħ $ be a submodule and flat deformation of $ I ⊂ Q $. Let $ f: P → Q $ be an $ A $-linear map with image $ I = \Im(f) $. Then there exists a deformation $ f_ħ: P_ħ → Q_ħ $ with image $ I_ħ $.
\end{lemma}

\begin{proof}
We split the proof into four parts. In the first part, we construct the map $ f_ħ $. In the second part, we show that $ f_ħ $ is a deformation of $ f $. In the third part, we prove that any $ x ∈ ħ^k I_ħ $ can be written in the form $ x = x' + x'' $ with $ x' ∈ ħ^k \Im(f_ħ) $ and $ x'' ∈ ħ^{k+1} I_ħ $. In the fourth part, we conclude that $ \Im(f_ħ) = I_ħ $.

For the first part, regard the map $ \bar{f} = f ∘ π_P: A^{⊕k} → I $. Define $ v_i = \bar{f} (e_i) $ for $ i = 1, …, k $. By assumption that $ I_ħ $ is a deformation of $ I $, there exist elements $ w_i ∈ \mathfrak{m} A_ħ^{⊕l} $ such that $ v_i + w_i ∈ I_ħ $ for all $ i $. Now define the map $ \bar{f}_ħ: A_ħ^{⊕k} → I_ħ $ given by $ \bar{f}_ħ (e_i) = v_i + w_i ∈ I_ħ $. Define the map $ f_ħ = \bar{f}_ħ \restr_{P_ħ}: P_ħ → Q_ħ $. By construction we have $ \Im(f_ħ) ⊂ I_ħ $.

For the second part, pick any $ p ∈ P $ and $ p' ∈ \mathfrak{m} A_ħ^{⊕k} $ such that $ p + p' ∈ P_ħ $. Then we indeed obtain
\begin{equation*}
f_ħ (p + p') - f(p) = (\bar{f}_ħ (p) - \bar{f} (p)) + f_ħ (p') ∈ ħ A_ħ^{⊕l}.
\end{equation*}
For the third part, start by writing $ x = ħ^k x_0 $ with $ x_0 ∈ I_ħ $. Then we can write $ x_0 = f(p) + y $ with $ p ∈ P $ and $ y ∈ ħ A_ħ^{⊕l} $. Pick $ p' ∈ ħ A_ħ^{⊕k} $ such that $ p + p' ∈ P_ħ $. We have
\begin{equation*}
f_ħ (p + p') - x_0 = (f_ħ (p) - f(p)) + (f(p) - x_0) + f_ħ (p') ∈ ħ A_ħ^{⊕l}.
\end{equation*}
We observe that this difference lies in $ I_ħ $ as well and by flatness of $ I_ħ $ we conclude that the difference is of the form $ ħ z $ with $ z ∈ I_ħ $. We conclude the claim by observing
\begin{equation*}
x = ħ^k x_0 = ħ^k f_ħ (p + p') - ħ^{k+1} z ∈ ħ^k \Im(f_ħ) + ħ^{k+1} I_h.
\end{equation*}
For the fourth part, let $ x ∈ I_ħ $. Then through inductive applications of the third part we can write $ x = \sum x_i $ with $ x_i ∈ ħ^i \Im(f_ħ) $. Since $ P_ħ $ is complete, this proves the claim.
\end{proof}

\begin{proposition}
Let $ M $ be an $ A $-module and let $ … \overset{d^1}{\to} P^0 \overset{d^0}{\to} M $ be a resolution by finitely generated projective modules. Then there exist projective modules $ P_ħ^i $ with maps $ d^i_ħ: P_ħ^i \to P_ħ^{i-1} $ such that $ P_ħ^i $ are flat deformations of $ P^i $, the maps $ d^i_ħ $ are deformations of $ d^i $ and $ P_ħ^{•} $ is a resolution of $ M_ħ $.
\end{proposition}

\begin{proof}
We start by picking lifts $ P_ħ^i $ of $ P^i $ for every $ i ∈ ℕ $. By \autoref{th:perfect-lift-surjective}, we can pick a lift $ d_ħ^0: P_ħ^0 → M_ħ $ which is automatically surjective. By \autoref{th:perfect-ker-flat}, we have that $ \Ker(d_ħ^0) ⊂ P_ħ^0 $ is a flat deformation of $ \Ker(d^0) ⊂ P^0 $. Since $ \Ker(d^0) = \Im(d^1) $, by \autoref{th:perfect-image-extension} there exists a lift $ d^1_ħ: P_ħ^1 → P_ħ^0 $ with $ \Im(d^1_ħ) = \Ker(d^0_ħ) $. The desired projective resolution is obtained by continuing this way inductively.
\end{proof}

In terms of the terminology established in \cite{Paper-III}, the category $ \Perf A_ħ $ is a loose object-cloning deformation of $ \Perf A $.

\subsection{Graded deformation theory}
\label{sec:prelimliou-gradeddefo}
In this section, we formulate a graded version of $ A_∞ $-deformation theory. We start by defining graded deformation bases and tensor products. We then examine graded $ L_∞ $-algebras and their graded Maurer-Cartan elements. Finally we recall graded $ A_∞ $-categories and examine their Maurer-Cartan elements with the help of their graded Hochschild DGLAs.

\begin{definition}
A \emph{grading group} is an abelian group $ (A, +) $. A \emph{completed $ A $-graded deformation base} is a deformation base $ (B, \mathfrak{m}) $ together with an $ A $-grading on each quotient algebra $ B/\mathfrak{m}^i $ such that the projection maps $ B/\mathfrak{m}^{i+1} → B/\mathfrak{m}^i $ are $ B $-linear.

Let $ M $ be an $ A $-graded vector space. A \emph{completed $ A $-graded $ B $-module} is a $ B $-module $ N $ together with an $ A $-grading on each quotient $ N/\mathfrak{m}^i N $ such that the projection maps $ N/\mathfrak{m}^{i+1} N → N/\mathfrak{m}^i N $ are $ A $-homogeneous. An element $ n ∈ N $ is \emph{$ A $-homogeneous} of degree $ a ∈ A $ if each projection of $ n $ along the projection maps $ N → N/\mathfrak{m}^i N $ is $ A $-homogeneous of degree $ a $.

If $ M $ is any $ A $-graded vector space, we have the standard completed $ A $-graded $ B $-module $ B \htensor M = \lim (B/\mathfrak{m}^i) ¤ M $. A completed $ A $-graded $ B $-module $ N $ is \emph{free} if there exists an $ A $-graded vector space $ M $ together with a $ B $-linear isomorphism $ φ: N \isoto B \htensor M $ such that the induced maps $ φ_i: N/\mathfrak{m}^i N = (B/\mathfrak{m}^i) ¤ M $ are $ A $-homogeneous isomorphisms.
\end{definition}

\begin{remark}
The notion of a completed $ A $-graded deformation base is different from the notion of deformation base with an $ A $-grading. While an $ A $-graded algebra $ B $ possesses a direct sum decomposition $ B = \bigoplus_{a ∈ A} B_a $ as vector spaces, a completed $ A $-graded deformation base does not possess such a direct sum decomposition. For example, regard the deformation base $ B = ℂ⟦h⟧ $. It is a completed $ A $-graded deformation base with assigning $ \deg(ħ) = 1 ∈ ℤ $. However, this assignment does not extend to a $ ℤ $-grading on $ B $, since there is no vector space decomposition $ B = \bigoplus_{a ∈ ℤ} B_a $ where each $ B_a $ is homogeneous in $ ħ $ of degree $ a $.
\end{remark}

\begin{example}
If $ B_0 $ is an $ A $-graded noetherian unital commutative algebra and $ \mathfrak{m} ⊂ B_0 $ is a maximal ideal, then the completion $ B = \hat{B_0}_{\mathfrak{m}} $ is a completed $ A $-graded deformation base. For instance, the $ ℤ $-graded deformation base $ ℂ⟦ħ⟧ $ is the completion of $ ℂ[ħ] $. It is unclear whether every completed $ A $-graded deformation base can be obtained this way.
\end{example}

\begin{lemma}
\label{th:prelim-gradeddefo-splithigher}
Let $ i ≥ 0 $ and $ 0 ≤ k ≤ i $ and let $ b ∈ \mathfrak{m}^k $. Then there is a decomposition $ b = \sum_{j = 1}^N b_j + m $ where $ b_j ∈ \mathfrak{m}^k $ and $ π_i (b_j) ∈ B/\mathfrak{m}^i $ is homogeneous and $ m ∈ \mathfrak{m}^i $.
\end{lemma}

\begin{proof}
Regard the projection $ π_i (b) ∈ B/\mathfrak{m}^i $ and write $ π_i (b) = \sum_{j = 1}^N c_j $ where $ c_j ∈ B / \mathfrak{m}^i $ are homogeneous of pairwise different degrees. Now choose any lifts $ b_j $ of $ c_j $ such that $ π_i (b_j) = c_j $. Putting $ m = b - \sum_{j = 1}^N b_j $, we have $ π_i (m) = 0 $ thus $ m ∈ \mathfrak{m}^i $. Finally, regard the sum $ π_k (b) = \sum_{j = 1}^N π_k (b_j) + π_k (m) $. Since $ π_k (b) = π_k (m) = 0 $ and each element $ b_j ∈ B / \mathfrak{m}^k $ is homogeneous of different degree, we conclude that each $ π_k (b_j) $ is homogeneous of different degree and thus $ π_k (b_j) = 0 $ for all $ j $. This shows $ b_j ∈ \mathfrak{m}^k $ and finishes the proof.
\end{proof}

\begin{lemma}
\label{th:prelim-gradeddefo-lifting}
If $ b ∈ B / \mathfrak{m}^i $ is homogeneous, then there exists a homogeneous lift in $ B / \mathfrak{m}^{i+1} $ of the same degree.
\end{lemma}

\begin{proof}
Denote by $ a ∈ A $ the degree of $ b $ and recall that $ b $ is nonzero by assumption of homogeneity. Let $ \tilde b ∈ B $ be any lift of $ b $. By \autoref{th:prelim-gradeddefo-splithigher} we can write $ \tilde b = \sum_{j = 1}^N b_j + m $ where $ π_{i+1} (b_j) $ are homogeneous elements of pairwise distinct degrees $ d_j $ and $ m ∈ \mathfrak{m}^{i+1} $. Regard the sum
\begin{equation*}
b = π_i (\tilde b) = \sum_{j = 1} π_i (b_j) + π_i (m).
\end{equation*}
Since $ b $ is homogeneous of degree $ d $ and $ π_i (m) = 0 $, we that for all $ j $ with $ d_j ≠ d $ we have $ π_i (b_j) = 0 $, thus $ b_j ∈ \mathfrak{m}^i $. Now define 
\begin{equation*}
b' = \tilde b - \sum_{\substack{j = 1, …, N \\ d_j ≠ d}} b_j - m ∈ B.
\end{equation*}
We have
\begin{equation*}
π_i (b') = b - \sum_{\substack{j = 1, …, N \\ d_j ≠ d}} π_i (b_j) - π_i (m) = b.
\end{equation*}
and
\begin{equation*}
π_{i+1} (b') = \sum_{j = 1}^N π_{i+1} b_j - \sum_{\substack{j = 1, …, N \\ d_j ≠ d}} π_{i+1} (b_j).
\end{equation*}
The result of this sum is at most one element of the form $ π_{i+1} (b_j) $ that is of degree $ d $. This finishes the proof.
\end{proof}

\begin{lemma}
Any element $ b ∈ B $ can be written as a finite or countably infinite sum $ b = \sum b_j $ where $ b_j ∈ \mathfrak{m}^{→ ∞} $ and each $ b_j $ is homogeneous. If $ b ∈ \mathfrak{m}^k $ for some $ k ≥ 0 $, it can be achieved that $ b_j ∈ \mathfrak{m}^{≥k, → ∞} $.
\end{lemma}

\begin{proof}
We divide the proof into three parts. In the first part, we construct auxiliary data. In the second part, we construct the desired data $ b_j $. In the third part, we comment on the second statement.

For the first part of the proof, we construct the following data by induction over $ i ≥ 1 $: a sequence $ (N_i)_{i ≥ 1} $ together with elements $ b_{i, j}^{(t)} ∈ \mathfrak{m}^t $ for $ j = 1, …, N_t $ and $ t = 0, …, i-1 $ and elements $ m_i ∈ \mathfrak{m}^i $ such that $ π_i (b_{i,j}^{(t)}) ∈ B / \mathfrak{m}^i $ is homogeneous, the difference $ b_{i, j}^{(t)} - b_{i+1, j}^{(t)} $ lies in $ \mathfrak{m}^i $ and the following decomposition holds:
\begin{equation*}
b = \sum_{j = 1}^{N_0} b_{i,j}^{(0)} + … + \sum_{j = 1}^{N_{i-1}} b_{i,j}^{(i-1)} + m_i.
\end{equation*}
For $ i = 1 $, split $ π_1 (b) = \sum_{j = 1}^{N_0} π_1 (b_{1,j}^{(0)}) $ where $ π_1 (b_{1,j}^{(0)}) $ is homogeneous. Then define $ m_1 = b - \sum_{j = 1}^{N_0} b_{1,j}^{(0)} $.

Assume now that the data has been constructed up to index $ i $. For every $ 0 ≤ s ≤ i-1 $ and every $ j = 1, …, N_s $ recall that $ π_i (b_{i,j}^{(s)}) $ is homogeneous and by \autoref{th:prelim-gradeddefo-lifting} there exists an element $ b_{i+1,j}^{(s)} $ such that $ π_{i+1} (b_{i+1,j}^{(s)}) $ is homogeneous of the same degree and $ π_i (b_{i+1,j}^{(s)}) = π_i (b_{i,j}^{(s)}) $. Define
\begin{equation*}
m' = m_i + \sum_{s = 0}^{i-1} \sum_{j = 1}^{N_s} (b_{i,j}^{(s)} - b_{i+1,j}^{(s)}) ∈ \mathfrak{m}^i.
\end{equation*}
We conclude that
\begin{equation*}
b = \sum b_{i+1,j}^{(0)} + … + \sum b_{i+1,j}^{(i-1)} + m'.
\end{equation*}
Now regard $ m' ∈ \mathfrak{m}^i $ and decompose $ π_{i+1} (m') = \sum_{j = 1}^{N_i} π_{i+1} (b_{i+1,j}^{(i)}) $ with $ π_{i+1} (b_{i+1,j}^{(i)}) ∈ B / \mathfrak{m}^{i+1} $ homogeneous of pairwise distinct degrees. Now write $ m_{i+1} = m' - \sum b_{i+1,j}^{(i)} $ and deduce $ b_{i+1,j}^{(i)} ∈ \mathfrak{m}^i $ and $ m_{i+1} ∈ \mathfrak{m}^{i+1} $. We further have
\begin{equation*}
b = \sum b_{i+1,j}^{(0)} + … + \sum b_{i+1,j}^{(i-1)} + \sum b_{i+1,j}^{(i)} + m_{i+1}.
\end{equation*}
This decomposition satisfies the requirements of the induction statement for $ i+1 $. This finishes the induction argument.

For the second part of the proof, we shall finish the construction of the sequence $ b_j $ for a given element $ b $. Pick the data $ (N_s) $, $ (b_{i,j}^{(s)}) $ constructed in the first part of the proof. For $ s ≥ 0 $ and $ j = 1, …, N_s $ define $ b_{s, j} = \lim_{i → ∞} b_{i,j} $. This limit exists and is homogeneous by construction. We have
\begin{equation*}
b = \sum_{s = 0}^∞ \sum_{j = 1}^{N_s} b_{s, j}.
\end{equation*}
In this sum, the elements $ b_{s, j} $ are homogeneous and we have $ b_{s, j} ∈ \mathfrak{m}^s $. This finishes the second part of the proof.

For the third part of the proof, assume that $ b ∈ \mathfrak{m}^k $ for some $ k ≥ 0 $. Then by [lemma] the construction of the elements $ b_{i, j}^{(s)} $ can be carried out from the beginning with elements that lie purely in powers of $ \mathfrak{m} $ of order at least $ k $. This finishes the proof.
\end{proof}

\begin{lemma}
\label{th:prelimdefo-gradedtensor-tensordecomp}
Any element $ x ∈ B \htensor V $ can be written as a finite or countably infinite sum $ x = \sum b_i ¤ v_i $ where $ b_i ∈ \mathfrak{m}^{→ ∞} $ are homogeneous and $ v_i ∈ V $ are homogeneous.
\end{lemma}

\begin{proof}
As a first step, we prove that for pure tensors of the form $ m ¤ v $ where $ m ∈ \mathfrak{m}^k $ and $ v ∈ V $, the statement holds with the additional data that $ b_i ∈ \mathfrak{≥k, → ∞} $. Write $ m = \sum m_j $ where $ m_j ∈ \mathfrak{m}^{≥k, → ∞} $ are homogeneous elements and write $ v = \sum v_l $ where $ v_l ∈ V $ are homogeneous elements. Then we have $ m ¤ v = \sum_j \sum_l m_j ¤ v_l $.

As a second step, let $ x ∈ B \htensor V $ be an arbitrary element. Write $ x = \sum b_j ¤ v_j $ where $ b_j ∈ \mathfrak{m}^{→ ∞} $ and $ v_j ∈ V $. Now decompose $ b_j ¤ v_j $ following the first step and reorder the sum. The result is a presentation of $ x $ as a countable sum of the desired form.
\end{proof}

Let us now treat graded $ L_∞ $-algebras and their Maurer-Cartan elements. The common situation entails that we are given an $ L_∞ $-algebra and a deformation base that are both graded over the same grading group. The task at hand is to identify whether two graded Maurer-Cartan elements are gauge-equivalent via a graded gauge-equivalence or not.

\begin{definition}
Let $ A $ be a grading group. An \emph{$ A $-graded $ L_∞ $-algebra} is an $ L_∞ $-algebra $ L = \bigoplus_{i ∈ ℤ} L^i $ together with an additional grading $ L^i = \bigoplus_{a ∈ A} L^i_a $ for $ i ∈ ℤ $, such that the higher product $ l^k $ is homogeneous with respect to $ A $ for every $ k ≥ 1 $. A \emph{(free) completed $ A $-graded $ B $-linear $ L_∞ $-algebra} is an $ L_∞ $-algebra $ L = \bigoplus_{i ∈ ℤ} L^i $ such that every $ L^i $ is equipped with the structure of (free) completed $ A $-graded $ B $-module and the brackets $ l^k $ are $ B $-linear and on each quotient $ \prod_{j = 1}^k L / \mathfrak{m}^{i_j} L $ the bracket $ l^k $ is $ A $-homogeneous.

Let $ B $ be a completed $ A $-graded deformation base and $ L $ be an $ A $-graded $ L_∞ $-algebra. Then the associated free completed $ A $-graded $ B $-linear $ L_∞ $-algebra is $ \mathfrak{m} \htensor L $ with its natural continuous multilinear extension of brackets. An \emph{$ A $-graded Maurer-Cartan element} of $ L $ over $ B $ is an element $ ν ∈ \MC(L, B) ⊂ B \htensor L^1 $ that is $ A $-homogeneous of degree $ 0 ∈ A $ in $ B \htensor L^1 $. The set of $ A $-graded Maurer-Cartan elements of $ L $ over $ B $ is denoted $ \MC(L, B)_0 $. Two $ A $-graded Maurer-Cartan elements $ ν, ν' ∈ \MC(L, B)_0 $ are \emph{$ A $-graded gauge-equivalent} if there exists an element $ η ∈ \mathfrak{m} \htensor L^0 $ such that $ η $ is $ A $-homogeneous of degree $ 0 ∈ A $ and we have $ η · ν = ν' $ where $ · $ denotes the gauge action of $ B \htensor L $.
\end{definition}

\begin{lemma}
\label{th:prelim-defo-gradedMCpush}
Let $ A $ be a grading group.
\begin{enumerate}
\item Let $ L $ be an $ A $-graded $ L_∞ $-algebra. Then there exists an $ A $-graded minimal $ L_∞ $-algebra $ \H L $ together with a $ A $-homogeneous $ L_∞ $-quasi-isomorphisms $ i: \H L → L $ and $ π: L → \H L $.
\item Let $ B $ be a completed $ A $-graded deformation base. If $ φ: L → L' $ is an $ A $-homogeneous $ L_∞ $-morphism and $ ν ∈ \MC(L, B)_0 $, then $ φ^{\MC} (ν) ∈ \MC(L', B)_0 $.
\end{enumerate}
\end{lemma}

\begin{proof}
Regard the first statement. Define $ \H L $ to be the standard minimal model for $ L $ following the Kadeishvili construction, as in \cite[Chapter 6, 3.3.2]{Kontsevich-Soibelman}, \cite[Section 3.2]{Bocklandt-book} and \cite{Paper-IIA}. Since the differential $ d_L $ is $ A $-homogeneous, the cohomology of $ L $ in every degree is an $ A $-graded vector space. The codifferential $ h $ can be chosen to be $ A $-homogeneous, and the homological decomposition $ H ⊕ I ⊕ R $ is then $ A $-homogeneous as well. The expression for the higher products $ l^k_{\H L} $ by means of trees is $ A $-homogeneous, and the maps $ π: L → \H L $ and $ i: \H L → L $ as well. This proves the first statement.

Regard the second statement. It is a standard fact that $ φ^{\MC} (ν) $ is a Maurer-Cartan element. We have
\begin{equation*}
φ^{\MC} (ν) = \sum_{j = 0}^∞ \frac{φ^j (ν, …, ν)}{j!}.
\end{equation*}
Since $ ν $ is $ A $-homogeneous of degree zero as element of $ B \htensor L $ and each $ φ^j $ is $ A $-homogeneous, we conclude that $ φ^{\MC} (ν) $ is $ A $-homogeneous of degree zero in $ B \htensor L' $. This finishes the proof.
\end{proof}

The following is a graded version of a folklore statement for which we know no proof.

\begin{lemma}
\label{th:prelimliou-ainfty-pushpullbasic}
Let $ L $ be an $ A $-homogeneous $ L_∞ $-algebra and $ π: L → \H L $ and $ i: \H L → L $ its $ A $-homogeneous minimal model projections and inclusions. Let $ ν ∈ \MC(L, B)_0 $. Then $ ν $ and $ i^{\MC} (π^{\MC} (ν)) $ are $ A $-homogeneously gauge equivalent.
\end{lemma}

Let us now treat graded $ A_∞ $-categories. Simply speaking, a graded $ A_∞ $-category has a grading on its hom spaces in addition to the structural $ ℤ $- or $ ℤ/2ℤ $-grading. We observe that the grading of the $ A_∞ $-category gives rise to an additional grading on its Hochschild DGLA as well.

\begin{definition}
Let $ A $ be a grading group. An \emph{$ A $-graded $ A_∞ $-category} is an $ A_∞ $-category $ \cat C $ with additional $ A $-gradings on each graded component $ \Hom^i (X, Y) $ for $ i ∈ ℤ $ and $ X, Y ∈ \cat C $, such that the $ A_∞ $-products are $ A $-homogeneous. The \emph{$ A $-graded Hochschild DGLA} of $ \cat C $ is the $ A $-graded $ L_∞ $-algebra $ \HC(\cat C) $ with $ A $-grading inherited from $ \cat C $.
\end{definition}

\begin{remark}
Let $ \cat C $ be an $ A $-graded $ A_∞ $-category and $ B $ a completed $ A $-graded deformation base. Then $ A $-graded Maurer-Cartan elements $ ν, ν' ∈ \MC(L, B)_0 $ are exactly the same as $ A $-graded (curved) $ A_∞ $-deformations of $ \cat C $. An $ A $-homogeneous gauge equivalence $ η · ν = ν' $ amounts precisely to an isomorphism functor between $ (B \htensor \cat C, μ_{\cat C} + ν) $ and $ (B \htensor \cat C, μ_{\cat C} + ν') $ that is $ A $-homogeneous in the sense that every reduction modulo $ \mathfrak{m}^i $ is $ A $-homogeneous.
\end{remark}

When $ L $ is the Hochschild DGLA of a graded $ A_∞ $-category, the push-pull lemma \autoref{th:prelimliou-ainfty-pushpullbasic} has a graded version which we state in \autoref{th:prelim-defo-pushpull}. In this lemma, the functor $ F $ is a priori an infinitesimally curved functor.

\begin{lemma}
\label{th:prelim-defo-pushpull}
Let $ \cat C $ be an $ A $-homogeneous $ A_∞ $-category and $ ν ∈ \MC(L, B)_0 $. Then there is an $ A $-homogeneous gauge equivalence $ F: ν → i^{\MC} (π^{\MC} (ν)) $.
\end{lemma}

Let us fix the following terminology of morphisms. It is directly inherited from the notion of morphisms of local rings, which are required to satisfy $ ψ(\mathfrak{m}_B) ⊂ \mathfrak{m}_{B'} $.

\begin{definition}
Let $ B, B' $ be deformation bases. Then a \emph{morphism of deformation bases} $ ψ: B → B' $ is an algebra morphism such that $ ψ(\mathfrak{m}_B) ⊂ \mathfrak{m}_{B'} $.
\end{definition}

\begin{lemma}
\label{th:prelim-defo-substitution}
Let $ B $ be a deformation base and $ ψ: B → B $ be an $ A $-homogeneous morphism of deformation bases. Let $ \cat C $ be an $ A $-homogeneous $ A_∞ $-category and let $ ν ∈ \MC(\HC(\cat C), B)_0 $. Then there exists a unique element $ ν_ψ ∈ \MC(\HC(\cat C), B)_0 $ such that there is a strict $ A $-homogeneous $ A_∞ $-functor $ F: ν → ν_ψ $ where $ F^1: B \htensor \Hom(X, Y) → B \htensor \Hom(X, Y) $ is given by the $ B $-linear continuous extension of $ ψ ¤ \Id $.
\end{lemma}

\begin{proof}
We start by noting $ F^1 $ is bijective as explained in \cite{Paper-IIA}. We define $ ν_ψ ∈ \mathfrak{m} \htensor \HC^1 (\cat C) $ via the following rule:
\begin{equation*}
ν_ψ (F^1 (a_k), …, F^1 (a_1)) = F^1 (ν(a_k, …, a_1)).
\end{equation*}
Let now $ b_1 ¤ v_1, …, b_k ¤ v_k $ be elements where $ b_i ∈ B $ and $ v_i ∈ \Hom_{\cat C} (X_i, X_{i+1}) $. We have
\begin{align*}
(μ + ν_ψ) (F^1 (b_k ¤ v_k), …, F^1 (b_1 ¤ v_1)) &= μ(ψ(b_k) ¤ v_k, …, ψ(b_1) ¤ v_1) + ν_ψ (ψ(b_k) ¤ v_k, …, ψ(b_1) ¤ v_1) \\
&= ψ(b_k) … ψ(b_1) ¤ μ(v_k, …, v_1) + F^1 (ν(b_k ¤ v_k), …, ν(b_1 ¤ v_1)), \\
F^1 ((μ + ν) (b_k ¤ v_k, …, b_1 ¤ v_1)) &= F^1 (μ(b_k ¤ v_k, …, b_1 ¤ v_1)) + F^1 (ν (b_k ¤ v_k, …, b_1 ¤ v_1)) \\
&= ψ(b_k … b_1) ¤ μ(v_k, …, v_1) + F^1 (ν (b_k ¤ v_k, …, b_1 ¤ v_1)).
\end{align*}
Since $ ψ $ is an algebra morphism, these two are equal. We conclude that $ F = F^1 $ is a morphism of $ A_∞ $-categories. Finally, we observe that $ F^1 $ is indeed $ A $-homogeneous as claimed. This finishes the proof.
\end{proof}

\subsection{Graded quiver algebras with reduction system}
In this section, we examine the graded deformation theory of graded quiver algebras with reduction system. We start by recalling the deformation theory of quiver algebras with reduction system from \cite{Barmeier-Wang}. Then we define graded quiver algebras with reduction system and explain that their grading carries over to the deformation theory.

We start by recalling the notion of quiver algebras with reduction system. In this subsection, the tensor product $ ¤ $ specifically denotes the tensor product over $ ℂQ_0 $ where $ Q $ is a quiver.

\begin{definition}
Let $ Q $ be a quiver. A \emph{reduction system} for $ Q $ consists of a set of pairs
\begin{equation*}
R = \{(s, φ_s) \running s ∈ S \text{ and } φ_s ∈ ℂQ\}
\end{equation*}
where
\begin{enumerate}
\item $ S $ is a subset of $ Q_{≥2} $ such that $ s $ is not a subpath of $ s' $ when $ s ≠ s' ∈ S $,
\item $ s $ and $ φ_s $ are parallel in $ ℂQ $ for all $ s ∈ S $,
\item $ φ_s $ is irreducible for all $ s ∈ S $.
\end{enumerate}
Here a path $ p $ in $ Q $ is \emph{irreducible} if it does not contain any element $ s ∈ S $ as a subpath. An element $ p ∈ ℂQ $ is \emph{irreducible} if it is a linear combination of irreducible paths. The set of irreducible paths is denoted $ \Irr_S $.
\end{definition}

\begin{remark}
Barmeier and Wang use a different notation for paths in quivers. In their notation, a path in a quiver is denoted $ a_1 … a_k $ where $ h(a_i) = t(a_{i+1}) $. Moreover, the concatenation $ p · q $ of paths is defined whenever $ h(p) = t(q) $. They equip $ ℂQ $ with the structure of left-$ ℂQ_0 $-module (right-$ ℂQ_0 $-module) by concatenating vertices on the left or right, effectively tail and head of paths. This contrasts with our convention that the structure of left-$ ℂQ_0 $-module (right-$ ℂQ_0 $-module) on $ ℂQ $ is given by concatenation at the head and tail of paths. In consequence, any expression that utilizes tensor products over $ ℂQ_0 $ differs from ours, such as $ A ¤_{ℂQ_0} R ¤_{ℂQ_0} A $. In order to accurately reflect all their definitions in our present notation, we would have to make all definitions with reference to the opposite quiver, which would place a high notational burden. Instead, we have opted to introduce analogs of their definitions that work in our present context to keep the presentation streamlined.
\end{remark}

\begin{definition}
Let $ Q $ be a quiver and $ R = \{(s, φ_s)\}_{s ∈ S} $ be a reduction system. Let $ s ∈ S $ and let $ q, r $ be two paths in $ Q $ such that $ q s r ≠ 0 $. Then the \emph{basic reduction} $ τ_{q, s, r}: ℂQ → ℂQ $ is the linear map defined by sending the path $ q s r $ to $ q φ_s r $ and any other path to itself. A \emph{reduction} is any finite composition $ τ = τ_{q_n, s_n, r_n} ∘ … ∘ τ_{q_1, s_1, r_1} $ of basic reductions.

A path $ p $ in $ Q $ is \emph{reduction-finite} if for any infinite sequence of reductions $ (τ_i)_{i ∈ ℕ} $ there exists $ n_0 ∈ ℕ $ such that for all $ n ≥ n_0 $ we have $ (τ_n ∘ … ∘ τ_1)(p) = (τ_{n_0} ∘ … ∘ τ_1)(p) $. A path $ p $ in $ Q $ is \emph{reduction-unique} if $ p $ is reduction finite and for any two reductions $ τ $ and $ τ' $ such that $ τ(p) $ and $ τ'(p) $ are irreducible, we have $ τ(p) = τ'(p) $. The reduction system $ R $ is \emph{reduction-unique} if every path $ p $ in $ Q $ is reduction-unique with respect to $ R $.
\end{definition}

\begin{definition}
Let $ Q $ be a quiver and $ R = \{(s, φ_s)\}_{s ∈ S} $ be a reduction system. The ideal associated to $ R $ is $ I = (s - φ_s)_{s ∈ S} $, and the associated quiver algebra with relations is $ A = ℂQ / I $. The standard projection is denoted $ π: ℂQ → A $.
\end{definition}

\begin{remark}
Let $ R $ be reduction-finite. Then the set of irreducible paths provides a basis for $ A $. In other words, the restriction of the map $ π $ to the linear span of irreducible paths provides a linear isomorphism $ π: ℂ\Irr_S → A $.
\end{remark}

When an algebra $ A $ is given as a quiver algebra with relations given by a reduction system $ A = ℂQ / I $, an explicit bimodule resolution of $ A $ can be provided and the Hochschild cohomology of $ A $ can be expressed explicitly due to the work of Chouhy-Solotar \cite{Chouhy-Solotar} and Barmeier-Wang \cite{Barmeier-Wang}. We start by recalling the notion of $ n $-ambiguities for a reduction system $ R $.

\begin{definition}
Let $ Q $ be a quiver and $ R $ a reduction system. If $ p $ is a path and $ p = qr $ for some paths $ q, r $, then $ q $ is a \emph{proper left subpath} of $ p $ if $ p ≠ q $. For $ n ≥ 0 $, a path $ p $ is an $ n $-ambiguity if there exist $ u_0 ∈ Q_1 $ and irreducible paths $ u_1, …, u_{n+1} $ such that
\begin{enumerate}
\item $ p = u_0 … u_{n+1} $,
\item for all $ i $, the path $ u_i u_{i+1} $ is reducible, and $ u_i d $ is irreducible for any proper left subpath $ d $ of $ u_{i+1} $.
\end{enumerate}
The set of n-ambiguities is denoted $ S_{n+2} $. The set $ S_0 $ is defined as $ Q_0 $ and the set $ S_1 $ is defined as $ Q_1 $.
\end{definition}

\begin{remark}
The 0-ambiguities of $ R $ correspond precisely to the elements $ s ∈ R $. The 1-ambiguities correspond to what is known as overlap ambiguities.
\end{remark}

\begin{definition}
Let $ R $ be a reduction-finite reduction system. The three maps $ \pathsplit_n, \pathsplitR_n, \pathsplitL_n: ℂQ → A ¤ ℂS_n ¤ A $ are defined as
\begin{align*}
\pathsplit_n (p) &= \sum_{\substack{w ∈ S_n \\ qwr = p}} π(q) ¤ w ¤ π(r), \\
\pathsplitR_n (p) &= π(q) ¤ w_R ¤ π(r), \\
\pathsplitL_n (p) &= π(q) ¤ w_L ¤ π(r).
\end{align*}
The sum on the first row runs over all paths $ q, w, r $ such that $ w ∈ S_n $ and $ qwr = p $. In the second row, the paths $ q, w_R, r $ are defined such that $ p = q w_R r $ and $ w_R $ is the rightmost subpath of $ p $ which is an element of $ S_n $. Similarly in the third row, $ q, w_L, r $ are defined such that $ p = q w_L r $ and $ w_L $ is the leftmost subpath of $ p $ which is an element of $ S_n $.
\end{definition}

We are now ready to recall the definition of the bimodule resolution for $ A $. We define
\begin{equation*}
P_{-1} = A, \quad P_n = A ¤ ℂS_n ¤ A \text{ for } n ≥ 0.
\end{equation*}

\begin{definition}
For $ n ≥ 1 $, the map $ δ_n: P_n → P_{n-1} $ is the $ A $-bimodule morphism determined by
\begin{equation*}
δ_n (1 ¤ w ¤ 1) = \begin{cases} π(w) ¤ 1 - 1 ¤ π(w) & \text{if } n = 1 \\ \pathsplitR_{n-1} (w) - \pathsplitL_{n-1} (w) & \text{if } n > 2 \text{ is odd}, \\ \pathsplit_{n-1} (w) & \text{ if } n ≥ 2 \text{ is even}. \end{cases}
\end{equation*}
For $ n ≥ 1 $, the map $ γ_{n-1}: P_{n-1} → P_n $ is the morphism of left $ A $-modules determined by
\begin{equation*}
γ_{n-1} (1 ¤ w ¤ π(u)) = (-1)^n \pathsplit_n (wu) \text{ for any } u ∈ \Irr_S.
\end{equation*}
The maps $ ∂_n: P_n → P_{n-1} $ for $ n ≥ 0 $ and $ ρ_{n-1}: P_{n-1} → P_n $ for $ n ≥ 1 $ are defined by
\begin{align*}
∂_0 (a ¤ b) &= ab, \\
∂_1 (a ¤ x ¤ b) &= aπ(x) ¤ b - a ¤ π(x) b, \\
∂_n (a ¤ w ¤ b) &= a((\id - ρ_{n-2} ∂_{n-1}) δ_n)(1 ¤ w ¤ 1)b, \quad a, b ∈ A, w ∈ S_n, n ≥ 2 \\
ρ_{n-1} &= γ_{n-1} + \sum_{i ≥ 1} γ_{n-1} (δ_n γ_{n-1} - ∂_n γ_{n-1})^i.
\end{align*}
\end{definition}

\begin{todo}
Correct for order issues.
\begin{itemize}
\item Reverse order in: notation $ u_{n+1} … u_0 $ in definition of $ S_n $
\item we keep Barmeier-Wang exactly but only change the notation
\item we will switch the naming of $ \pathsplitL $ to $ \pathsplitR $ and vice versa
\item the projective resolution map will thus be called $ \pathsplitL - \pathsplitR $
\end{itemize}
\end{todo}

\begin{theorem}[\cite{Chouhy-Solotar, Barmeier-Wang}]
Let $ Q $ be a quiver and $ R $ a reduction-finite reduction system. Then the following is a projective $ A $-bimodule resolution:
\begin{equation*}
… \overset{∂_3}{→} P_2 \overset{∂_2}{→} P_1 \overset{∂_1}{→} P_0 \overset{∂_0}{→} P_{-1} = A.
\end{equation*}
\end{theorem}

The datum of a projective $ A $-bimodule resolution $ (P_{•})_{• ≥ 0} $ gives rise to a description of the Hochschild cohomology of $ A $. Indeed, denote the dual chain complex by $ (P^{•}, ∂^{•})_{• ≥ 0} = \Hom_{\Bimod A} (P_{•}, A) $. Denote by $ (P^{•}, ∂^{•})[1] $ the chain complex with entries in degrees $ -1, 0, 1, … $. Recall that the Hochschild cohomology $ \HH^{•} (A) $ has values in degrees $ -1, 0, 1, … $ as well. With these conventions, there is an isomorphism of graded vector spaces $ \HH^{•} (A) ≅ \H(P^{•}, ∂^{•})[1] $. The standard definition of Hochschild cohomology uses the bar resolution $ P_{•} = \Barcon_{•} $ for $ A $, which is inefficient in terms of size. Meanwhile, the resolution of Chouhy-Solotar in the case of quiver algebras with reduction system provides a simple way of evaluating the Hochschild cohomology. We record this as follows.

\begin{lemma}
Let $ Q $ be a quiver and $ R $ be a reduction-unique reduction system. There is an explicit $ L_∞ $-structure on the chain complex $ (P^{•}, ∂^{•})[1] $ together with an explicit $ L_∞ $-quasi-isomorphism $ φ: P^{•} → \HC(A, A) $.
\end{lemma}

We shall denote the $ L_∞ $-algebra given by $ (P^{•}, ∂^{•})[1] $ and its $ L_∞ $-structure simply by $ P^{•}[1] $. The graded part $ \HH^{-1} (A) $ is simply the center of $ A $ and $ \HH^0 (A) $ is the quotient of the derivations on $ A $ by the inner derivations $ [a, -] $ on $ A $. The graded part $ \HH^1 (A) $ describes the infinitesimal deformations of $ A $. We shall now describe how to interpret the map $ φ^1 $ on the graded component of degree one in both $ P^{•}[1] $ and $ \HC(A) $.

\begin{definition}
Let $ Q $ be a quiver and $ R $ a reduction-unique reduction system. Let $ \tilde{φ} ∈ P^2 $ and write $ \tilde{φ} = (\tilde{φ}_s)_{s ∈ S}: S → A $. Denote by $ \tilde{φ}': S → ℂ\Irr_S $ the lift to irreducible. Then the \emph{deformed reduction system} is the set of pairs $ \{(s, φ_s + ε\tilde{φ}')\}_{s ∈ S} $. The notion of \emph{reduction-finite} and \emph{reduction-unique} elements $ x ∈ ℂQ[ε] $ is defined in similarly to the case of classical reduction systems. The deformed reduction system is reduction-unique if all paths are reduction-unique.
\end{definition}

We shall now explain how these deformed reduction systems give rise to deformations of $ A = ℂQ/I $. Recall that $ π: ℂ\Irr_S → A $ provides a linear isomorphism and we can thus interpret the algebra structure of $ A $ as an algebra structure $ (ℂ\Irr_S, *) $ instead, where multiplication is given by concatenation of paths and subsequent reduction with respect to $ R $.

\begin{lemma}[\cite{Barmeier-Wang}]
Let $ Q $ be a quiver and $ R $ a reduction-unique reduction system. An element $ \tilde{φ} $ lies in the kernel of the differential $ ∂^2 $ if and only if the deformed reduction system $ \{(s, φ_s + ε \tilde{φ}')\}_{s ∈ S} $ is reduction-unique. For such an element, define the $ ℂ[ε]/(ε^2) $-bilinear operation $ *_ε $ on $ ℂ\Irr_S[ε] $ by concatenation of paths and subsequent reduction with respect to $ \{(s, φ_s + \tilde{φ}'_s)\}_{s ∈ S} $. Denote by $ *_ε $ the induced operation on $ A[ε]/(ε^2) $. Then $ (ℂ\Irr_S[ε], *_ε) $ and $ (A[ε]/(ε^2), *_ε) $ are associative algebras and we have $ φ^1 (\tilde{φ}) = *_ε - * $.
\end{lemma}

\begin{remark}
Interpreting the element $ φ^1 (\tilde{φ}) $ as a deformation of $ A $, we can simply speaking say that the deformation $ φ^1 (\tilde{φ}) $ is exactly the deformation obtained from using $ φ_s + \tilde{φ}'_s $ as reduction rules instead of $ φ_s $.
\end{remark}

Let us now introduce graded quiver algebras with relations and examine their graded deformation theory. The approach builds on the notions of graded $ L_∞ $-algebras from \autoref{sec:prelimliou-gradeddefo}. We start by explaining how the construction of Barmeier and Wang behaves when $ Q $ is a graded quiver and the relations of the reduction system are homogeneous.

\begin{definition}
Let $ Σ $ be a grading group and let $ Q $ be a quiver with $ Σ $-graded arrows. Let $ R = \{(s, φ_s)\}_{s ∈ S} $ be a reduction system. Then $ R $ is \emph{$ Σ $-homogeneous} if $ φ_s $ is homogeneous of degree $ \deg(s) $ for every $ s ∈ S $.
\end{definition}

When $ R $ is a $ Σ $-graded reduction system, the quotient algebra $ A = ℂQ / I $ is $ A $-graded as well. In consequence, the Hochschild DGLA $ \HC(A) $ inherits a $ Σ $-grading. As we shall now explain, also the bimodule resolution $ P_{•} $ and the $ L_∞ $-algebra $ P^{•}[1] $ are $ Σ $-graded.

\begin{lemma}
Let $ Q $ be a quiver and $ R $ be reduction-unique $ Σ $-graded reduction system. Then the projective resolution $ (P_{•}, ∂_{•}) $ is $ Σ $-graded and the $ L_∞ $-algebra $ P^{•}[1] $ is $ Σ $-graded as well. Moreover, the $ L_∞ $-quasi-isomorphism $ φ: P^{•}[1] → \HC(A) $ is $ Σ $-homogeneous.
\end{lemma}

\begin{proof}
We comment on all three statements individually. For the first statement, note that $ δ_n $ is $ Σ $-homogeneous, $ γ_{n-1} $ is $ Σ $-homogeneous and thus inductively also $ ∂_n $ and $ ρ_{n-1} $ are $ Σ $-homogeneous. For the second statement, we regard the comparison maps $ F_{•}: (P_{•}, ∂_{•}) → (\Barcon_{•}, d_{•}) $ and $ G_{•}: (\Barcon_{•}, d_{•}) → (P_{•}, ∂_{•}) $ from \cite[section 5.1]{Barmeier-Wang}. These comparison maps are defined in terms of the maps $ \varOmega_n $ which are $ Σ $-homogeneous, and thus $ F $ and $ G $ are $ Σ $-homogeneous as well. The maps of chain complexes $ F^{•}: \HC(A) → P^{•}[1] $ and $ G^{•}: P^{•}[1] → \HC(A) $ in \cite[section 7.2]{Barmeier-Wang} and therefore $ Σ $-homogeneous as well. In \cite[Theorem 7.10]{Barmeier-Wang}, the $ L_∞ $-structure on $ P^{•}[1] $ and the $ L_∞ $-quasi-isomorphism $ φ: P^{•}[1] → \HC(A) $ are defined explicitly in terms of $ F^{•} $ and $ G^{•} $. Both are evidently $ Σ $-homogeneous. We finish the proof.
\end{proof}

We are now ready to comment on graded deformations of these algebras. Let $ B $ be a completed $ Σ $-graded deformation base.

\begin{lemma}
Let $ \tilde{φ} ∈ \mathfrak{m} \htensor P^2 $ be a $ Σ $-homogeneous Maurer-Cartan element. Then $ φ^{\MC} (\tilde{φ}) ∈ \mathfrak{m} \htensor \HC(A) $ is a $ Σ $-homogeneous Maurer-Cartan element. The deformed product of the deformation $ φ^{\MC} (\tilde{φ}) $ is given by concatenation and subsequent reduction according to the deformed reduction system $ \{(s, φ_s + \tilde{φ}'_s)\} $ on $ B \htensor ℂ\Irr_S $.
\end{lemma}

\begin{proof}
The ungraded statement is due to Barmeier and Wang. The graded version is an immediate consequence of \autoref{th:prelim-defo-gradedMCpush}.
\end{proof}

\subsection{Liouville manifolds, domains and sectors}
\label{sec:prelimliou-sectors}
In this section, we recall the notions of Liouville manifolds, domains and sectors. In particular, we recall several standard Liouville structures and explain how the notion of potentials and stops can be used to turn elementary Liouville manifolds into Liouville sectors. We explain how to form products of Liouville sectors with special attention to the standard examples and their stop structure. For further reading on Liouville sectors, we refer to \cite{Bocklandt-book} and \cite{GPS-I}.

\begin{definition}
A \emph{Liouville domain} $ (X, ω, λ) $ is a compact exact symplectic manifold with boundary such that $ ω = dλ $ and the Liouville vector field $ X_λ $ given by $ ω(X_λ, -) = λ $ is outward-pointing at the boundary $ ∂X $. The \emph{symplectic collar extension} of $ (X, ω, λ) $ consists of attaching $ [1, ∞) × ∂X $ to $ X $ together with the Liouville structure $ λ = t λ\restr_{∂X} $. A \emph{Liouville manifold} is an exact symplectic manifold with boundary that is isomorphic to a symplectic collar extension of a Liouville domain.
\end{definition}

\begin{todo}
Insert definition of Liouville manifold-with-boundary here – it refers to having additional non-completed boundary parts.
\end{todo}

\begin{definition}
A \emph{Liouville sector} is a Liouville manifold-with-boundary $ (X, ω, λ) $ for which there exists a function $ I: ∂X → ℝ $ such that
\begin{itemize}
\item $ I $ is linear at infinity, meaning $ X_λ I = I $ outside a compact set.
\item $ dI \restr_{C} > 0 $ where $ C $ is the characteristic foliation of $ ∂X $ oriented so that $ ω(N, C) > 0 $ for inward-pointing vectors $ N $.
\end{itemize}
\end{definition}

\paragraph*{Standard Liouville structures}
Let us recall several standard Liouville manifolds. These concern the space $ ℂ $ together with the standard symplectic structure but two different Liouville structures, and the cylinder $ ℂ^* $ together with its standard Liouville structure.

The Liouville domain $ R = [-1, 1] + i[-1, 1] ⊂ ℂ ≅ T^* ℝ $ has coordinates $ z = x + iy $. Its Liouville form $ λ = y dx $ arises from the cotangent bundle construction and we have $ ω = dy ∧ dx $. The Liouville vector field $ X_λ = y ∂_y $ points outward at the boundary $ ∂R $. Its symplectic collar extension is the Liouville manifold given by $ \hat R = [-1, 1] + iℝ ⊂ ℂ $ together with the same expressions for $ λ $ and $ ω $.

The Liouville domain $ A_{r_1, r_2} ⊂ ℂ^* ≅ T^* S^1 $ for small $ r_1 ∈ (0, 1) $ and large $ r_2 ∈ (1, ∞) $ has complex coordinates $ z $ and alternatively polar coordinates $ (t, θ) $ via $ z = e^t e^{iθ} $. Its Liouville form $ λ = t dθ = \frac{-i}{2|z|^2} \log{|z|} (\bar{z} dz - z d\bar{z}) $ arises from the cotangent bundle construction and we have $ ω = dt ∧ dθ $. The Liouville vector field $ X_λ = t ∂_t $ is outward-pointing at the boundary $ ∂A_{r_1, r_2} $. Its symplectic collar extension is the Liouville manifold given by $ \hat A_{r_1, r_2} ≅ ℂ^* $ together with the same expressions for $ λ $ and $ ω $.

\paragraph*{Stops and potentials}
There are two standard ways to modify the behavior of the Reeb chords of a Liouville manifold $ \hat X $. They are relevant to the definition of Fukaya categories, as they reduce the number of morphisms between the objects. The first operation is the addition of stops on the boundary, while the second amounts to removing the area $ \{\Re W < -N\} $ for a given potential $ W: \hat X → ℂ $. While both operations have been developed in different contexts, they achieve approximately the same goal.

The addition of stops to a Liouville manifold $ \hat X $ is due to Sylvan \cite{Sylvan} and Ekholm-Lekili \cite{Ekholm-Lekili} and elaborated in the language of Weinstein manifolds by Eliashberg \cite{Eliashberg}. Simply speaking, this procedure consists of demarcating a codimension-1 portion of the boundary $ ∂X $ as a stop. More precisely, a \emph{Weinstein hypersurface} in a contact manifold $ (Y, ξ) $ is a codimension-1 submanifold $ Σ ⊆ Y $ with boundary such that there exists a contact form $ λ $ for $ ξ $ such that $ (Σ, λ\restr_{Σ}) $ is part of a Weinstein structure on $ Σ $. A \emph{Weinstein pair} is a Weinstein domain $ (X, λ, φ) $ together with a Weinstein hypersurface in its boundary $ ∂X $. In precise terms, this Weinstein hypersurface in $ ∂X $ is known as the stop. Eliashberg demonstrates how to attach a symplectic chimney to the stop area and how to adapt the Liouville structure near the chimney to take the attachment into account. The resulting geometrical shape is a Liouville sector. This is depicted in \autoref{fig:prelimliou-sectors-stop}.

\begin{figure}
\centering
\begin{tikzpicture}
\path[draw] (0, 0) circle[x radius=1.5, y radius=0.5];
\path[draw] (0, 3) circle[x radius=1.5, y radius=0.5];
\path[fill] (0, 2.5) circle[radius=0.1];
\path[draw] (-1.5, 0) -- (-1.5, 3);
\path[draw] (1.5, 0) -- (1.5, 3);
\path[] (0, 2.5) circle[radius=0.1] coordinate (A);
\path[draw] ($ (A) + (-0.5, 0.2) $) to[bend right] coordinate[pos=0.8] (X1) coordinate[pos=0.6] (X2) coordinate[pos=0.3] (X3) ($ (A) + (-0.2, 1.5) $);
\path[draw] ($ (A) + (0.5, 0.2) $) to[bend left] coordinate[pos=0.8] (X4) coordinate[pos=0.6] (X5) coordinate[pos=0.3] (X6) ($ (A) + (0.2, 1.5) $);
\path[draw, very thick] (A) -- ++(up:1.5) coordinate[pos=0.8] (X7) coordinate[pos=0.6] (X8) coordinate[pos=0.3] (X9);
\path[draw, ->] (X1) -- (X7);
\path[draw, ->] (X2) -- (X8);
\path[draw, ->] (X3) -- (X9);
\path[draw, ->] (X4) -- (X7);
\path[draw, ->] (X5) -- (X8);
\path[draw, ->] (X6) -- (X9);
\path[draw, very thick] (A) -- (0, 1) coordinate (B);
\path[draw, very thick] (0, 1.5) circle[x radius=1.5, y radius=0.5];
\path[draw, ->] ($ (A) + (left:0.5) $) -- ++(down:0.3);
\path[draw, ->] ($ (A) + (right:0.5) $) -- ++(down:0.3);
\path[draw, ->] ($ (A) + (left:1) $) -- ++(0, -0.3);
\path[draw, ->] ($ (A) + (right:1) $) -- ++(0, -0.3);
\path[draw, ->] ($ (A) + (-1, -1) $) -- ++(0, -0.2);
\path[draw, ->] ($ (A) + (1, -1) $) -- ++(0, -0.2);
\end{tikzpicture}
\caption{This figure depicts the attachment of a chimney around a stop. The result is a Liouville sector.}
\label{fig:prelimliou-sectors-stop}
\end{figure}
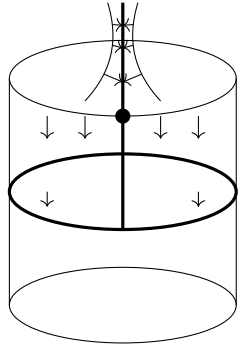

The removal of area according to a potential originates in Seidel's study of Lefschetz fibrations and the definition of Fukaya-Seidel categories \cite{Seidel-Lefschetz-I}. It is fitted into the framework of Liouville sectors by Ganatra-Pardon-Shende \cite[Example 1.4]{GPS-I}. Let $ X $ be a Liouville manifold and $ π: X → ℂ $ an exact symplectic (“Liouville”) Lefschetz fibration. Then removing the locus $ π^{-1} (ℂ_{\Re < -N}) ⊂ X $ for some large $ N $ produces a Liouville sector which we may denote by $ X_W $. The sector $ X_W $ has boundary along the line $ \Re W = -N $.

The wrapping behavior of Liouville sectors that arise from stopped Liouville domains can also be understood in a classical fashion. Let $ X $ be a Liouville domain with stops $ S ⊂ ∂X $. Simply speaking, the prescription for the wrapping behavior of a Lagrangian $ L ⊂ X $ is that the stops $ S $ prevents the wrapping. More precisely, the Lagrangian $ L^{(i)} $ does move along the boundary, but only by a small movement that never crosses the stop set $ S $.


\paragraph*{Kähler structures}
The adaptation of the classical notion of Kähler manifold to the setting of Liouville sectors is rich and complex. For the purposes of this paper, it suffices to know that our Liouville sectors $ X $ are Kähler in the sense that there exists a Kähler potential $ φ: X → ℝ $ such that $ λ = d^c φ = φ_y dx - φ_x dy $. For instance, the rectangle $ R ⊂ T^* ℝ $ is equipped with $ φ = y^2 / 2 $.


\paragraph*{Products and smoothing procedure}
Ganatra, Pardon and Shende define a product operation for Liouville sectors in \cite[section 7.1]{GPS-III}. They also offer an interpretation of this product in case the sectors arise from adding stops to Liouville manifolds. The basic observation is that the naive product as manifolds with boundary does not yield cylindrical ends and therefore the operation has to be performed on the level of the underlying Liouville domains. If $ (X_1, λ_1) $ and $ (X_2, λ_2) $ are Liouville domains, then the boundary of the cornered manifold $ X_1 × X_2 $ is given by $ ∂X_1 × ∂X_2 $. The neighborhood of this cornered manifold inside $ X_1 × X_2 $ is smoothened out by a simple procedure, and the result is a new Liouville sector.

\subsection{Fukaya categories}
In this section, we recall wrapped Fukaya categories from \cite{GPS-I}. The starting point for this construction is a Liouville sector $ (X, λ) $ together with a set $ I $ of (not necessarily transverse) cylindrical exact Lagrangians in $ X $ disjoint from $ ∂X $ and equipped with spin structures. For a complete construction of the Fukaya category, the set $ I $ is meant to contain at least one representative in every isotopy class of such Lagrangians, however it is possible to proceed without this assumption. For the purpose of defining the higher products, one chooses strip-like coordinates and $ ω $-compatible almost complex structures on $ X $.

The wrapped Floer chain complex $ \CF(L_1, L_2) $ for transverse $ L_1, L_2 ∈ I $ is essentially defined as the free vector space $ \CF(L_1, L_2) = \vspan(L_1 ∩ L_2) $. More precisely, the intersection points are given individual $ ℤ/2ℤ $-degrees and come with an associated \emph{orientation line} that is used for the determination of the signs in the higher products. For every $ L ∈ I $ one chooses a cofinal wrapping sequence $ L = L^{(0)}, L^{(1)}, … $ such that any collection $ L_1^{(i_1)}, …, L_N^{(i_N)} $ with $ i_1 < … < i_N $ is mutually transverse. The partial order on the set $ ℤ_{≥0} × I = \{L^{(i)}\}_{L ∈ I, i ∈ ℤ_{≥0}} $ is defined by $ L^{(i)} ≤ K^{(j)} $ if and only if $ i ≤ j $.

\begin{definition}
The \emph{ordered Fukaya category} $ \ordFuk(X) $ is the following strictly unital $ ℤ/2ℤ $-graded $ A_∞ $-category:
\begin{itemize}
\item The set of objects is $ ℤ × I $.
\item The hom spaces are given by:
\begin{equation*}
\Hom(L_1, L_2) = \begin{cases}
\CF(L_1, L_2) & \text{if } L_1 > L_2, \\
ℂ & \text{if } L_1 = L_2, \\
0 & \text{else}.
\end{cases}
\end{equation*}
\item The higher products are given by counting holomorphic disks:
\begin{equation*}
μ^k (p_k, …, p_1) = \sum_{\substack{q ∈ L_1 ∩ L_{k+1} \\ D ∈ \Disk(p_1, …, p_k; q)}} (-1)^{\sgn(D)} q.
\end{equation*}
\end{itemize}
\end{definition}

Let us recall the construction of continuation elements. Simply speaking, these play the same role as identity morphisms and as intersection points they can be found with the same rules as identity points. More precisely, $ L^{(i+1)} $ is obtained from $ L^{(i)} $ by means of a wrapping operation and thus there is a family of intermediate Lagrangians $ L_t $. For a sufficiently fine sequence $ 0 = t_0 < … < t_N = 1 $ there are identifications $ \HF(L_{t_i}, L_{t_i}) ≅ \HF(L_{t_{i+1}}, L_{t_i}) $ and thus elements $ c_i ∈ \HF(L_{t_{i+1}}, L_{t_i}) $ corresponding to the identities $ \id_{L_{t_i}} $. The continuation element associated with $ (i, L) $ is then $ c_{i, L} = c_{N-1} ∘ … ∘ c_0 ∈ \HF(L^{(i+1)}, L^{(i)}) $. These continuation elements are gathered in the set $ C = \{c_{i, L}\}_{i ∈ ℤ_{≥0}, L ∈ I} $. The (wrapped) Fukaya category of $ X $ is defined by localizing $ \ordFuk(X) $ with respect to these continuation elements. This localization is carried out via the $ A_∞ $-localization procedure of Drinfeld, which is reproduced explicitly in \cite[section 3.1.3]{GPS-I}. This localization procedure equips the cones of all continuation elementts $ c_{i, L} $ with a formal morphism $ ε $ in degree $ -1 $ such that $ dε = \id_{\cone(c_{i, L})} $. In effect, in the localized category all these cones are quasi-isomorphic to the zero object. In other words, the morphisms $ c_{i, L} $ become quasi-isomorphisms.

\begin{definition}
The \emph{(wrapped) Fukaya category} $ \Fuk(X) $ is the strictly unital $ A_∞ $-category defined by $ \Fuk(X) = \H (\Tw\ordFuk(X)[C^{-1}]) $.
\end{definition}

Wrapped Floer cohomology $ \HW(L_1, L_2) $ is an invariant that counts intersection points of $ L_1 $ and $ L_2 $ after an infinite amount of wrapping. While it seems tempting to use this as the definition for hom spaces of the wrapped Fukaya category, it has the disadvantage of requiring additional contributions from popsicle maps to the $ A_∞ $-products which cannot be expressed as holomorphic disks, a fact illustrated in the appendix of \cite{Bocklandt}. In contrast, the construction via localization of Ganatra, Pardon and Shende is much simpler on the level of moduli spaces. As it turns out, the hom spaces after localization agree with wrapped Floer cohomology as expected. This fact is proven in \cite[Lemma 3.37]{GPS-I} in slightly different notation.

\begin{lemma}[{\cite[Lemma 3.37]{GPS-I}}]
For $ L_1, L_2 ∈ ℤ × I $ we have a canonical identification $ \Hom(L_1, L_2) ≅ \HW(L_1, L_2) $.
\end{lemma}

\subsection{Relative Fukaya categories}
\label{sec:prelimfuk-grading}
In this section, we explain a procedure of constructing relative Fukaya categories as $ ℤ $-graded deformations.This procedure is based on a standard topological approach which automatically makes the $ A_∞ $-products homogeneous with respect to the grading.

\begin{center}
\begin{tikzpicture}
\path (0, 0) node[align=center] (A) {\emph{Fukaya category} $ \Fuk(X) $ \\ Base category};
\path (6, 0) node[right, align=center] (B) {\emph{Relative Fukaya category} $ \relFuk(X, f) $ \\ $ ℤ $-graded deformation over $ ℂ⟦ħ⟧_1 $};
\path[draw, -{To[scale=1.5]}, decorate, decoration={snake, amplitude=0.5pt}] ($ (A.east) + (right:0.3) $) to node[above, midway] {Function $ f: X → ℂ $} ($ (B.west) + (left:0.3) $);
\end{tikzpicture}
\end{center}
The procedure starts with the datum of a function $ f: X → ℂ $. We fix the terminology for these functions as follows.

\begin{definition}
Let $ X $ be a Liouville sector with complex structure. A \emph{grading function} is a holomorphic function $ f: X → ℂ $. We write $ \mathring{X} = f^{-1} (ℂ^*) $.
\end{definition}

\begin{definition}
Let $ X $ be a Liouville sector and $ f: X → ℂ $ a grading function.
\begin{itemize}
\item Let $ L ⊂ \mathring{X} $ be a contractible Lagrangian and $ L^{(1)} ⊂ \mathring{X} $ a transversal deformation of $ L $. Denote by $ \id_L ∈ L ∩ L^{(1)} $ the identity morphism. Then the \emph{$ f $-degree} of an intersection point $ p ∈ L ∩ L^{(1)} $ is defined as
\begin{equation*}
\deg_f (p) = ∫_{\id_L \overset{L^{(1)}}{→} p \overset{L}{→} \id_L} \frac{1}{2πi} d \log f ∈ ℤ.
\end{equation*}
\item Let $ L^{(0)}, …, L^{(k)} ⊂ \mathring{X} $ be an increasing sequence of deformations of the Lagrangian $ L = L^{(0)} $ and let $ p_1, …, p_k $ be intersection points $ p_i ∈ L^{(i-1)} ∩ L^{(i)} $. Let $ D ∈ \Disk(p_1, …, p_k) $ be a holomorphic disk with boundary $ p_1, …, p_k $. Then the \emph{$ f $-degree} of the disk is
\begin{equation*}
\deg_f (D) = ∫_{∂D} \frac{1}{2πi} d \log f ∈ ℤ_{≥0}.
\end{equation*}
\end{itemize}
\end{definition}

\begin{remark}
The degrees of morphisms and disks are well-defined. Indeed, the integrals are independent of the paths since the Lagrangians $ L^{(i)} $ are contractible. The degree is integral since $ d \log f $ is the pullback of the standard integral 1-form $ dz/(2πiz) $ on $ ℂ^* $. The degree of a holomorphic disk $ u: \mathbb{D} → X $ is non-negative since the composition $ f ∘ u $ is holomorphic and thus $ f ∘ u: ∂\mathbb{D} → ℂ^* $ has non-negative winding number.
\end{remark}

\begin{definition}
$ ℂ⟦ħ⟧_1 $ is the completed $ ℤ $-graded deformation base given by $ ℂ⟦ħ⟧ $ with grading naturally induced from $ \deg(ħ) = 1 $ on all subquotients.
\end{definition}

\begin{definition}
Let $ X $ be a Liouville sector and $ f: X → ℂ $ a grading function. Then the $ ℤ $-graded deformation $ \ordrelFuk(X) $ over $ ℂ⟦ħ⟧_1 $ is defined as follows:
\begin{itemize}
\item The objects are the same as in $ \ordFuk(X) $.
\item The hom spaces are $ \Hom_{\ordrelFuk(X)} (L_1, L_2) = ℂ⟦ħ⟧_1 \htensor \ordFuk(L_1, L_2) $.
\item $ A_∞ $-products are weighted with $ ħ^{\deg_f (D)} $.
\item Curvature of objects is given by $ ħ $-weighted counting of tear drops.
\end{itemize}
\end{definition}

\begin{remark}
\label{rem:prelimfuk-grading-basecat}
When defining a deformation of an $ A_∞ $-category, one usually departs from a given $ A_∞ $-category and deforms the $ A_∞ $-products. In the case of $ \ordrelFuk(X) $ it is not immediately clear what the underlying $ A_∞ $-category is. However, in case $ \mathring{X} $ can be turned into a Liouville sector, it is expected that $ \ordFuk(\mathring{X}) $ is the underlying $ A_∞ $-category. Indeed, let us compare the $ A_∞ $-products of $ \ordrelFuk(X) $ and $ \ordFuk(\mathring{X}) $ on zeroth $ ħ $-order. If $ u: \mathbb{D} → X $ is any holomorphic disk contributing to $ \ordFuk(X) $ with $ \deg(u) = 0 $, then the boundary is a contractible loop in $ ℂ^* $, thus its interior does not touch the zero point and the disk lies entirely in $ \mathring{X} $ and contributes to $ \ordFuk(\mathring{X}) $. Conversely, if a disk lies entirely in $ \mathring{X} $, then it is also counted in $ \ordrelFuk(X) $ and it clearly has zero winding number thus $ \deg(u) = 0 $. In this paper, we shall however stick to the terminology that $ \ordrelFuk(X) $ is a deformation of an anonymous underlying category which we denote by $ \ordrelFuk(X)_{ħ=0} $.
\end{remark}

\begin{definition}
Let $ \mathcal{O} $ be an ordered $ A_∞ $-category and $ C ⊂ \mathcal{O} $ a set of closed morphisms in degree zero. Let $ A $ be a grading group and $ B $ a completed $ A $-graded deformation base. Let $ \mathcal{O}_q $ be a deformation of $ \mathcal{O} $ over $ B $. Then the deformed localization $ \mathcal{O}_q [C^{-1}] $ is the $ A $-graded deformation of $ \mathcal{O} [C^{-1}] $ whose deformed $ A_∞ $-products are given by the Drinfeld localization construction.
\end{definition}

\begin{definition}
\label{def:prelimfuk-grading-relfuk}
Let $ X $ be a Liouville sector and $ f: X → ℂ $ a smooth function. Denote the set of continuation morphisms in $ \ordFuk(X, f) $. Then the relative Fukaya category $ \relFuk(X, f) $ is the $ ℤ $-graded deformation $ \relFuk(X, f) = \ordrelFuk(X, f) [C^{-1}] $ over $ ℂ⟦ħ⟧_1 $.
\end{definition}

\section{Preliminaries on the horizontal Hilbert scheme}
\label{sec:prelimhilbhor}
In this section, we recall horizontal Hilbert schemes, and how they can be turned into Liouville sectors. There are three relevant Liouville sectors, and we recall their construction separately. We also explain the relationships between these Liouville sectors. 
In this section, we recall preliminaries on the Fukaya category of the horizontal Hilbert scheme from \cite{ADLSZ}. We recall a set of standard Lagrangians and their wrappings, as well as the specific construction of the Fukaya category. We recall several basic properties and calculations within this Fukaya category and the equivalence with the $ ħ = 0 $ version of the NilHecke algebras. Finally, we explain how to equip the Fukaya category with an additional grading.

\subsection{The horizontal Hilbert scheme}
\begin{definition}
The \emph{horizontal Hilbert scheme} $ \Hilbhor_k (ℂ^* × ℂ) $ is the affine algebraic variety defined by
\begin{align*}
\Hilbhor_k (ℂ^* × ℂ) &= \Spec\big(ℂ\big[x_1^{±1}, …, x_k^{±1}, y_1, …, y_k, \big\{\frac{x_i - x_j}{y_i - y_j}\}_{i≠j}\big]^{S_k}\big) \\
& ≅ \bigg(\frac{ℂ[x_1^{±1}, …, x_k^{±1}, y_1, …, y_k, \{q_{ij}\}_{i≠j}]}{(q_{ij} (y_i - y_j) - (x_i - x_j))_{i≠j}}\bigg)^{S_k}.
\end{align*}
\end{definition}

A point in $ \Hilbhor_k (ℂ^* × ℂ) $ is given by the datum of $ x ∈ (ℂ^*)^k $, $ y ∈ ℂ^k $ and $ q = (q_{ij})_{i≠j} ∈ ℂ^{k(k-1)/2} $ such that $ q_{ij} (y_i - y_j) = x_i - x_j $, and up to equivalence under the $ S_k $-action given by $ σ(x, y, q) = (\tilde x, \tilde y, \tilde q) $ with $ \tilde{x}_i = x_{σ^{-1} (i)} $, $ \tilde{y}_i = y_{σ^{-1} (i)} $ and $ \tilde{q}_{ij} = q_{σ^{-1} (i) σ^{-1} (j)} $.

The horizontal Hilbert scheme comes with a projection map $ π: \Hilbhor_k (ℂ^* × ℂ) → \Sym^k (ℂ) $ to the symmetric product variety $ \Sym^k (ℂ) $ by means of the coordinates $ y_1, …, y_k $. The symmetric product has an open subset $ \Sym^k_{≠} (ℂ) ⊂ \Sym^k (ℂ) $ given by the complement of the big diagonal $ Δ = \{[y_1, …, y_k] \running ∃i ≠ j: ~ y_i = y_j\} $. The open part $ \Hilbhor_{k, ≠} (ℂ^* × ℂ) $ is defined as the open part given by the preimage of $ \Sym^k_{≠} (ℂ) $ under the projection map.

There are three approaches to defining coordinates on $ \Hilbhor_{k, ≠} (ℂ^* × ℂ) $.
\begin{enumerate}
\item Regard the $ S_k $-covering $ φ_1: (ℂ^*)^k × ℂ^k_{≠} → \Hilbhor_{k, ≠} (ℂ^* × ℂ) $ given by $ (x, y) ↦ [(x, y, q)] $ with $ q_{ij} = (x_i - x_j) / (y_i - y_j) $. This covering provides local coordinates around every point $ [(x, y, q)] $.
\item There is an $ S_k $-covering $ φ_2: (ℂ^*)^k × ℂ^k_{≠} → \Hilbhor_{k, ≠} (ℂ^* × ℂ) $ given by $ (u, y) ↦ [(x, y, q)] $ with $ x_i = u_i (y_i - y_1) … \widehat{(y_i - y_i)} … (y_i - y_k) $ and  $ q_{ij} = (x_i - x_j) / (y_i - y_j) $.
\item There is an identification
\begin{equation*}
φ_3: \Hilbhor_k (ℂ^* × ℂ) \isoto \{Q(y) / P(y) \running \deg(P) = k, ~ \deg Q < k, ~ P \text{ monic}, ~ \gcd(P, Q) = 1\}
\end{equation*}
given by sending $ [(x, y, q)] $ to $ (P, Q) $ given by the polynomial $ Q $ with the property that $ Q(y_i) = x_i $ and the polynomial $ P = \prod_{i = 1}^k (y - y_i) $.
\end{enumerate}

\begin{example}
In the case $ k = 1 $, the transformation $ \Hilbhor_1 (ℂ^* × ℂ) → ℂ^2 $ is given by $ (x, y, q) ↦ (y, x) $.
\end{example}

\begin{example}
In the case $ k = 2 $, the transformation $ \Hilbhor_2 (ℂ^* × ℂ) → ℂ^4 $ is given by
\begin{equation*}
(x_1, x_2, y_1, y_2, q) ↦ (y_1 y_2, - y_1 - y_2, - q y_2 + x_2, q).
\end{equation*}
\end{example}


\subsection{The Liouville sectors $ Y_W $ and $ Y_{O, W} $}
\label{sec:prelim-hilbhor-sectors}
In this section, we define the Liouville sectors $ Y_W $ and $ Y_{O, W} $. Simply speaking, these Liouville sectors consist of the horizontal Hilbert scheme $ \Hilbhor_k (ℂ^* × ℂ) $ and its version with removed diagonal $ \Hilbhor_{k, ≠} (ℂ^* × ℂ) $ under addition of stops provided by a potential $ W $.

\paragraph*{The Liouville manifold $ Y $}
We start by equipping the manifold $ \Hilbhor_k (ℂ^* × ℂ) $ with the structure of a Liouville manifold. As a smooth affine complex variety, there exists a non-canonical choice of embedding $ i: \Hilbhor_k (ℂ^* × ℂ) \embeds ℂ^N $. The space $ ℂ^N $ comes with a standard Liouville structure
\begin{align*}
λ = \sum_{j = 1}^N \frac{-i z_j}{4} dz_j + \frac{i z_j}{4} d\bar{z_j}, \qquad ω = \sum_{j = 1}^N \frac{i}{2} dz_j ∧ d\bar{z_j}.
\end{align*}

\begin{definition}
The Liouville manifold $ Y $ is the complex manifold (without boundary) $ \Hilbhor_k (ℂ^* × ℂ) $ together with the pullback form $ i^* λ $.
\end{definition}

\paragraph*{The Liouville sector $ Y_W $}
We shall now define the Liouville sector $ Y_W $. Simply speaking, it is given by cutting away the stop area given by the potential $ W $. Let us make this precise as follows. The Liouville manifold $ Y $ comes with the potential $ W: Y → ℂ $ given by
\begin{equation*}
W(x, y, q) = \sum_{i = 1}^k u_i + \sum_{i = 1}^k y_i^2 = \sum_{i = 1}^k x_i \prod_{j ≠ i} (y_i - y_j) + \sum_{i = 1}^k y_i^2.
\end{equation*}
The standard procedure recalled in \autoref{sec:prelimliou-sectors} applies and amounts to cutting away the locus of points $ (x, y, q) $ with $ \Re W(x, y, q) < -N $ for a large number $ N $ and performing a chimney extension around the boundary near those pounts. We note that the stop area $ \Re W < -N $ is somewhat unpredictably distributed among the $ u $-coordinates.

\begin{definition}
The Liouville sector $ Y_W $ is obtained from $ Y $ as the chimney extension after cutting away the locus $ \Re W < -N $.
\end{definition}

\paragraph*{The Liouville domain $ Y_O $}
Simply speaking, this Liouville domain is obtained from $ \Hilbhor_k (ℂ^* × ℂ) $ by removing a neighborhood of the big diagonal, cutting off infinite ends, and modeling the remainder as a symmetric product of small standard Liouville domains.

Let us start by removing a neighborhood of the big diagonal. Regard the diagonal $ Δ ⊂ \Sym^k (ℂ) $ and the complement $ \Hilbhor_{k, ≠} (ℂ^* × ℂ) = \Hilbhor_k (ℂ^* × ℂ) \setminus π^{-1} (Δ) $. Regard the $ (u, y) $ coordinate mapping $ φ_2: (ℂ^*)^k × ℂ^k → \Hilbhor_{k, ≠} (ℂ^* × ℂ) $. Fix an $ ε > 0 $ sufficiently small. We choose the neighborhood $ U $ of the big diagonal in $ (ℂ^*)^k × ℂ^k $ given by all points such that $ \sum_{i, j = 1}^k |y_i - y_j|^2  ≤ ε $. The preimage of the boundary $ π^{-1} (∂U) $ consists of the points $ (x, y) ∈ (ℂ^*)^k × ℂ^k $ such that $ f(y) = ε $ where $ f(y) = \sum_{i, j = 1}^k |y_i - y_j|^2 = ε $. Given that $ ∂f/∂y ≠ 0 $ on the boundary, the boundary is smooth.

The space $ ((ℂ^*)^k × ℂ^k) \setminus U $ is a complex manifold with boundary. It sits inside $ (ℂ^*)^k × ℂ^k $ which is a Liouville manifold with the standard Liouville form on $ ℂ^* $ and the standard Liouville form on $ ℂ ≅ T^* ℝ $:
\begin{align*}
λ &= \sum_{j = 1}^k \frac{-i}{2|x_j|^2} \log{|x_j|} (\bar{x_j} dx_j - x_j d\bar{x_j}) + \sum_{j = 1}^k \frac{-\Im(z)}{2} (dy_j + d\bar{y_j}) \\
ω &= \sum_{j = 1}^k \frac{i}{2|x_j|^2} dx_j ∧ d\bar{x_j} + \sum_{j = 1}^k \frac{i}{2} dy_j ∧ d\bar{y_j}.
\end{align*}
We shall now demonstrate how to cut off the infinite ends in the $ u $- and $ y $-coordinates. We start by treating the $ y $-coordinates $ y_1, …, y_k $. The idea is to model these coordinates on the standard plane $ ℂ ≅ ℝ^2 $ and cut off the infinite ends in the imaginary axis. In addition to this standard procedure, the approach of \cite{ADLSZ} also inserts one stop each in the positive and negative end of the real axis, which entails in particular cutting off also the infinite ends in the real axis. We shall make this precise as follows. Regard a rectangle $ R ⊂ ℂ $ with smoothened corners that includes an interval of the real axis, such as a smoothing of the rectangle $ [-1, 1] + [-1, 1]i ⊂ ℂ $. As a subset of $ ℂ ≅ T^* ℝ $, this rectangle comes with the standard Liouville structure inherited from $ T^* ℝ $.

We shall now treat the cutting in the $ u $-coordinates. Choose $ r_1 > 0 $ small and $ r_2 > r_1 $ large and regard the closed annulus $ A_{r_1, r_2} ⊂ ℂ^* $ given by all points $ u $ with $ r_1 ≤ |u| ≤ r_2 $. This annulus inherits the standard Liouville form from the ambient space $ ℂ^* ≅ T^* S^1 $.

Finally, let us assemble the set obtained after cutting away the ends in the $ u $- and $ y $-coordinates:
\begin{equation*}
S = (A_{r_1, r_2}^k × R^k) \setminus U ⊂ (ℂ^*)^k × ℂ^k.
\end{equation*}
Explicitly, the set $ S $ consists of all points $ (u, y) $ where the $ y_i $ coordinates lie inside the rectangle $ R $ and differ sufficiently and where $ x $ lies in the product of the annuli. The set $ S $ is a manifold with boundary with corners. The boundary consists of those points $ (x, y) $ such that at least one of the boundary conditions $ |u_j| ∈ \{r_1, r_2\} $ for $ j = 1, …, k $ and $ f(y) = ε $ is satisfied, and the corners consist of those points where at least two of these conditions are satisfied.

We shall now explain how to turn $ S $ into a Liouville domain. The set $ S $ is a manifold with boundary with corners. While it inherits the Liouville form from $ (ℂ^*) × ℂ^k $ as described above, it is not a Liouville domain due to the cornered nature of its boundary. The idea is to invoke the smoothing procedure recalled in \autoref{sec:prelimliou-sectors} to cure this. We start by commenting on the outward-pointing nature of the Liouville vector field $ X_λ $. At a corner in which $ l $-many boundary conditions are satisfied, the notion of outward-pointing shall mean that $ X_λ $ is outward-pointing with respect to all of the outward directions of the $ l $-many involved boundary hyperplanes.

\begin{lemma}
The Liouville vector field $ X_λ $ is outward-pointing on the boundary $ ∂S $.
\end{lemma}

\begin{proof}
At any point $ (u, y) ∈ (ℂ^*)^k × ℂ^k \setminus π^{-1} (U) $, the Liouville field is given by
\begin{align*}
X_λ = \sum_{j = 1}^k \log{|u_j|} u_j ∂_{u_j} + \log{|u_j|} \bar{u_j} ∂_{\bar{u_j}} + \sum_{j = 1}^k \frac{y_j}{2} ∂_{y_j}.
\end{align*}
For any point $ (u, y) $ that lies in the boundary with respect to any of the $ k+1 $ boundary conditions, we shall check that $ X_λ $ is outward-pointing with respect to that boundary condition.

We start by regarding the case that $ (u, y) $ lies in the boundary of the annulus $ A_{r_1, r_2} $ with respect to coordinate $ u_j $, that is, $ |u_j| ∈ \{r_1, r_2\} $. If $ |u_j| = r_1 $, then the vector $ \log{|u_j|} u_j ∂_{u_j} $ is outward-pointing at the inner boundary of the annulus since $ |u_j| = r_1 $ is a small positive number. If $ |u_j| = r_2 $, then the vector $ \log{|u_j|} u_j ∂_{u_j} $ is outward-pointing at the outer boundary of the annulus since $ |u_j| = r_2 $ is a large positive number.

We now regard the case that $ (u, y) $ lies on the boundary with respect to $ U $, that is, $ f(y) = ε $. Regard the Liouville vector component $ v = \sum_{j = 1}^k y_j ∂_{y_j} $. We have $ f(y + tv) = (1+t)^2 f(y) $ and conclude $ df (v) = 2f(y) > 0 $, thus $ X_λ $ is outward-pointing at $ (u, y) $. This finishes the proof.
\end{proof}

We are now ready to smoothen these boundaries and simultaneously turns $ S $ into a Liouville domain. We start by observing that $ S $ is a product of Liouville domains. Therefore the smoothing procedure recalled in of \cite[section 7.1]{GPS-III} applies and yields a Liouville domain $ \tilde S $. The difference between $ \tilde S $ and $ S $ is entirely cosmetic and lies in a smoothing of the corners and interpolation of the Liouville form near those points.

\begin{definition}
The Liouville domain $ Y_O $ is the quotient of $ \tilde S $ by the $ S_k $-action.
\end{definition}

This quotient is again a manifold since the $ S_k $-action does not have fixed points. The Liouville structure on $ \tilde S $ need not, a priori, be $ S_k $-invariant. Yet, it is rather clear from the construction in \cite[section 7.1]{GPS-III} that this can be achieved, and therefore $ λ_{\tilde S} $ descends to a Liouville form on $ Y_O $. Alternatively, one can obtain the Liouville domain $ Y_O $ by first taking the quotient of $ S $ by the $ S_k $-action and then smoothing out the corners via a procedure analogous analogous to \cite[section 7.1]{GPS-III}. Both options provide the structure of Liouville domain to $ Y_O $.

The Liouville domain $ Y_O $ can be extended to a Liouville manifold $ \widehat{Y_O} $ by means of the symplectic collar extension. In formulas, we have
\begin{align*}
S &= (A_{r_1, r_2}^k × R^k) \setminus U ⊂ (ℂ^*)^k × ℂ^k, \\
\tilde S &= \text{smoothing of } S, \\
Y_O &= \tilde S / S_k, \\
\widehat{Y_O} &= \text{collar extension of } Y_O.
\end{align*}
We shall note that the Liouville domain $ Y_O $ differs from $ \Hilbhor_k (ℂ^* × ℂ) $, since its ends in the $ y $ direction are noncompact.

\paragraph*{The Liouville sector $ Y_{O, W} $}
This Liouville sector is obtained from $ Y_O $ by taking the potential $ W $ into account. The standard procedure recalled in \autoref{sec:prelimliou-sectors} applies and amounts to cutting away the locus of points $ (u, y) $ with $ \Re W(u, y) < -N $ for a large number $ N $ and performing a chimney extension around the boundary near those pounts. We note that the stop area $ \Re W < -N $ is somewhat unpredictably distributed among the $ u $-coordinates.

\begin{definition}
The Liouville sector $ Y_{O, W} $ is obtained from $ Y_O $ by removing the locus $ \Re W < -N $ and performing the chimney construction.
\end{definition}

\paragraph*{Simplified model for $ Y_{O, W} $}
There is a simplified model for this Liouville sector in which the stop loci are equally and predictably distributed among all $ u $-coordinates. This model is advantageous for computing with the Fukaya category $ \Fuk(Y_{O, W}) $.

Simply speaking, the model is given by imposing the potential $ W $ on each of the factors before taking the quotient by $ S_k $, rather than the other way around. Denote by $ A_{r_1, r_2, W} $ the Liouville sector arising from the Liouville domain $ A_{r_1, r_2} $ together with the potential $ W = y^2 $, where $ y $ is the standard coordinate of $ A_{r_1, r_2} ⊂ ℂ^* $. Denote by $ R_W $ the Liouville sector arising from the Liouville domain $ R $ with the potential $ W = u $, where $ u $ is the standard coordinate of $ R ⊂ ℂ $. These base and fiber spaces with their stops are depicted in \autoref{fig:prelimhilbhor-sectors-simplified}. We record the construction as follows:

\begin{definition}
The \emph{simplified model for} $ Y_{O, W} $ is the quotient of the product of Liouville sectors $ A_{r_1, r_2, W}^k × R_W^k $ by the group $ S_k $.
\end{definition}

Points in $ Y_{O, W} $ may be specified through the set of coordinates $ (u, y) $ or $ (x, y) $. In case of $ (u, y) $ coordinates, the requirements are $ u_i ∈ A_{r_1, r_2} $ and $ y_i ∈ R $, and a different special set of conditions is imposed for points near the boundary at $ \Re W = -N $ and the boundary given by products of individual boundaries $ ∂A_{r_1, r_2} $ or $ ∂R $ due to the smoothing procedure. In case of $ (x, y) $ coordinates, the requirements are that $ (x, y) ∈ ((ℂ^*)^k × ℂ^k) \setminus U $ and any lift to $ (u, y) $ satisfies the requirements for $ (u, y) $ coordinates.

\begin{remark}
The Fukaya categories of $ Y_{O, W} $ and its simplified model are in fact equivalent. This fact can be checked by comparing the skeleta of these Liouville sectors. By the chimney construction, these skeleta are explicitly given by the skeleta of the underlying Liouville domains joint with connectors to the stop loci.
\end{remark}

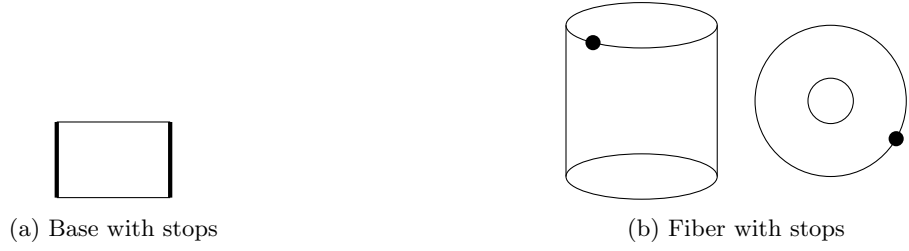
\begin{figure}
\centering
\begin{subfigure}{0.45\linewidth}
\centering
\begin{tikzpicture}
\path[draw] (0, 0) -- (1.5, 0);
\path[draw] (0, 1) -- (1.5, 1);
\path[draw, ultra thick] (0, 0) -- (0, 1);
\path[draw, ultra thick] (1.5, 0) -- (1.5, 1);
\end{tikzpicture}
\caption{Base with stops}
\end{subfigure}
\hspace{0.05\linewidth}
\begin{subfigure}{0.45\linewidth}
\centering
\begin{tikzpicture}
\begin{scope}
\path[draw] (0, 0) ellipse (1 and 0.3);
\path[draw] (0, 2) ellipse (1 and 0.3);
\path[draw] ($ (0, 0) + (180:1 and 0.3) $) to coordinate[pos=0.6] (L) ($ (0, 2) + (180:1 and 0.3) $);
\path[draw] ($ (0, 0) + (0:1 and 0.3) $) to coordinate[pos=0.4] (R) ($ (0, 2) + (0:1 and 0.3) $);
\path[fill] ($ (0, 2) + (230:1 and 0.3) $) circle[radius=0.1];
\end{scope}
\begin{scope}[shift={(2.5, 1)}]
\path[draw] (0, 0) circle[radius=0.3];
\path[draw] (0, 0) circle[radius=1];
\path[fill] (330:1) circle[radius=0.1];
\end{scope}
\end{tikzpicture}
\caption{Fiber with stops}
\end{subfigure}
\caption{The simplified model of $ Y_{O, W} $ is given as a symmetric product of $ k $ base rectangles and $ k $ fiber cylinders. Each of the base rectangles has two stops whereas the fiber cylinders have one stop. The stops of the fiber cylinders are located at zero polar angle and are depicted with a slight offset to make space for later visual convenience.}
\label{fig:prelimhilbhor-sectors-simplified}
\end{figure}

\subsection{Comparison of Liouville structures}
The functionality of the two Liouville manifolds $ \Hilbhor_k (ℂ^* × ℂ) $ and $ Y_O $ is similar in several aspects. We shall make this explicit through a comparison theorem.

\begin{definition}
A Liouville form $ λ $ on a complex manifold $ M $ is a Kähler Liouville form if the Hermitian form $ g(u, v) = ω(u, iv) $ is positive-definite.
\end{definition}

\begin{remark}
We shall remark that both $ λ_{Y_O} $ and $ λ_{\Hilbhor_k (ℂ^* × ℂ)} $ are Kähler Liouville structures. For $ \Hilbhor_k (ℂ^* × ℂ) $ this holds by the fact that $ λ_{\Hilbhor_k (ℂ^* × ℂ)} $ is pulled back from the standard Liouville form on $ ℂ^N $ which is Kähler Liouville. For $ Y_O $ this holds by the fact that $ λ_{Y_O} $ is descended from the standard Liouville form on $ (ℂ^*)^k × ℂ^k $ which is Kähler Liouville, and the smoothed product and collar extension procedures do not substantially affect the property of being Kähler Liouville.
\end{remark}

\begin{lemma}
There is a Liouville domain $ Y ⊂ \Hilbhor_k (ℂ^* × ℂ) $ which completes to $ \Hilbhor_k (ℂ^* × ℂ) $ upon symplectic collar extension. In consequence, $ \Hilbhor_k (ℂ^* × ℂ) $ is also denoted $ Y $.
\end{lemma}

We have a map from the collar extension $ Y_O $ to the collar extension $ Y $.

\begin{lemma}
Denote by $ (Y_O, λ_{Y_O}) $ and $ (Y, λ_Y) $ the Liouville domains. Then the Liouville forms on the Liouville domain $ Y $ and the partial collar extension of $ Y_O $ are isotopic in the sense that there is a continuous path of Liouville Kähler forms connecting them.
\end{lemma}

\begin{lemma}
Let $ Z $ be a smooth affine complex algebraic variety and $ f: Z → ℂ $ a smooth function. Let $ O ⊂ Z $ be an open subset. Let $ λ_Z $ be a Kähler Liouville form on $ Z $ and $ λ_O $ a Kähler Liouville form on $ O $. Then there exists an isotopy of Kähler Liouville forms between $ λ_O $ and $ λ_Z \restr_{O} $.
\end{lemma}

\subsection{Lagrangians and Fukaya categories}
\label{sec:prelimhilbhor-lag}
In this section, we recall the standard Lagrangian $ L ⊂ Y_{O, W} $ and the Fukaya categories of $ Y_{O, W} $ and $ Y_W $. The starting point is the simplified model for $ Y_{O, W} $ from \autoref{sec:prelimliou-sectors}. We recall that points in $ Y_{O, W} $ may be specified through either $ (u, y) $ or $ (x, y) $ coordinates.

We start by defining the standard Lagrangian $ L $. The Lagrangian is the projection of a direct product of Lagrangians, with the base Lagrangians lying sufficiently apart from each other that their $ y $ coordinates differ sufficiently and the projection lies in the interior of $ Y_{O, W} $. Pick $ -1 < s_1 < … < s_k < 1 $ and $ θ_1 < … < θ_k ∈ (-π/2, π/2) $. Both sequences are supposed to lie close to zero but sufficiently far apart from each other. The precise requirements depend on the choices of $ r_1, r_2 $ and the neighborhood $ U $ of the diagonal. Regard the standard projection $ π: A_{r_1, r_2}^k × R^k → Y_{O, W} $, which sends $ (u, y) $ coordinates to their corresponding point in $ Y_{O, W} $. We define $ L $ as the image
\begin{equation*}
L = π([r_1, r_2] e^{i θ_1} × … × [r_1, r_2] e^{i θ_k} × (s_1 + [-1, 1]i) × … × (s_k + [-1, 1]i)).
\end{equation*}
This Lagrangian is depicted in \autoref{fig:prelimhilbhor-lag-std}. The wrapped Lagrangian $ L^{(i)} $ for $ i ∈ ℤ_{≥0} $ is defined similarly to $ L $. It can be described explicitly by wrapping positively along the standard Hamiltonian $ H = r $ in the symplectic collar of $ \hat Y_{O, W} $ and squeezing the result back into $ Y_{O, W} $. The wrapped Lagrangian $ L^{(i)} $ is depicted in \autoref{fig:prelimhilbhor-lag-wrapped}.

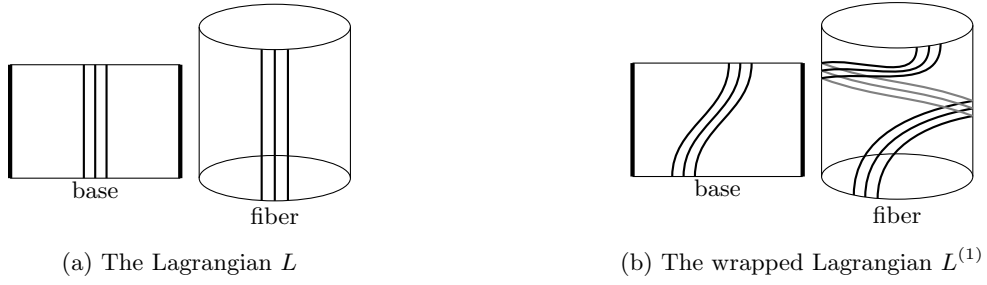
\begin{figure}
\centering
\begin{subfigure}{0.45\linewidth}
\centering
\begin{tikzpicture}
\begin{scope}[shift={(2.25, 0)}, scale=1.5]
\path[draw] (0, 0) -- (1.5, 0);
\path[draw] (0, 1) -- (1.5, 1);
\path[draw, ultra thick] (0, 0) -- (0, 1);
\path[draw, ultra thick] (1.5, 0) -- (1.5, 1);
\path[draw, thick] (0.65, 0) to[out=90, in=270] (0.65, 1);
\path[draw, thick] (0.75, 0) to[out=90, in=270] (0.75, 1);
\path[draw, thick] (0.85, 0) to[out=90, in=270] (0.85, 1);
\path (0.75, -0.1) node {\small base};
\end{scope}
\begin{scope}[shift={(5.75, 0)}]
\path[draw] (0, 0) ellipse (1 and 0.3);
\path[draw] (0, 2) ellipse (1 and 0.3);
\path[draw] ($ (0, 0) + (180:1 and 0.3) $) to ($ (0, 2) + (180:1 and 0.3) $);
\path[draw] ($ (0, 0) + (0:1 and 0.3) $) to ($ (0, 2) + (0:1 and 0.3) $);
\path[draw, thick] ($ (0, 0) + (260:1 and 0.3) $) to coordinate[pos=0.52] (I1) coordinate[pos=0.41] (S1) ($ (0, 2) + (260:1 and 0.3) $);
\path[draw, thick] ($ (0, 0) + (270:1 and 0.3) $) to ($ (0, 2) + (270:1 and 0.3) $);
\path[draw, thick] ($ (0, 0) + (280:1 and 0.3) $) to coordinate[pos=0.63] (S2) coordinate[pos=0.51] (I2) ($ (0, 2) + (280:1 and 0.3) $);
\path (0.0, -0.5) node {\small fiber};
\end{scope}
\end{tikzpicture}
\caption{The Lagrangian $ L $}
\label{fig:prelimhilbhor-lag-std}
\end{subfigure}
\hspace{0.05\linewidth}
\begin{subfigure}{0.45\linewidth}
\centering
\begin{tikzpicture}
\begin{scope}[shift={(2.25, 0)}, scale=1.5]
\path[draw] (0, 0) -- (1.5, 0);
\path[draw] (0, 1) -- (1.5, 1);
\path[draw, ultra thick] (0, 0) -- (0, 1);
\path[draw, ultra thick] (1.5, 0) -- (1.5, 1);
\path[draw, thick] (0.35, 0) to[out=90, in=270] (0.85, 1);
\path[draw, thick] (0.45, 0) to[out=90, in=270] (0.95, 1);
\path[draw, thick] (0.55, 0) to[out=90, in=270] (1.05, 1);
\path (0.75, -0.1) node {\small base};
\end{scope}
\begin{scope}[shift={(5.75, 0)}]
\path[draw] (0, 0) ellipse (1 and 0.3);
\path[draw] (0, 2) ellipse (1 and 0.3);
\path[draw] ($ (0, 0) + (180:1 and 0.3) $) to ($ (0, 2) + (180:1 and 0.3) $);
\path[draw] ($ (0, 0) + (0:1 and 0.3) $) to ($ (0, 2) + (0:1 and 0.3) $);
\path[draw, thick] ($ (0, 0) + (235:1 and 0.3) $) to[out=90, in=190] (1, 1);
\path[draw, thick, gray] (1, 1) to[out=160, in=340] (-1, 1.5);
\path[draw, thick] (-1, 1.5) to[out=10, in=270] ($ (0, 2) + (285:1 and 0.3) $);
\path[draw, thick] ($ (0, 0) + (245:1 and 0.3) $) to[out=90, in=190] (1, 0.9);
\path[draw, thick, gray] (1, 0.9) to[out=160, in=340] (-1, 1.4);
\path[draw, thick] (-1, 1.4) to[out=10, in=270] ($ (0, 2) + (295:1 and 0.3) $);
\path[draw, thick] ($ (0, 0) + (255:1 and 0.3) $) to[out=90, in=190] (1, 0.8);
\path[draw, thick, gray] (1, 0.8) to[out=160, in=340] (-1, 1.3);
\path[draw, thick] (-1, 1.3) to[out=10, in=270] ($ (0, 2) + (305:1 and 0.3) $);
\path (0.0, -0.5) node {\small fiber};
\end{scope}
\end{tikzpicture}
\caption{The wrapped Lagrangian $ L^{(1)} $}
\label{fig:prelimhilbhor-lag-wrapped}
\end{subfigure}
\caption{This figure depicts the standard Lagrangian $ L $ and its wrappings $ L^{(1)} $ in the case $ k = 3 $. The higher wrappings $ L^{(i)} $ are given by successively further turns around the polar axis in the fiber. The start and end point of each Lagrangian $ L^{(i)} $ is offset in horizontal direction from the vertical axis of the base and fiber, with increasing offsets for increasing $ i $ so that the $ L^{(i)} $ form a cofinal positive wrapping sequence.}
\end{figure}

Let us now recall the construction of the Fukaya categories of $ Y_W $ and $ Y_{O, W} $. The focus lies on the standard Lagrangian $ L $, therefore we shall construct these Fukaya categories with the Lagrangian set $ I = \{L\} $ and use notation in which this restriction is manifest. The Fukaya categories $ \Fuk(Y_{O, W}) $ and $ \Fuk(Y_W) $ are the Fukaya categories of these two Liouville sectors with respect to $ I = \{L\} $. Their objects are the wrapped Lagrangians $ L^{(i)} $, all of which are quasi-isomorphic. We denote by $ \Fuk(Y_W)_L $ and $ \Fuk(Y_{O, W})_L $ the subcategories of these categories with the single elements $ L $. We shall later introduce a relative version of the Fukaya category which encompasses both $ \Fuk(Y_{O, W})_L $ and $ \Fuk(Y_W)_L $ as its specializations at $ ħ = 0 $ and $ ħ = 1 $, respectively.

\subsection{NilHecke algebras}
\label{sec:prelim-fuk-NHalgebra}
In this section, we recall NilHecke algebras. In the visual sense, the NilHecke algebra is a strand algebra with dots and bigon vanishing relation over the base ring $ ℂ⟦ħ⟧ $. In the algebraic sense, the algebra is similar to the smash product of a polynomial algebra with the symmetric group but with modified multiplication operation.

The NilHecke algebra $ \NH_k $ is a $ ℂ⟦ħ⟧ $-linear algebra given as vector space by $ ℂ[X_1, …, X_n] ¤ ℂ[S_n]⟦ħ⟧ $. The strand diagram associated with $ X_1^{k_1} … X_n^{k_n} ¤ s ∈ \NH_k $ is the diagram on $ n $ strands which permutes the strands according to $ s $ and above that applies $ k_1, …, k_n $ dots to the strands. This correspondence is depicted in \autoref{fig:diag-nh-stranddiag}. The multiplication operation of the algebra is different from the semidirect product $ ℂ[X_1, …, X_n] \rtimes ℂ[S_n] $ and incorporates the nil aspect of the algebra. For two basis elements $ f ¤ s $ and $ g ¤ t $, the product $ (f ¤ s) (g ¤ t) $ is given by stacking the corresponding strand diagrams on top of each other and applying a set of rules to arrive at a $ ℂ⟦ħ⟧ $-linear combination of standard strand diagrams. The reduction rules are illustrated in \autoref{fig:diag-nh-diagrules}. We shall fix terminology as follows:

\begin{definition}
Let $ k ≥ 1 $. The \emph{NilHecke algebra} $ \NH_k $ is $ ℂ[X_1, …, X_n] ¤ ℂ[S_n]⟦ħ⟧ $ with the following resolving resolving rules:
\begin{itemize}
\item Bigons vanish in the fashion $ s_i^2 = 0 $.
\item Local dot-pass crossings are resolved in the fashion $ X_i s_i - s_i X_{i+1} = ħ \id $ and $ s_i X_i - X_{i+1} s_i = ħ \id $.
\item Braid crossings are resolved in the form $ s_i s_{i+1} s_i = s_{i+1} s_i s_{i+1} $.
\end{itemize}
For $ τ = (τ_1, …, τ_k) ∈ \Comp(n) $ the \emph{NilHecke algebra} $ \NH_τ $ is defined as
\begin{equation*}
\NH_τ ≔ \NH_{τ_1} ¤ … ¤ \NH_{τ_k}.
\end{equation*}
\end{definition}

\begin{example}
For $ σ = (1, …, 1) ∈ \Comp(n) $ we have $ \NH_σ = ℂ[X_1, …, X_n]⟦ħ⟧ $. For $ σ = (2) ∈ \Comp(2) $ we have $ \NH_σ = ℂ[X_1, X_2] ¤ \vspan(e, s)⟦ħ⟧ $ with $ s X_1 = X_2 s + ħ $ and $ s X_2 = X_1 s - ħ $ and $ s^2 = 0 $.
\end{example}

\begin{convention}
\label{conv:nh-algebra-strandconv}
Strand diagrams are read from bottom to top, and composition $ A · B $ denotes the diagram $ A $ sitting on top of $ B $ as illustrated in \autoref{fig:diag-nh-mult}.
\end{convention}

\begin{figure}
\centering
\begin{subfigure}{0.65\linewidth}
\centering
\begin{tikzpicture}
\path[draw] (0.5, 0) -- (7.5, 0);
\path[draw] (0.5, 2) -- (7.5, 2);
\path[draw] (1, 0) node[below] {1} -- (2, 1) -- (2, 2) coordinate[pos=0.3] (A1) coordinate[pos=0.5] (A2) coordinate[pos=0.7] (A3);
\path[draw] (2, 0) node[below] {2} -- (3, 1) -- (3, 2) coordinate[pos=0.5] (B1);
\path[draw] (3, 0) node[below] {3} -- (4, 1) -- (4, 2) coordinate[pos=0.4] (C1) coordinate[pos=0.6] (C2);
\path[draw] (4, 0) node[below] {4} -- (1, 1) -- (1, 2) coordinate[pos=0.3] (D1) coordinate[pos=0.5] (D2) coordinate[pos=0.7] (D3);
\path[draw] (6, 0) node[below] {5} -- (7, 1) -- (7, 2) coordinate[pos=0.4] (E1) coordinate[pos=0.6] (E2);
\path[draw] (7, 0) node[below] {6} -- (6, 1) -- (6, 2) coordinate[pos=0.5] (F1);
\foreach \i in {A1,A2,A3,B1,C1,C2,D1,D2,D3,E1,E2,F1} {\path[fill] (\i) circle[radius=0.05];};
\end{tikzpicture}
\caption{$ X_1^3 X_2 X_3^2 X_4^3 X_5^2 X_1 (1234)(56) ∈ S_{(4, 2)} $ as strand diagram.}
\label{fig:diag-nh-stranddiag}
\end{subfigure}
\hspace{0.05\linewidth}
\begin{subfigure}{0.2\linewidth}
\centering
\begin{tikzpicture}[yscale=1]
\path[draw] (0.5, 0) -- (2.5, 0);
\path[draw] (0.5, 2) -- (2.5, 2);
\path[draw] (1, 0) -- (2, 1) -- (1, 2);
\path[draw] (2, 0) -- (1, 1) -- (2, 2);
\path (3, 1) node {$ = 0 $};
\end{tikzpicture}
\caption{Bigons vanish.}
\label{fig:diag-nh-bigons}
\end{subfigure}
\\[2ex]
\begin{subfigure}{0.64\linewidth}
\centering
\begin{tikzpicture}
\begin{scope}[scale=0.5]
\path[draw] (0.5, 0) -- (2.5, 0);
\path[draw] (0.5, 2) -- (2.5, 2);
\path[draw] (1, 0) -- (2, 2);
\path[draw] (2, 0) -- (1, 2) coordinate[pos=0.7] (A);
\path[fill] (A) circle[radius=0.1];
\path (3, 1) node {$ - $};
\begin{scope}[shift={(3, 0)}]
\path[draw] (0.5, 0) -- (2.5, 0);
\path[draw] (0.5, 2) -- (2.5, 2);
\path[draw] (1, 0) -- (2, 2);
\path[draw] (2, 0) -- (1, 2) coordinate[pos=0.3] (A);
\path[fill] (A) circle[radius=0.1];
\end{scope}
\path (6.5, 1) node {$ = ~~ ħ $};
\begin{scope}[shift={(6.5, 0)}]
\path[draw] (0.5, 0) -- (2.5, 0);
\path[draw] (0.5, 2) -- (2.5, 2);
\path[draw] (1, 0) -- (1, 2);
\path[draw] (2, 0) -- (2, 2);
\end{scope}
\end{scope}
\begin{scope}[scale=0.5, shift={(11, 0)}]
\path[draw] (0.5, 0) -- (2.5, 0);
\path[draw] (0.5, 2) -- (2.5, 2);
\path[draw] (1, 0) -- (2, 2) coordinate[pos=0.3] (A);
\path[draw] (2, 0) -- (1, 2);
\path[fill] (A) circle[radius=0.1];
\path (3, 1) node {$ - $};
\begin{scope}[shift={(3, 0)}]
\path[draw] (0.5, 0) -- (2.5, 0);
\path[draw] (0.5, 2) -- (2.5, 2);
\path[draw] (1, 0) -- (2, 2) coordinate[pos=0.7] (A);
\path[draw] (2, 0) -- (1, 2);
\path[fill] (A) circle[radius=0.1];
\end{scope}
\path (6.5, 1) node {$ = ~~ ħ $};
\begin{scope}[shift={(6.5, 0)}]
\path[draw] (0.5, 0) -- (2.5, 0);
\path[draw] (0.5, 2) -- (2.5, 2);
\path[draw] (1, 0) -- (1, 2);
\path[draw] (2, 0) -- (2, 2);
\end{scope}
\end{scope}
\end{tikzpicture}
\caption{Resolution of dot-pass crossings.}
\label{fig:diag-nh-dotpass}
\end{subfigure}
\\[2ex]

\begin{subfigure}{0.45\linewidth}
\centering
\begin{tikzpicture}[scale=0.5]
\path[draw] (0.5, 0) -- (3.5, 0);
\path[draw] (0.5, 2) -- (3.5, 2);
\path[draw] (1, 0) -- (3, 2);
\path[draw] (3, 0) -- (1, 2);
\path[draw, bend right] (2, 0) to (2, 2);
\begin{scope}[shift={(4, 0)}]
\path[draw] (0.5, 0) -- (3.5, 0);
\path[draw] (0.5, 2) -- (3.5, 2);
\path[draw] (1, 0) -- (3, 2);
\path[draw] (3, 0) -- (1, 2);
\path[draw, bend left] (2, 0) to (2, 2);
\end{scope}
\path (4, 1) node {$ = $};
\end{tikzpicture}
\caption{Resolution of braids.}
\label{fig:diag-nh-nobraid}
\end{subfigure}
\hspace{0.05\linewidth}
\begin{subfigure}{0.2\linewidth}
\centering
\begin{tikzpicture}
\path node {$ A · B = \begin{array}{l} A \\ B \end{array} $};
\end{tikzpicture}
\caption{Multiplication.}
\label{fig:diag-nh-mult}
\end{subfigure}
\caption{This figure illustrates how to interpret elements of the NilHecke algebra as strand diagrams. It also illustrates the relations of the NilHecke algebra in terms of strand diagrams. Multiplication of strand diagrams is given by stacking them upon each other. It is worth remembering that strand diagrams are read from bottom to top. In particular, a crossings in a strand diagrams correspond to the permutation given by reading the strands from bottom to top. For instance, in \autoref{fig:diag-nh-stranddiag} the depicted crossing corresponds to the permutation $ (1234)(56) $.}
\label{fig:diag-nh-diagrules}
\end{figure}
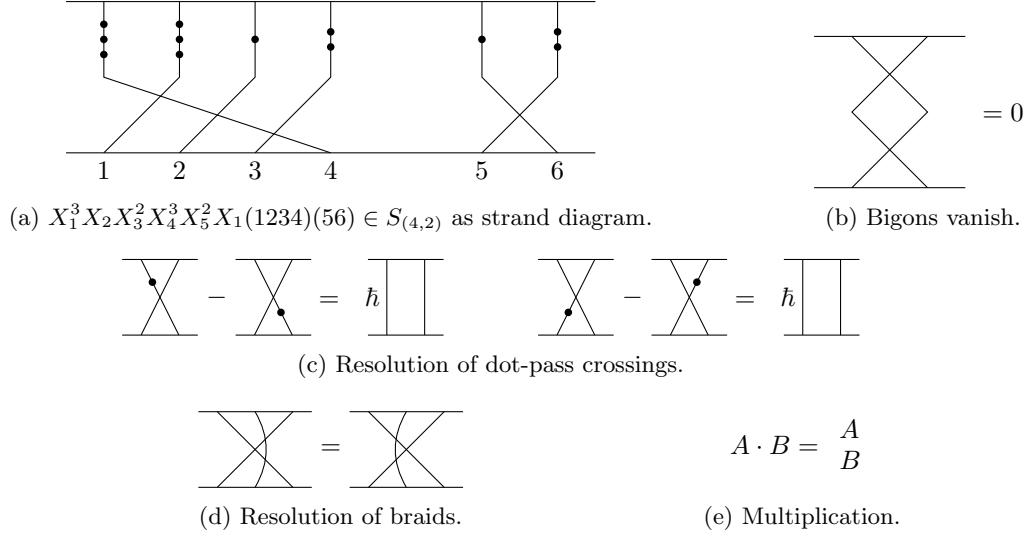

\subsection{Endomorphism algebra of standard Lagrangian}
\label{sec:prelimhilbhor-calc}
In this section, we recall the description of $ \Fuk(Y_{O, W})_L $ in terms of the undeformed NilHecke algebras $ \NH_{k, ħ=0} $ from \cite[section 7]{ADLSZ}. In \autoref{sec:unique}, we use this description to establish the desired equivalence between $ \Fuk(Y_W) $ and the NilHecke algebras $ \NH_k $.

We start by regarding the category $ \Fuk(Y_{O, W})_L $. Its single object $ L $ has a large endomorphism space obtained from the localization procedure. After passing to the minimal model, its endomorphism space can be described explicitly as wrapped Floer homology. Simply speaking, its endomorphism space is spanned by intersections between $ L $ and the wrapped version $ L^{(1)} $. Thanks to the big diagonal being removed, it is doable to compute the $ A_∞ $-products among these intersection points, and this calculation was carried out in \cite[section 7]{ADLSZ}. We recall this description as follows:

\begin{theorem}[{\cite[Theorem 7.1]{ADLSZ}}]
\label{th:prelimhilbhor-calc-iso}
There exists an isomorphism of $ A_∞ $-algebras $ \NH_{k, ħ=0} → \H\End_{\Fuk(Y_{O, W})} (L) $.
\end{theorem}

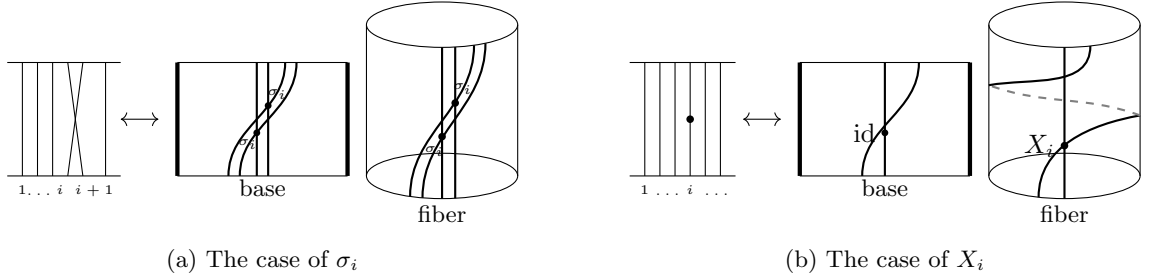
\begin{figure}
\centering
\begin{subfigure}{0.45\linewidth}
\centering
\begin{tikzpicture}
\begin{scope}
\path[draw] (0, 0) -- (1.5, 0);
\path[draw] (0, 1.5) -- (1.5, 1.5);
\path[draw] (0.2, 0) -- (0.2, 1.5);
\path[draw] (0.4, 0) -- (0.4, 1.5);
\path[draw] (0.6, 0) -- (0.6, 1.5);
\path[draw] (0.8, 0) -- (1.0, 1.5);
\path[draw] (1.0, 0) -- (0.8, 1.5);
\path[draw] (1.3, 0) -- (1.3, 1.5);
\path (0.2, 0) node[below, shift={(0, 0)}] {\tiny $ 1 $};
\path (0.4, 0) node[below, shift={(0.02, -0.1)}] {\tiny $ … $};
\path (0.8, 0) node[below, shift={(-0.1, 0)}] {\tiny $ i $};
\path (1.0, 0) node[below, shift={(0.15, 0)}] {\tiny $ i+1 $};
\end{scope}
\path[draw, <->] (1.5, 0.75) to (2, 0.75);
\begin{scope}[shift={(2.25, 0)}, scale=1.5]
\path[draw] (0, 0) -- (1.5, 0);
\path[draw] (0, 1) -- (1.5, 1);
\path[draw, ultra thick] (0, 0) -- (0, 1);
\path[draw, ultra thick] (1.5, 0) -- (1.5, 1);
\path[draw, thick] (0.7, 0) to[out=90, in=270] (0.7, 1);
\path[draw, thick] (0.8, 0) to[out=90, in=270] (0.8, 1);
\path[draw, thick] (0.45, 0) to[out=90, in=270] (0.95, 1);
\path[draw, thick] (0.55, 0) to[out=90, in=270] (1.05, 1);
\path[fill] (0.7, 0.38) circle[radius=0.03] node[shift={(-0.12, -0.15)}] {\tiny $ σ_i $};
\path[fill] (0.8, 0.62) circle[radius=0.03] node[shift={(0.12, 0.15)}] {\tiny $ σ_i $};
\path (0.75, -0.1) node {\small base};
\end{scope}
\begin{scope}[shift={(5.75, 0)}]
\path[draw] (0, 0) ellipse (1 and 0.3);
\path[draw] (0, 2) ellipse (1 and 0.3);
\path[draw] ($ (0, 0) + (180:1 and 0.3) $) to ($ (0, 2) + (180:1 and 0.3) $);
\path[draw] ($ (0, 0) + (0:1 and 0.3) $) to ($ (0, 2) + (0:1 and 0.3) $);
\path[draw, thick] ($ (0, 0) + (270:1 and 0.3) $) to coordinate[pos=0.52] (I1) coordinate[pos=0.41] (S1) ($ (0, 2) + (270:1 and 0.3) $);
\path[draw, thick] ($ (0, 0) + (280:1 and 0.3) $) to coordinate[pos=0.63] (S2) coordinate[pos=0.51] (I2) ($ (0, 2) + (280:1 and 0.3) $);
\path[draw, thick] ($ (0, 0) + (245:1 and 0.3) $) to[out=90, in=270] ($ (0, 2) + (295:1 and 0.3) $);
\path[draw, thick] ($ (0, 0) + (255:1 and 0.3) $) to[out=90, in=270] ($ (0, 2) + (305:1 and 0.3) $);
\path[fill] (S1) circle[radius=0.05] node[below, shift={(-0.1, 0)}] {\tiny $ σ_i $};
\path[fill] (S2) circle[radius=0.05] node[above, shift={(0.1, 0)}] {\tiny $ σ_i $};
\path (0.0, -0.5) node {\small fiber};
\end{scope}
\end{tikzpicture}
\caption{The case of $ σ_i $}
\end{subfigure}
\hspace{0.05\linewidth}
\begin{subfigure}{0.45\linewidth}
\centering
\begin{tikzpicture}
\begin{scope}
\path[draw] (0, 0) -- (1.4, 0);
\path[draw] (0, 1.5) -- (1.4, 1.5);
\path[draw] (0.2, 0) -- (0.2, 1.5);
\path[draw] (0.4, 0) -- (0.4, 1.5);
\path[draw] (0.6, 0) -- (0.6, 1.5);
\path[draw] (0.8, 0) -- (0.8, 1.5) coordinate[pos=0.5] (A);
\path[draw] (1.0, 0) -- (1.0, 1.5);
\path[draw] (1.2, 0) -- (1.2, 1.5);
\path[fill] (A) circle[radius=0.05];
\path (0.2, 0) node[below, shift={(0, 0)}] {\tiny $ 1 $};
\path (0.5, 0) node[below, shift={(0.0, -0.1)}] {\tiny $ … $};
\path (0.8, 0) node[below, shift={(0, 0)}] {\tiny $ i $};
\path (1.0, 0) node[below, shift={(0.15, -0.1)}] {\tiny $ … $};
\end{scope}
\path[draw, <->] (1.5, 0.75) to (2, 0.75);
\begin{scope}[shift={(2.25, 0)}, scale=1.5]
\path[draw] (0, 0) -- (1.5, 0);
\path[draw] (0, 1) -- (1.5, 1);
\path[draw, ultra thick] (0, 0) -- (0, 1);
\path[draw, ultra thick] (1.5, 0) -- (1.5, 1);
\path[draw, thick] (0.75, 0) to[out=90, in=270] (0.75, 1);
\path[draw, thick] (0.55, 0) to[out=90, in=270] (1.05, 1);
\path[fill] (0.75, 0.38) circle[radius=0.03] node[left] {$ \id $};
\path (0.75, -0.1) node {\small base};
\end{scope}
\begin{scope}[shift={(5.75, 0)}]
\path[draw] (0, 0) ellipse (1 and 0.3);
\path[draw] (0, 2) ellipse (1 and 0.3);
\path[draw] ($ (0, 0) + (180:1 and 0.3) $) to coordinate[pos=0.6] (L) ($ (0, 2) + (180:1 and 0.3) $);
\path[draw] ($ (0, 0) + (0:1 and 0.3) $) to coordinate[pos=0.4] (R) ($ (0, 2) + (0:1 and 0.3) $);
\path[draw, thick] ($ (0, 0) + (270:1 and 0.3) $) to coordinate[pos=0.35] (S) coordinate[pos=0.81] (I) ($ (0, 2) + (270:1 and 0.3) $);
\path[draw, thick] ($ (0, 0) + (250:1 and 0.3) $) to[out=90, in=190] (R);
\path[draw, thick, gray, dashed] (R) to[out=160, in=340] (L);
\path[draw, thick] (L) to[out=10, in=270] ($ (0, 2) + (290:1 and 0.3) $);
\path[fill] (S) circle[radius=0.05] node[left] {$ X_i $};
\path (0.0, -0.5) node {\small fiber};
\end{scope}
\end{tikzpicture}
\caption{The case of $ X_i $}
\end{subfigure}
\caption{This figure depicts the correspondence between the generators $ X_i $ and $ σ_i $ of $ \NH_{k, ħ=0} $ with specific self-intersection points in $ L ∩ L^{(1)} $. In the case of $ σ_i $, only the $ i $-th and $ i+1 $-th Lagrangian are depicted. In the case of $ X_i $, only the $ i $-th Lagrangian is depicted. In both cases, the intersection is trivial in all other coordinates.}
\label{fig:prelimfuk-nh-correspondence}
\end{figure}

The correspondence of morphisms established by this theorem is depicted in \autoref{fig:prelimfuk-nh-correspondence}. Following this correspondence, we denote individual morphisms in $ \H\End(L) $ by the same labels as in $ \NH_{k, ħ=0} $. In particular we have the morphisms $ X_i ∈ \H\End(L) $ for $ i = 1, …, k $ and $ σ_i ∈ \H\End(L) $ for $ i = 1, …, k-1 $.

\section{Uniqueness of deformations}
\label{sec:unique}
In this section, we construct the relative Fukaya category $ \relFuk(Y_W)_L $ and show that it is isomorphic to the NilHecke algebra $ \NH_k $. In \autoref{sec:unique-grading}, we construct $ \relFuk(Y_W)_L $ as a $ ℤ^2 $-graded deformation over the deformation base $ ℂ⟦ħ⟧_{(0, 1)} $. In \autoref{sec:unique-HH}, we compute the $ (0, -1) $-graded component of Hochschild cohomology of $ \NH_{k, ħ=0} $. In \autoref{sec:unique-classification}, we classify the $ ℤ^2 $-graded deformations of $ \NH_{k, ħ=0} $ over $ ℂ⟦ħ⟧_{(0, 1)} $. In \autoref{sec:unique-main}, we state the main theorem.

\subsection{Relative $ ℤ^2 $-graded Fukaya category}
\label{sec:unique-grading}
In this section, we present the algebra $ \NH_k $ as a $ ℤ^2 $-graded deformation of $ \NH_{k, ħ=0} $, present the relative Fukaya category $ \relFuk(Y_W)_L $ as a $ ℤ^2 $-graded deformation of $ \Fuk(Y_{O, W}) $, and establish equality of the gradings between $ \NH_{k, ħ=0} $ and $ \Fuk(Y_{O, W}) $. In particular, we extend the comparison theorem \autoref{th:prelimhilbhor-calc-iso} to the $ ℤ^2 $-graded setting.

\begin{equation*}
\begin{tikzcd}
\relFuk(Y_W)_L \arrow[d, dash, "ℤ^2\text{-defo}"] & \NH_k \arrow[d, dash, "ℤ^2\text{-defo}"] \\
\Fuk(Y_{O, W})_L \arrow[r, "\sim", "ℤ^2"'] & \NH_{k, ħ=0}.
\end{tikzcd}
\end{equation*}

\begin{definition}
$ ℂ⟦ħ⟧_{(0, 1)} $ is the completed $ ℤ^2 $-graded deformation base given by $ ℂ⟦ħ⟧ $ with grading on all subquotients induced naturally from $ \deg(ħ) = (0, 1) $.
\end{definition}

We start by presenting $ \NH_k $ as a $ ℤ^2 $-graded deformation of $ \NH_{k, ħ=0} $ over $ ℂ⟦ħ⟧_{(0, 1)} $. The $ ℤ^2 $-grading $ \deg = (\deg_q, \deg_{ħ}) $ consists of two individual $ ℤ $-components, which we refer to as q-grading and ħ-grading. These degrees are given on the standard generators as follows:
\begin{equation*}
\deg(X_i) = (-1, 0), \quad \deg(σ_i) = (1, 1).
\end{equation*}

\begin{lemma}
$ \NH_k $ is a $ ℤ^2 $-graded deformation of $ \NH_{k, ħ=0} $ over $ ℂ⟦ħ⟧_{(0, 1)} $.
\end{lemma}

\begin{proof}
It suffices to observe that the reduction relations $ σ_i^2 = 0 $, $ X_i X_j = X_j X_i $, $ σ_i X_i = X_{i+1} σ_i + ħ $, $ σ_i X_{i+1} = X_i σ_i - ħ $ and $ σ_i σ_{i+1} σ_i = σ_{i+1} σ_i σ_{i+1} $ are $ ℤ^2 $-homogeneous.
\end{proof}

Let us now construct the relative Fukaya category $ \relFuk(Y_W)_L $ as a $ ℤ^2 $-graded deformation of $ \Fuk(Y_{O, W})_L $ over $ ℂ⟦ħ⟧_{(0, 1)} $. The construction of this deformation follows the procedure outlined in \autoref{sec:prelimfuk-grading}. We pick the following grading functions:
\begin{align*}
f_q &= \prod_{i = 1}^k x_i, \\
f_ħ &= \prod_{i < j} (y_i - y_j)^2.
\end{align*}
Two applications of \autoref{def:prelimfuk-grading-relfuk} provide for the construction of a $ ℤ^2 $-graded deformation $ \relFuk(Y_W, (f_q, f_ħ))_L $ over $ ℂ⟦q, ħ⟧_{(1, 0), (0, 1)} $ with $ \deg(q) = (1, 0) $ and $ \deg(ħ) = (0, 1) $. Taking the quotient by $ q $ produces the $ ℤ^2 $-graded deformation $ \relFuk(Y_W)_L $ over $ ℂ⟦ħ⟧_{(0, 1)} $.

\begin{lemma}
The category $ \relFuk(Y_W)_L $ is a $ ℤ^2 $-graded deformation of $ \Fuk(Y_{O, W})_L $ over $ ℂ⟦ħ⟧_1 $.
\end{lemma}

\begin{proof}
We start by observing that $ \relFuk(Y_W, (f_q, f_ħ))_L $ is indeed a $ ℤ^2 $-graded deformation over $ ℂ⟦q, ħ⟧_{(1, 0), (0, 1)} $. After taking the quotient by $ q $, the $ ℤ^2 $-grading survives while the deformation base reduces from $ ℂ⟦q, ħ⟧_{(1, 0), (0, 1)} $ to $ ℂ⟦ħ⟧_{(0, 1)} $. Therefore $ \relFuk(Y_W)_L $ is a $ ℤ^2 $-graded deformation over $ ℂ⟦ħ⟧_{(0, 1)} $. It remains to comment on the base category of $ \relFuk(Y_W)_L $. As discussed in \autoref{rem:prelimfuk-grading-basecat}, the base category is not in general obvious, however in the present case of $ Y_W $, the Liouville sector $ Y_{O, W} $ is the desired incarnation of $ \mathring{X} $. We finish the proof.
\end{proof}

\begin{figure}
\centering
\begin{subfigure}{0.45\linewidth}
\centering
\begin{tikzpicture}
\path[draw] (0, 0) ellipse (1 and 0.3);
\path[draw] (0, 2) ellipse (1 and 0.3);
\path[draw] ($ (0, 0) + (180:1 and 0.3) $) to ($ (0, 2) + (180:1 and 0.3) $);
\path[draw] ($ (0, 0) + (0:1 and 0.3) $) to ($ (0, 2) + (0:1 and 0.3) $);
\path[draw, thick] ($ (0, 0) + (270:1 and 0.3) $) to coordinate[pos=0.52] (I1) coordinate[pos=0.41] (S1) ($ (0, 2) + (270:1 and 0.3) $);
\path[draw, thick] ($ (0, 0) + (280:1 and 0.3) $) to coordinate[pos=0.63] (S2) coordinate[pos=0.51] (I2) ($ (0, 2) + (280:1 and 0.3) $);
\path[draw, thick] ($ (0, 0) + (245:1 and 0.3) $) to[out=90, in=270] ($ (0, 2) + (295:1 and 0.3) $);
\path[draw, thick] ($ (0, 0) + (255:1 and 0.3) $) to[out=90, in=270] ($ (0, 2) + (305:1 and 0.3) $);
\path[fill] (I1) circle[radius=0.05] node[above left] {\tiny $ \id $};
\path[fill] (I2) circle[radius=0.05] node[below right] {\tiny $ \id $};
\path[fill] (S1) circle[radius=0.05] node[below, shift={(-0.1, 0)}] {\tiny $ σ_i $};
\path[fill] (S2) circle[radius=0.05] node[above, shift={(0.1, 0)}] {\tiny $ σ_i $};
\begin{scope}[shift={(2.5, 1)}]
\path[draw] (0, 0) circle[radius=0.3];
\path[draw] (0, 0) circle[radius=1];
\path[draw, thick] (-10:0.3) to coordinate[pos=0.36] (S1) coordinate[pos=0.51] (I1) (-3:1);
\path[draw, thick] (10:0.3) to coordinate[pos=0.51] (I2) coordinate[pos=0.67] (S2) (3:1);
\path[draw, thick] (-30:0.3) to[out=-20, in=189] (15:1);
\path[draw, thick] (-50:0.3) to[out=-50, in=186] (9:1);
\path[fill] (I1) circle[radius=0.03] node[shift={(0.1, -0.1)}] {\tiny $ \id $};
\path[fill] (I2) circle[radius=0.03] node[shift={(-0.1, 0.1)}] {\tiny $ \id $};
\path[fill] (S1) circle[radius=0.03];
\path[fill] (S2) circle[radius=0.03];
\end{scope}
\begin{scope}[shift={(5, 1.2)}]
\path (0, 0) coordinate (S1);
\path (0, 0.5) coordinate (I1);
\path (-0.5, 0.25) coordinate (I2);
\path (-0.5, 0.75) coordinate (S2);
\path[fill] (I1) circle[radius=0.05] node[right] {$ \id $};
\path[fill] (I2) circle[radius=0.05] node[left] {$ \id $};
\path[fill] (S1) circle[radius=0.05] node[right] {$ σ_i $};
\path[fill] (S2) circle[radius=0.05] node[left] {$ σ_i $};
\path[draw, rounded corners, ->] ($ (I2) + (330:0.1) $) to (S1) to ($ (I1) + (270:0.1) $);
\path[draw, rounded corners, ->] ($ (I1) + (150:0.1) $) to (S2) to ($ (I2) + (90:0.1) $);
\path[draw, ->] (0, -0.5) arc(60:420:0.4);
\path[fill] ($ (0, -0.5) + (240:0.4) $) circle[radius=0.03] node[below] {\tiny 0};
\path (-0.2, -1.5) node {\small $ (y_i - y_{i+1})^2 $};
\end{scope}
\end{tikzpicture}
\caption{The case of $ σ_i $}
\end{subfigure}
\hspace{0.05\linewidth}
\begin{subfigure}{0.45\linewidth}
\centering
\begin{tikzpicture}
\path[draw] (0, 0) ellipse (1 and 0.3);
\path[draw] (0, 2) ellipse (1 and 0.3);
\path[draw] ($ (0, 0) + (180:1 and 0.3) $) to coordinate[pos=0.6] (L) ($ (0, 2) + (180:1 and 0.3) $);
\path[draw] ($ (0, 0) + (0:1 and 0.3) $) to coordinate[pos=0.4] (R) ($ (0, 2) + (0:1 and 0.3) $);
\path[draw, thick] ($ (0, 0) + (270:1 and 0.3) $) to coordinate[pos=0.35] (S) coordinate[pos=0.81] (I) ($ (0, 2) + (270:1 and 0.3) $);
\path[draw, thick] ($ (0, 0) + (250:1 and 0.3) $) to[out=90, in=190] (R);
\path[draw, thick, gray, dashed] (R) to[out=160, in=340] (L);
\path[draw, thick] (L) to[out=10, in=270] ($ (0, 2) + (290:1 and 0.3) $);
\path[fill] (I) circle[radius=0.05] node[above left] {$ \id $};
\path[fill] (S) circle[radius=0.05] node[left] {$ X_i $};
\begin{scope}[shift={(2.5, 1)}]
\path[draw] (0, 0) circle[radius=0.3];
\path[draw] (0, 0) circle[radius=1];
\path[draw, thick] (0:0.3) to coordinate[pos=0.16] (S) coordinate[pos=0.62] (I) (0:1);
\path[draw, thick] (-20:0.3) to[out=330, in=290] (45:0.5) to[out=120, in=0] (90:0.5) to[out=180, in=90] (180:0.55) to[out=-90, in=180] (270:0.7) to[out=0, in=200] (10:1);
\path[fill] (I) circle[radius=0.05] node[above] {\tiny $ \id $};
\path[fill] (S) circle[radius=0.05] node[below] {\tiny $ X_i $};
\end{scope}
\begin{scope}[shift={(4.5, 1)}]
\path[draw] (0:1) to[out=270, in=0] (270:0.8) to[out=180, in=270] (180:0.7) to[out=90, in=180] (90:0.55) to[out=0, in=90] (0:0.45);
\path[draw, ->] (0:0.45) to (0:0.9);
\path[fill] (0, 0) circle[radius=0.05] node[below] {0};
\end{scope}
\end{tikzpicture}
\caption{The case of $ X_i $}
\end{subfigure}
\caption{The degree of the standard morphisms $ p = X_i, σ_i ∈ L ∩ L^{(1)} $ is the turning number of the functions $ f_q $ and $ f_ħ $ when moving on the composite path from the identity along $ L^{(1)} $ to $ p $ and back along $ L $. Since the Lagrangians are symmetric products themselves, the composite path is drawn as two individual paths. In the case of $ σ_i $, the difference $ y_i - y_{i+1} $ makes a 180 degrees positive turn around zero, therefore $ \deg_q (σ_i) = \deg_ħ (σ_i) = 1 $. In the case of $ X_i $, the coordinate $ u_i $ makes a 360 degrees negative turn, therefore $ \deg_ħ (X_i) = -1 $. For convenience, each self-intersection is depicted in two equivalent graphics, representing the cylindrical and the annulus model.}
\label{fig:unique-grading-fukdeg}
\end{figure}
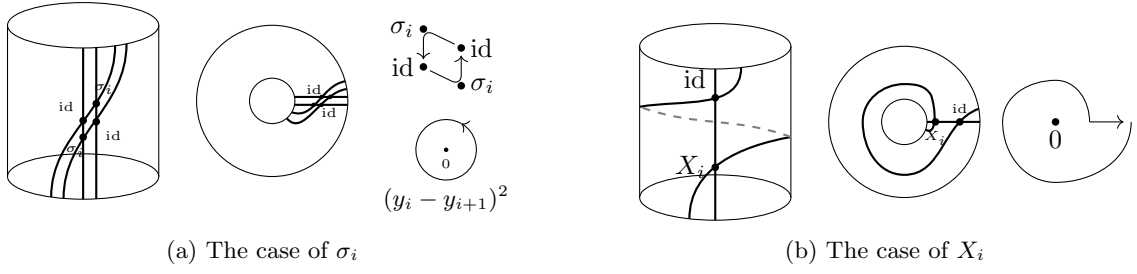

\begin{lemma}
The identification $ \Fuk(Y_W)_L ≅ \NH_{k, ħ=0} $ is $ ℤ^2 $-graded.
\end{lemma}

\begin{proof}
We start by observing that the standard Lagrangian and its wrappings by construction do not intersect the preimage $ π^{-1} (Δ) $ of the big diagonal. In particular, they do not touch the zero locus of $ f_q $ and $ f_ħ $ and we can use the $ u_i $ coordinates to express $ f_q $:
\begin{equation*}
f_q = \prod_{i = 1}^k u_i \prod_{i < j} (y_i - y_j)^2.
\end{equation*}
The degrees can now be read off graphically, depicted in \autoref{fig:unique-grading-fukdeg}.
\end{proof}

Along the identification $ \H\End_{\Fuk(Y_{O, W})} (L) ≅ \NH_{k, ħ=0} $, both $ \relFuk(Y_W)_L $ and $ \NH_k $ are $ ℤ^2 $-graded deformations of $ \NH_{k, ħ=0} $. We shall finally remark that the $ q $-degree $ \deg_q $ is already present in \cite[section 6]{ADLSZ}, while we have introduced the $ ħ $-degree in order to guarantee the uniqueness of deformations.

\subsection{Hochschild cohomology of $ \NH_{k, ħ=0} $ in degree $ (0, -1) $}
\label{sec:unique-HH}
In this section, we determine the degree zero part $ \HH(\NH_{k, ħ=0})_0 $ of the Hochschild cohomology of the algebra $ \NH_{k, ħ=0} $. The idea is to present this algebra as a quiver algebra with relations given by a reduction system. As it turn out, this degree zero part is one-dimensional. We provide explicit representation of the deformations.

We start by presenting $ \NH_{k, ħ=0} $ as a quiver algebra with reduction system. We have
\begin{equation*}
\NH_{k, ħ=0} = ℂ⟨X_1, …, X_k, σ_1, …, σ_{k-1}⟩ / I.
\end{equation*}
The ideal $ I $ is generated by the following elements of the form $ s - φ_s $:
\begin{alignat*}{2}
σ_i X_i &- X_{i+1} σ_i, \quad && i ∈ \{1, …, k-1\}, \\
σ_i X_{i+1} &- X_i σ_i, \quad && i ∈ \{1, …, k-1\}, \\
σ_i X_j &- X_j σ_i \quad && i ∈ \{1, …, k-1\}, j \notin \{i, i+1\}, \\
σ_i^2 &- 0, \quad && i ∈ \{1, …, k-1\}, \\
σ_i σ_{i+1} σ_i &- σ_{i+1} σ_i σ_{i+1}, \quad && i ∈ \{1, …, k-1\}, \\
σ_i σ_j &- σ_j σ_i, \quad && i ∈ \{1, …, k-1\}, j ≤ i-2, \\
X_i X_j &- X_j X_i, \quad && i ∈ \{1, …, k\}, j < i.
\end{alignat*}
These elements together form a reduction-unique reduction system $ R = \{(s, φ_s)\}_{s ∈ S} $. An irreducible path with respect to $ R $ has the form $ X_1^{a_1} … X_k^{a_k} σ_{b_1} … σ_{b_l} $ where the word $ σ_{b_1} … σ_{b_l} $ contains none of the terms $ σ_i^2 $ or $ σ_i σ_{i+1} σ_i $ or $ σ_i σ_j $ with $ j ≤ i-2 $.

Let us now determine the sets $ S_n $ of ambiguities. We have $ S_0 = \{*\} $ and $ S_1 = \{X_1, …, X_k, σ_1, …, σ_{k-1}\} $ and $ S_2 = S $. The set $ S_3 $ is the set of 1-ambiguities, namely
\begin{align*}
S_3 &= \{σ_i X_i X_j\}_{j < i} ∪ \{σ_i X_{i+1} X_j\}_{1 ≤ i ≤ k-1; j < i} ∪ \{σ_i X_j X_l\}_{j ≠ i, i+1; l < j} ∪ \{σ_i^2 X_i\} ∪ \{σ_i^2 X_{i+1}\} \\
& \quad ∪ \{σ_i^2 X_j\}_{j ≠ i, i+1} ∪ \{σ_i^3\} ∪ \{σ_i^2 σ_{i+1} σ_i\} ∪ \{σ_i^2 σ_j\}_{j ≤ i-2} ∪ \{σ_i σ_{i+1} σ_i X_i\} ∪ \{σ_i σ_{i+1} σ_i X_{i+1}\} \\
& \quad ∪ \{σ_i σ_{i+1} σ_i X_j\}_{j ≠ i, i+1} ∪ \{σ_i σ_{i+1} σ_i^2\} ∪ \{σ_i σ_{i+1} σ_i σ_j\}_{j ≤ i-2} ∪ \{σ_i σ_j X_j\}_{j ≤ i-2} ∪ \{σ_i σ_j X_{j+1}\}_{j ≤ i-2} \\
& \quad ∪ \{σ_i σ_j X_l\}_{j ≤ i-2; l ≠ j, j+1} ∪ \{σ_i σ_j^2\}_{j ≤ i-2} ∪ \{σ_i σ_j σ_{j+1} σ_j\}_{j ≤ i-2} ∪ \{σ_i σ_j σ_l\}_{j ≤ i-2; l ≤ j-2} \\
& \quad ∪ \{X_i X_j X_l\}_{1 ≤ i < j < l ≤ k}.
\end{align*}
The indices $ i, j, l $ in these enumerations all run from $ 1 $ to $ k $, or to $ k-1 $ or $ k-2 $ in case of obvious subscript constraints.

The Hochschild cohomology of $ \NH_{k, ħ=0} $ inherits a $ ℤ^2 $-grading. We shall be interested primarily in the component $ \HH^1 (\NH_{k, ħ=0})_{(0, -1)} $ of degree $ (0, -1) $. We have
\begin{equation*}
\HH^1 (\NH_{k, ħ=0})_{(0, -1)} = \H^1 (P^{•}[1]_{(0, -1)}).
\end{equation*}
We shall now determine the part of $ \Hom_{\Bimod A} (A ¤ ℂS_n ¤ A, A) $ in degree $ (0, -1) $. Observe that this space is isomorphic to $ \Map(S_n, A) $ as $ ℤ^2 $-graded space. Let $ f ∈ \Map(S_n, A)_{(0, -1)} $. Pick $ s ∈ S_n $ and regard an element $ s ∈ S_n $ of the form $ X^a σ^b $ and a linear component $ X^c σ^d $ of $ f(s) $. We have $ \deg(X^a σ^b) = (2a-2b, b) $ thus $ (2c-2d, d) = (2a-2b, b) + (0, -1) $, thus $ d = b-1 $ and $ c-d = a-b $ thus $ c-(b-1) = a-b $ thus $ c = a-b+(b-1) = a-1 $ thus $ (c, d) = (a, b) + (-1, -1) $. For $ (a, b) = (1, 1) $ for example we get $ (c, d) = (0, 0) $.

It is our task to compute the kernel modulo image in degree $ (0, -1) $ of the sequence $ \Map(S_1, A) → \Map(S_2, A) → \Map(S_3, A) $. We have $ \Map(S_1, A)_{(0, -1)} = 0 $.


We have $ ∂_3: A ¤ S_3 ¤ A → A ¤ S_2 ¤ A $ given by
\begin{align*}
∂_3 (1 ¤ w ¤ 1) &= ((\id - ρ_1 ∂_2)δ_3)(1 ¤ w ¤ 1) \\
&= (\id - ρ_1 ∂_2) (1 ¤ u_2 u_1 ¤ u_0 - u_2 ¤ u_1 u_0 ¤ 1) \\
&= 1 ¤ u_2 u_1 ¤ u_0 - u_2 ¤ u_1 u_0 ¤ 1 - ρ_1 (∂_2 (1 ¤ u_2 u_1 ¤ u_0 - u_2 ¤ u_1 u_0 ¤ 1)) \\
&= 1 ¤ u_2 u_1 ¤ u_0 - u_2 ¤ u_1 u_0 ¤ 1 - ρ_1 (\pathsplit_1 (u_2 u_1 - φ_{u_2 u_1}) u_0 - u_2 \pathsplit_1 (u_1 u_0 - φ_{u_1 u_0}))
\end{align*}
For $ j < i $ we have
\begin{align*}
∂_3 (1 ¤ σ_i X_i X_j ¤ 1) &= 1 ¤ σ_i X_i ¤ X_j - 1 ¤ σ_i X_j ¤ X_i - X_j ¤ σ_i X_i ¤ 1 \\
& \quad - σ_i ¤ X_i X_j ¤ 1 + 1 ¤ X_{i+1} X_j ¤ σ_i + X_{i+1} ¤ σ_i X_j ¤ 1.
\end{align*}
For $ j < i $ we have
\begin{align*}
∂_3 (1 ¤ σ_i X_{i+1} X_j ¤ 1)
&= 1 ¤ σ_i X_{i+1} ¤ X_j
- σ_i ¤ X_{i+1} X_j ¤ 1
- 1 ¤ σ_i X_j ¤ X_{i+1} \\
& \quad
- X_j ¤ σ_i X_{i+1} ¤ 1
+ 1 ¤ X_i X_j ¤ σ_i
+ X_i ¤ σ_i X_j ¤ 1
\end{align*}
For $ j ≠ i,i+1 $ and $ l < j $ we have
\begin{alignat*}{2}
∂_3 (1 ¤ σ_i X_j X_l ¤ 1)
&=
1 ¤ σ_i X_j ¤ X_{i+1}
- σ_i ¤ X_j X_{i+1} ¤ 1
- 1 ¤ σ_i X_{i+1} ¤ X_j \\
& \quad - X_i ¤ σ_i X_j ¤ 1
+ 1 ¤ X_j X_i ¤ σ_i
+ X_j ¤ σ_i X_{i+1} ¤ 1
&&
\quad \text{ if } l = i+1, \\
∂_3 (1 ¤ σ_i X_j X_l ¤ 1) &=
1 ¤ σ_i X_j ¤ X_i
- σ_i ¤ X_j X_i ¤ 1
- 1 ¤ σ_i X_i ¤ X_j \\
& \quad - X_{i+1} ¤ σ_i X_j ¤ 1
+ 1 ¤ X_j X_{i+1} ¤ σ_i
+ X_j ¤ σ_i X_i ¤ 1
&&
\quad \text{ if } l = i, \\
∂_3 (1 ¤ σ_i X_j X_l ¤ 1) &=
1 ¤ σ_i X_j ¤ X_l
- σ_i ¤ X_j X_l ¤ 1
- 1 ¤ σ_i X_l ¤ X_j \\
& \quad - X_l ¤ σ_i X_j ¤ 1
+ 1 ¤ X_j X_l ¤ σ_i
+ X_j ¤ σ_i X_l ¤ 1
&&
\quad \text{ if } l ≠ i,i+1.
\end{alignat*}
For all $ i $ we have
\begin{align*}
∂_3 (1 ¤ σ_i^2 X_i ¤ 1) &=
1 ¤ σ_i^2 ¤ X_i
- σ_i ¤ σ_i X_i ¤ 1
- 1 ¤ σ_i X_{i+1} ¤ σ_i
- X_i ¤ σ_i^2 ¤ 1
\end{align*}
For all $ i $ we have
\begin{align*}
∂_3 (1 ¤ σ_i^2 X_{i+1} ¤ 1)
&= 1 ¤ σ_i^2 ¤ X_{i+1} ¤ 1
- σ_i ¤ σ_i X_{i+1} ¤ 1
- 1 ¤ σ_i X_i ¤ σ_i
- X_{i+1} ¤ σ_i σ_i ¤ 1
\end{align*}
For $ j ≠ i,i+1 $ we have
\begin{align*}
∂_3 (1 ¤ σ_i^2 X_j ¤ 1)
&=
1 ¤ σ_i^2 ¤ X_j
- σ_i ¤ σ_i X_j ¤ 1
- 1 ¤ σ_i X_j ¤ σ_i
- X_j ¤ σ_i σ_i ¤ 1
\end{align*}
For all $ i $ we have
\begin{align*}
∂_3 (1 ¤ σ_i σ_{i+1} σ_i X_i ¤ 1) &=
1 ¤ σ_i σ_{i+1} σ_i ¤ X_i
- σ_i σ_{i+1} ¤ σ_i X_i ¤ 1
- 1 ¤ σ_i X_{i+2} ¤ σ_{i+1} σ_i \\
& \quad
- X_{i+2} ¤ σ_i σ_{i+1} σ_i ¤ 1
- σ_i ¤ σ_{i+1} X_{i+1} ¤ σ_i
+ 1 ¤ σ_{i+1} X_{i+1} ¤ σ_i σ_{i+1} \\
& \quad
+ σ_{i+1} ¤ σ_i X_i ¤ σ_{i+1}
+ σ_{i+1} σ_i ¤ σ_{i+1} X_i ¤ 1
\end{align*}
For all $ i $ we have
\begin{align*}
∂_3 (1 ¤ σ_i σ_{i+1} σ_i X_{i+1} ¤ 1)
&= 1 ¤ σ_i σ_{i+1} σ_i ¤ X_{i+1}
- σ_i σ_{i+1} ¤ σ_i X_{i+1} ¤ 1
- 1 ¤ σ_i X_i ¤ σ_{i+1} σ_i \\
& \quad
- X_{i+1} ¤ σ_i σ_{i+1} σ_i ¤ 1
- σ_i ¤ σ_{i+1} X_i ¤ σ_i
- σ_i σ_{i+1} ¤ σ_i X_{i+1} ¤ 1 \\
& \quad
+ 1 ¤ σ_{i+1} X_{i+2} ¤ σ_i σ_{i+1}
+ σ_{i+1} ¤ σ_i X_{i+2} ¤ σ_{i+1}
+ σ_{i+1} σ_i ¤ σ_{i+1} X_{i+1} ¤ 1 \\
& \quad
+ σ_i σ_{i+1} ¤ σ_i X_{i+1} ¤ 1
\end{align*}
For $ j ≠ i,i+1 $ we have
\begin{align*}
∂_3 (1 ¤ σ_i σ_{i+1} σ_i X_j ¤ 1)
&=
1 ¤ σ_i σ_{i+1} σ_i ¤ X_j
- σ_i σ_{i+1} ¤ σ_i X_j ¤ 1
- 1 ¤ σ_i X_{i+1} ¤ σ_{i+1} σ_i \\
& \quad
- X_i ¤ σ_i σ_{i+1} σ_i ¤ 1
- σ_i ¤ σ_{i+1} X_{i+2} ¤ σ_i
- σ_i σ_{i+1} ¤ σ_i X_{i+2} ¤ 1 \\
& \quad
+ 1 ¤ σ_{i+1} X_i ¤ σ_i σ_{i+1}
+ σ_{i+1} ¤ σ_i X_{i+1} ¤ σ_{i+1}
+ σ_{i+1} σ_i ¤ σ_{i+1} X_{i+2} ¤ 1 \\
& \quad
+ σ_i σ_{i+1} ¤ σ_i X_{i+2} ¤ 1
\quad \text{ if } j = i+2, \\
∂_3 (1 ¤ σ_i σ_{i+1} σ_i X_j ¤ 1)
&= 1 ¤ σ_i σ_{i+1} σ_i ¤ X_j
- σ_i σ_{i+1} ¤ σ_i X_j ¤ 1
- 1 ¤ σ_i X_j ¤ σ_{i+1} σ_i \\
& \quad
- X_j ¤ σ_i σ_{i+1} σ_i ¤ 1
- σ_i ¤ σ_{i+1} X_j ¤ σ_i
- σ_i σ_{i+1} ¤ σ_i X_j ¤ 1 \\
& \quad
+ 1 ¤ σ_{i+1} X_j ¤ σ_i σ_{i+1}
+ σ_{i+1} ¤ σ_i X_j ¤ σ_{i+1}
+ σ_{i+1} σ_i ¤ σ_{i+1} X_j ¤ 1 \\
& \quad
+ σ_i σ_{i+1} ¤ σ_i X_j ¤ 1
\quad \text{ if } j ≠ i+2.
\end{align*}
For $ j ≤ i-2 $ we have
\begin{align*}
∂_3 (1 ¤ σ_i σ_j X_j ¤ 1)
&= 1 ¤ σ_i σ_j ¤ X_j
- σ_i ¤ σ_j X_j ¤ 1
- 1 ¤ σ_i X_{j+1} ¤ σ_j \\
& \quad
- X_{j+1} ¤ σ_i σ_j ¤ 1
- σ_i ¤ σ_j X_j ¤ 1
+ 1 ¤ σ_j X_j ¤ σ_i \\
& \quad
+ σ_j ¤ σ_i X_j ¤ 1
+ σ_i ¤ σ_j X_j ¤ 1
\end{align*}
For $ j ≤ i-2 $ we have
\begin{align*}
∂_3 (1 ¤ σ_i σ_j X_{j+1} ¤ 1) &=
1 ¤ σ_i σ_j ¤ X_{j+1}
- σ_i ¤ σ_j X_{j+1} ¤ 1
- 1 ¤ σ_i X_j ¤ σ_j \\
& \quad
- X_j ¤ σ_i σ_j ¤ 1
- σ_i ¤ σ_j X_{j+1} ¤ 1
+ 1 ¤ σ_j X_{j+1} ¤ σ_i \\
& \quad
+ σ_j ¤ σ_i X_{j+1} ¤ 1
+ σ_i ¤ σ_j X_{j+1} ¤ 1
\end{align*}
For $ j ≤ i-2 $ and $ l ≠ j,j+1 $ we have
\begin{alignat*}{2}
∂_3 (1 ¤ σ_i σ_j X_l ¤ 1) &=
1 ¤ σ_i σ_j ¤ X_i
- σ_i ¤ σ_j X_i ¤ 1
- 1 ¤ σ_i X_i ¤ σ_j \\
& \quad 
- X_{i+1} ¤ σ_i σ_j ¤ 1
- σ_i ¤ σ_j X_i ¤ 1
+ 1 ¤ σ_j X_{i+1} ¤ σ_i \\
& \quad
+ σ_j ¤ σ_i X_i ¤ 1
+ σ_i ¤ σ_j X_i ¤ 1
&& \quad \text{if } l = i, \\
&= 1 ¤ σ_i σ_j ¤ X_{i+1}
- σ_i ¤ σ_j X_{i+1} ¤ 1
- 1 ¤ σ_i X_{i+1} ¤ σ_j \\
& \quad
- X_i ¤ σ_i σ_j ¤ 1
- σ_i ¤ σ_j X_{i+1} ¤ 1
+ 1 ¤ σ_j X_i ¤ σ_i \\
& \quad
+ σ_j ¤ σ_i X_{i+1} ¤ 1
+ σ_i ¤ σ_j X_{i+1} ¤ 1
&& \quad \text{if } l = i+1, \\
&= 1 ¤ σ_i σ_j ¤ X_l
- σ_i ¤ σ_j X_l ¤ 1
- 1 ¤ σ_i X_l ¤ σ_j \\
& \quad
- X_l ¤ σ_i σ_j ¤ 1
- σ_i ¤ σ_j X_l ¤ 1
+ 1 ¤ σ_j X_l ¤ σ_i \\
& \quad
+ σ_j ¤ σ_i X_l ¤ 1 
+ σ_i ¤ σ_j X_l ¤ 1 
&& \quad \text{if } l ≠ i,i+1.
\end{alignat*}

We are now ready to compute the graded component of Hochschild cohomology.

\begin{lemma}
An element $ ν ∈ \Map(S_2, A)_{(0, -1)} $ is of the following form with scalars $ A_i, B_i, C_{i, j} $:
\begin{align*}
ν(σ_i X_i) = A_i, \quad ν(σ_i X_{i+1}) = B_i, \quad ν(σ_i X_j) = C_{i, j}, \quad ν(σ_i^2) = ν(σ_i σ_{i+1} σ_i) = ν(σ_i σ_j) = ν(X_i X_j) = 0.
\end{align*}
The kernel of $ ∂^2: \Map(S_2, A)_{(0, -1)} → \Map(S_3, A)_{(0, -1)} $ is one-dimensional and spanned by the following element $ ν_0 $:
\begin{equation*}
ν_0 (σ_i X_i) = 1, \quad ν_0 (σ_i X_{i+1}) = -1, \quad ν_0 (σ_i X_j) = 0, \quad ν_0 (σ_i^2) = ν_0 (σ_i σ_{i+1} σ_i) = ν_0 (σ_i σ_j) = ν_0 (X_i X_j) = 0.
\end{equation*}
\end{lemma}

\begin{proof}
We shall ignore the overall signs of $ ∂^2 $ for this proof. The element $ ∂^2 (ν) ∈ \Map(S_3, A)_{(0, -1)} $ is determined by the following equation for $ w ∈ S_3 $:
\begin{align*}
∂^2 (ν)(1 ¤ w ¤ 1) &= ν(∂_3 (1 ¤ w ¤ 1)).
\end{align*}
We shall now calculate $ ν(∂_3 (1 ¤ w ¤ 1)) $ in terms of $ A_i $, $ B_i $, $ C_{i,j} $ for all $ w ∈ S_3 $.
For $ j < i $ we have
\begin{align*}
∂^2 (ν)(1 ¤ σ_i X_i X_j ¤ 1) &= ν(∂_3 (1 ¤ σ_i X_i X_j ¤ 1)) \\
&= A_i X_j - C_{i,j} X_i - A_i X_j - C_{i,j} X_{i+1} = - C_{i,j} (X_i - X_{i+1}).
\end{align*}
For $ j < i $ we have
\begin{align*}
∂^2 (ν)(1 ¤ σ_i X_{i+1} X_j ¤ 1) &= ν(∂_3 (1 ¤ σ_i X_{i+1} X_j ¤ 1)) \\
&= B_i X_j - C_{i,j} X_{i+1} - B_i X_j + C_{i,j} X_i = C_{i,j} (X_i - X_{i+1}).
\end{align*}
For $ j ≠ i,i+1 $ and $ l < j $ we have
\begin{align*}
∂^2 (ν)(1 ¤ σ_i X_j X_l ¤ 1) &= ν(∂_3 (1 ¤ σ_i X_j X_l ¤ 1)) \\
&= \begin{cases}
C_{i,j} X_{i+1} - B_i X_j - C_{i,j} X_i + B_i X_j = C_{i,j} (X_{i+1} - X_i) & \text{if } l = i+1, \\
C_{i,j} X_i - A_i X_j - C_{i,j} X_{i+1} + A_i X_j = C_{i,j} (X_i - X_{i+1}) & \text{if } l = i, \\
C_{i,j} X_l - C_{i,l} X_j - C_{i,j} X_l + C_{i,l} X_j = 0 & \text{if } l ≠ i,i+1.
\end{cases}
\end{align*}
For all $ i $ we have
\begin{align*}
∂^2 (ν)(1 ¤ σ_i^2 X_i ¤ 1) &= ν(∂_3 (1 ¤ σ_i^2 X_i ¤ 1)) \\
&= - A_i σ_i - B_i σ_i = - (A_i + B_i) σ_i.
\end{align*}
For all $ i $ we have
\begin{align*}
∂^2 (ν)(1 ¤ σ_i^2 X_{i+1} ¤ 1) &= ν(∂_3 (1 ¤ σ_i^2 X_{i+1} ¤ 1)) \\
&= - B_i σ_i - A_i σ_i = -(A_i + B_i) σ_i.
\end{align*}
For $ j ≠ i,i+1 $ we have
\begin{align*}
∂^2 (ν)(1 ¤ σ_i^2 X_j ¤ 1) &= ν(∂_3 (1 ¤ σ_i^2 X_j ¤ 1)) \\
&= - C_{i,j} σ_i - C_{i,j} σ_i = 0.
\end{align*}
For all $ i $ we have
\begin{align*}
∂^2 (ν)(1 ¤ σ_i σ_{i+1} σ_i X_i ¤ 1) &= ν(∂_3 (1 ¤ σ_i σ_{i+1} σ_i X_i ¤ 1)) \\
&= (A_{i+1} - A_i) σ_i σ_{i+1} + (C_{i+1, i} - C_{i, i+2}) σ_{i+1} σ_i.
\end{align*}
For all $ i $ we have
\begin{align*}
∂^2 (ν)(1 ¤ σ_i σ_{i+1} σ_i X_{i+1} ¤ 1) &= ν(∂_3 (1 ¤ σ_i σ_{i+1} σ_i X_{i+1} ¤ 1)) \\
&= (A_{i+1} - A_i) σ_{i+1} σ_i + (B_{i+1} - B_i) σ_i σ_{i+1}.
\end{align*}
For $ j ≠ i,i+1 $ we have
\begin{align*}
∂^2 (ν)(1 ¤ σ_i σ_{i+1} σ_i X_j ¤ 1) &= ν(∂_3 (1 ¤ σ_i σ_{i+1} σ_i X_j ¤ 1)) \\
&= \begin{cases}
(C_{i+1,i} - C_{i,j}) σ_i σ_{i+1} + (B_{i+1} - B_i) σ_{i+1} σ_i & \text{ if } j = i+2, \\
(C_{i+1,j} - C_{i,j}) σ_i σ_{i+1} + (C_{i+1,j} - C_{i,j}) σ_{i+1} σ_i & \text{ if } j ≠ i+2.
\end{cases}
\end{align*}
For $ j ≤ i-2 $ we have
\begin{align*}
∂^2 (ν)(1 ¤ σ_i σ_j X_j ¤ 1) &= ν(∂_3 (1 ¤ σ_i σ_j X_j ¤ 1)) \\
&= (C_{i,j} - C_{i,j+1}) σ_j.
\end{align*}
For $ j ≤ i-2 $ we have
\begin{align*}
∂^2 (ν)(1 ¤ σ_i σ_j X_{j+1} ¤ 1) &= ν(∂_3 (1 ¤ σ_i σ_j X_{j+1} ¤ 1)) \\
&= (C_{i,j+1} - C_{i,j}) σ_j.
\end{align*}
For $ j ≤ i-2, l ≠ j,j+1 $ we have
\begin{align*}
∂^2 (ν)(1 ¤ σ_i σ_j X_l ¤ 1) &= ν(∂_3 (1 ¤ σ_i σ_j X_l ¤ 1)) \\
&= \begin{cases}
(- 2 C_{j,i} + C_{j,i+1} + C_{j,i}) σ_i & \text{if } l = i, \\
(- 2 C_{j,i+1} + C_{j,i} + C_{j,i+1}) σ_i & \text{if } l = i+1, \\
(- 2 C_{j,l} + C_{j,l} + C_{j,l}) σ_i & \text{if } l ≠ i,i+1, \\
\end{cases}
\end{align*}
Assuming that $ ν ∈ \Ker(∂^2) $ yields that $ C_{i,j} = 0 $ for all $ i ∈ \{1, …, k-1\} $ and $ j ≠ i,i+1 $, and further $ A_i = - B_i = A_0 $ for all $ i ∈ \{1, …, k\} $ for a certain scalar $ A_0 $. Finally, we have that $ ν = A_0 ν_0 $ where $ ν_0 ∈ \Map(S_2, A) $ is the map given by sending $ σ_i X_i $ to $ 1 $, $ σ_i X_{i+1} $ to $ -1 $ and all other elements to zero.
\end{proof}

The space $ \Map(S_3, A)_{(0, -1)} $ is spanned by the functions $ f $ which take zero values everywhere except on a single element, namely one of the following:
\begin{align*}
f(σ_i X_i X_j) &= X_l, \\
f(σ_i X_{i+1} X_j) &= X_l, \\
f(σ_i X_j X_l) &= X_m, \\
f(σ_i^2 X_i) &= σ_j, \\
f(σ_i^2 X_{i+1}) &= σ_j, \\
f(σ_i^2 X_j) &= σ_l, \\
f(σ_i σ_{i+1} σ_i X_i) &= σ_j σ_l, \\
f(σ_i σ_{i+1} σ_i X_{i+1}) &= σ_j σ_l, \\
f(σ_i σ_{i+1} σ_i X_j) &= σ_l σ_m, \\
f(σ_i σ_j X_j) &= σ_l, \\
f(σ_i σ_j X_{j+1}) &= σ_l, \\
f(σ_i σ_j X_l) &= σ_m.
\end{align*}

\subsection{Classification of $ ℤ^2 $-graded deformations of $ \NH_{k, ħ=0} $}
\label{sec:unique-classification}
In this section, we classify the $ ℤ^2 $-graded deformations of $ \NH_{k, ħ=0} $. More precisely, we determine the gauge equivalence classes of $ ℤ^2 $-graded deformations over $ ℂ⟦ħ⟧_{(0, 1)} $ and show how to determine the equivalence class for a given deformation. We also explain that all non-trivial deformations are isomorphic as $ A_∞ $-categories.

\begin{lemma}
Denote by $ ν_0 ∈ \HH(\NH_{k, ħ=0})_{(0, -1)} $ the standard element of Hochschild cohomology. We have
\begin{equation*}
\MC(\HH(\NH_{k, ħ=0}), ℂ⟦ħ⟧_{(0, 1)})_0 = \{A ħ ν_0 \running A ∈ ℂ\}.
\end{equation*}
\end{lemma}

\begin{proof}
Let $ ν ∈ \MC(\HH(\NH_{k, ħ=0}), ℂ⟦ħ⟧_{(0, 1)})_0 $. By \autoref{th:prelimdefo-gradedtensor-tensordecomp}, we can have a decomposition
\begin{equation*}
ν = \sum_{i = 1}^∞ ħ^i ν_i, \quad \text{where } ν_i ∈ \HH^1 (\NH_{k, ħ=0})_{(0, -i)}.
\end{equation*}
Since $ \HH^1 (\NH_{k, ħ=0})_{(0, -i)} $ is $ \vspan(ν_0) $ in case $ i = 1 $ and zero otherwise, we finish the proof.
\end{proof}

\begin{remark}
As Maurer-Cartan elements, the elements $ A ħ ν_0 $ are all pairwise non-gauge-equivalent.
\end{remark}

Let us comment now on the algebraicity question, asking which deformations of $ \NH_{k, ħ=0} $ can be defined over $ ℂ[ħ] $ instead of $ ℂ⟦ħ⟧ $. As it turns out, all graded deformations and graded gauge functors are at least locally algebraic:

\begin{lemma}
Any element $ ν ∈ \MC(\NH_{k, ħ=0}, ℂ⟦ħ⟧_{(0, 1)})_0 $ is automatically locally algebraic. Similarly, any $ ℤ^2 $-graded functor $ F: ν → ν' $ between any deformations of $ \NH_{k, ħ=0} $ is automatically locally algebraic.
\end{lemma}

\begin{proof}
We start with the first statement. Pick any sequence of $ ℤ^2 $-homogeneous elements $ a_1, …, a_m ∈ \NH_{k, ħ=0} $. Let us write the element $ ν(a_m, …, a_1) ∈ ℂ⟦ħ⟧_{(0, 1)} \htensor \NH_{k, ħ=0} $ in the form
\begin{equation*}
ν(a_m, …, a_1) = \sum_{i = 0}^∞ ħ^i c_i, \quad \text{where } c_i ∈ \NH_{k, ħ=0}.
\end{equation*}
The elements $ c_i $ are not a priori homogeneous. Let us write $ c_i = \sum_{a ∈ ℤ^2} c_i^a $ where $ \deg(c_i^a) = a $. Since $ ν $ is homogeneous and $ a_1, …, a_m $ are homogeneous, the right-hand side of the equation is homogeneous of degree $ D = \deg ν + \deg a_1 + … + \deg a_m $. In terms of the $ c_i^a $ elements, the component of this degree is
\begin{equation*}
\sum_{i = 0}^∞ \sum_{\substack{a ∈ ℤ^2 \\ (0, i) + a = D}}  ħ^i c_i^a.
\end{equation*}
Since $ \NH_{k, ħ=0} $ is concentrated in nonnegative $ ħ $-degrees, the inner sum is empty for $ i > D_2 $. Since $ ν(a_m, …, a_1) $ is exactly this sum, we conclude that $ ν(a_1, …, a_1) ∈ ℂ[ħ] ¤ \NH_{k, ħ=0} $. This finishes the first statement. The second statement is proven in an analogous fashion. This finishes the proof.
\end{proof}

We are now ready to show that any nontrivial graded deformation is $ A_∞ $-isomorphic to the standard deformation $ \NH_k $ of $ \NH_{k, ħ=0} $. As preparation, recall that gauge equivalence of deformations refers to the existence of a $ ℂ⟦ħ⟧_{(0, 1)} $-linear (curved) functor between the two deformations that is the identity up to higher $ ħ $-orders. Meanwhile, the statement we prove establishes an $ A_∞ $-equivalence of a given deformation with the standard deformation. This refers to an $ A_∞ $-functor that is merely linear over the base field.

\begin{proposition}
\label{th:unique-unique-iso}
Let $ ν ∈ \MC(\NH_{k, ħ=0}, ℂ⟦ħ⟧_{(0, 1)})_0 $ and assume that $ ν $ is a nontrivial deformation. Then there exists a $ ℤ^2 $-homogeneous $ A_∞ $-isomorphism $ F: ν → \NH_k $.
\end{proposition}

\begin{proof}
We can write $ π^{\MC} (ν) = A ħ ν_0 $ for some scalar $ A $. Denote by $ ν' ∈ \MC(\HH(\NH_{k, ħ=0}), ℂ⟦ħ⟧_{(0, 1)})_0 $ the Maurer-Cartan element associated with the standard deformation $ \NH_k $ of $ \NH_{k, ħ=0} $ and write $ π^{\MC} (ν') = A' ħ ν_0 $. Define the automorphism of deformation bases $ ψ: ℂ⟦ħ⟧_{(0, 1)} → ℂ⟦ħ⟧_{(0, 1)} $ given by the natural substitution $ ψ(A ħ) = A' ħ $. By \autoref{th:prelim-defo-substitution} we have have a $ ℤ^2 $-homogeneous isomorphism $ F^1: i^{\MC} (A ħ ν_0) → i^{\MC} (A' ħ ν_0) $.

By \autoref{th:prelim-defo-pushpull}, there are $ ℤ^2 $-homogeneous isomorphisms $ ν → i^{\MC} (A ħ ν_0) $ and $ ν' → i^{\MC} (A' ħ ν_0) $. The curvature of these functors is of the form $ F^0 = ħ F^0_1 + ħ^2 F^0_2 + … $ and is a $ ℤ^2 $-homogeneous element of degree zero in $ ℂ⟦ħ⟧_{(0, 1)} \htensor \NH_{k, ħ=0} $. In consequence, the curvature vanishes. Finally combining all three $ A_∞ $-isomorphisms into a diagram, we obtain
\begin{equation*}
\begin{tikzcd}
i^{\MC} (A ħ ν_0) \arrow[r, "\sim"] & i^{\MC} (A' ħ ν_0) \\
ν \arrow[u, "\sim"] & ν' \arrow[u, "\sim"].
\end{tikzcd}
\end{equation*}
All three maps are $ ℤ^2 $-homogeneous. This finishes the proof.
\end{proof}

Let us now parametrize the $ ℤ^2 $-graded deformations of $ \NH_{k, ħ=0} $ explicitly. In particular, we provide an explicit description of the deformation $ i^{\MC} (A ħ ν_0) $ for scalars $ A $. We also provide an explicit way to determine for a given deformation $ ν $ the scalar $ A $ such that $ π^{\MC} (ν) = A ħ ν_0 $.

As a starting point, write $ A = \NH_{k, ħ=0} $ and consider the $ L_∞ $-mapping $ φ: P^{•}[1] → \HC(A) $ together with the $ L_∞ $-mapping $ i': \HH(A) → P^{•}[1] $. Our task is to compute the composition $ φ ∘ i' $ on Maurer-Cartan elements. While we could in principle compute these $ L_∞ $-maps explicitly by hand, it is more economical in the given circumstances to identify the correct Maurer-Cartan elements via their components of first order in $ ħ $. 

\begin{lemma}
\label{th:unique-param-exist}
Let $ ν ∈ \MC(\HC(\NH_{k, ħ=0}), ℂ⟦ħ⟧_{(0, 1)})_0 $. Let $ i ∈ \{1, …, k-1\} $. Then there exists a scalar $ A $ such that
\begin{equation*}
ν(X_i, σ_i) - ν(σ_i, X_{i+1}) = A ħ.
\end{equation*}
If $ ν, ν' $ are two such elements with scalars $ A $ and $ A' $, then $ ν $ and $ ν' $ are $ ℤ^2 $-graded gauge-equivalent if and only if $ A = A' $.
\end{lemma}

\begin{proof}
We start with the observation that $ ν(X_i, σ_i) $ and $ ν(σ_i, X_{i+1}) $ have degree $ (0, 1) $. Therefore both are scalar multiples of $ ħ $. Let now $ F: ν → ν' $ be a $ ℤ^2 $-graded gauge-equivalence. Since $ \NH_{k, ħ=0} $ is an ordinary algebra, $ F $ is a strict unital functor and an ordinary morphism of algebra deformations. We have
\begin{align*}
A ħ &= F((μ + ν)(X_i, σ_i) - (μ + ν)(σ_i, X_{i+1})) \\
&= (μ + ν')(F(X_i), F(σ_i)) - (μ + ν')(F(σ_i), F(X_{i+1})) \\
&= (μ + ν')(X_i + \landau(ħ), σ_i + \landau(ħ)) - (μ + ν')(σ_i + \landau(ħ), X_{i+1} + \landau(ħ)) \\
&= ħ \landau(X_j, σ_j) + A' ħ + \landau(ħ^2).
\end{align*}
We conclude that $ A = A' $. The converse statement follows from \autoref{th:prelim-defo-pushpull}. This finishes the proof.
\end{proof}

\begin{lemma}
The scalar associated with the deformation $ i^{\MC} (A ħ ν_0) $ is $ A $.
\end{lemma}

\begin{proof}
Regard the deformation $ ν = i^{\MC} (A ħ ν_0) $. Then $ ν $ can be determined explicitly according to \cite{Barmeier-Wang} or alternatively its shape up to $ ħ^2 $ terms can be determined via the first components of the $ L_∞ $-morphisms $ i' $ and $ φ $. We arrive at
\begin{align*}
(μ + ν)(X_i, σ_i) &= σ_i X_{i+1} + A ħ ν_0 (X_i σ_i) + \landau(ħ^2) \\
(μ + ν)(σ_i, X_{i+1}) &= σ_i X_{i+1}.
\end{align*}
Recalling that $ ν_0 (X_i σ_i) = 1 $, we conclude that the scalar associated with $ ν $ is $ A $.
\end{proof}

\begin{lemma}
Let $ ν ∈ \MC(\HC(\NH_{k, ħ=0}), ℂ⟦ħ⟧_{(0, 1)})_0 $ be the Maurer-Cartan element defined by the deformation $ \NH_k $ of $ \NH_{k, ħ=0} $. Then we have $ π^{\MC} (ν) = ħ ν_0 $.
\end{lemma}

\begin{proof}
Write $ π^{\MC} (ν) = A ħ ν_0 $. Then the scalar associated with $ i^{\MC} (π^{\MC} (ν)) $ is $ A $ and simultaneously there is a $ ℤ^2 $-graded gauge-equivalence with $ ν $ which has scalar $ 1 $. Thanks to \autoref{th:unique-param-exist} we conclude $ A = 1 $.
\end{proof}

\begin{definition}
For a scalar $ A $, define the deformation $ ν_A ∈ \MC(\HC(\NH_{k, ħ=0}), ℂ⟦ħ⟧_{(0, 1)})_0 $ as the deformation arising from the deformed reduction system with
\begin{align*}
φ_s + g_s &= X_{i+1} σ_i + A ħ, \quad s = σ_i X_i, \quad i ∈ \{1, …, k-1\}, \\
φ_s + g_s &= φ_s, \quad s ≠ σ_i X_i.
\end{align*}
\end{definition}

\begin{lemma}
A full set of non-gauge-equivalent $ ℤ^2 $-graded deformations of $ \NH_{k, ħ=0} $ is given by the elements $ ν_A $.
\end{lemma}

\begin{proof}
Let $ ν $ be any $ ℤ^2 $-graded deformation. Write $ π^{\MC} (ν) = A ħ $. Since $ ν_A $ has scalar $ A $ by definition, the scalars of $ ν $ and $ ν_A $ agree and by \autoref{th:unique-param-exist} we conclude that $ ν $ and $ ν_A $ are $ ℤ^2 $-graded gauge equivalent. Conversely, we recall that the elements $ A ħ $ are not gauge-equivalent in $ \MC(\HH(\NH_{k, ħ=0}), ℂ⟦ħ⟧_{(0, 1)})_0 $.
\end{proof}

\subsection{Main theorem}
\label{sec:unique-main}
In this section, we formulate the main theorem. This entails an isomorphism between $ \relFuk(Y_W)_L $ and the category $ \NH_k $. Recall that we abuse notation and denote by $ \NH_k $ the one-element category with endomorphism algebra $ \End(*) = \NH_k $. The argument is based on the relatively easy non-deformed description of $ \Fuk(Y_W)_L $ recalled in \autoref{sec:prelimhilbhor-calc} and part of the calculations from \cite[section 8]{ADLSZ}, however circumventing the final assembly of their result.

\begin{equation*}
\begin{tikzcd}
\relFuk(Y_W)_L \arrow[d, dash, "\text{defo}"] \arrow[r, "\sim"] & \NH_k \arrow[d, dash, "\text{defo}"] \\
\Fuk(Y_{O, W})_L \arrow[r, "\sim"] & \NH_{k, ħ=0}.
\end{tikzcd}
\end{equation*}

\begin{theorem}
\label{th:unique-main-th}
There is an isomorphism of $ A_∞ $-categories $ F: \relFuk(Y_W)_L \isoto \NH_k $.
\end{theorem}

\begin{proof}
We recall from \cite{ADLSZ} that the deformation $ \relFuk(Y_W)_L $ has nonzero $ A $ scalar and is thus a non-trivial deformation of $ \Fuk(Y_{O, W})_L $ over $ ℂ⟦ħ⟧_{(0, 1)} $. Along the standard identification $ \Fuk(Y_{O, W})_L $ of \autoref{sec:prelimhilbhor-calc}, this constitutes a non-trivial deformation of $ \NH_{k, ħ=0} $. Due to \autoref{th:unique-unique-iso} we conclude that there is an $ A_∞ $-isomorphism $ \relFuk(Y_W)_L → \NH_k $.
 \end{proof}

\begin{remark}
While both $ \relFuk(Y_W)_L $ and $ \NH_k $ are defined over $ ℂ⟦ħ⟧_{(0, 1)} $, the functor $ F $ is not necessarily $ ℂ⟦ħ⟧_{(0, 1)} $-linear. Indeed, its construction includes a potential substitution of the $ ħ $ variable.
\end{remark}

\appendix

\section{Nontriviality and examples}
\label{sec:nontriviality-examples}
In this section, we collect a direct approach to establish nontriviality of the deformation $ \relFuk(Y_{O, W})_L $ and examples of complexes in $ \Tw\NH_{k, ħ=0} $ together with their uncurvings in $ \Tw'\NH_k $.

\subsection{An approach to nontriviality of the deformation}
In this section, we approach the question whether the deformation $ \relFuk(Y_{O, W})_L $ is a nontrivial deformation. We construct a sequence of three Lagrangians in $ Y_{O, W} $ with three consecutive intersection points.

We start by recalling the coordinates $ (y_i) $ and $ (u_i) $ for $ Y_{O, W} $. We shall now construct a class of objects of $ \relFuk(Y_{O, W}) $. Denote by $ R ⊂ (ℂ^*)^2 $ the product of two standard lines in the fiber cylinder. For a position vector $ v ∈ ℂ \setminus \{0\} $ and a direction vector $ w ∈ ℂ $ we construct the object
\begin{equation*}
X'_{v, w} = \{(sw, v + tw) \running s, t ∈ ℝ\} × R ⊂ ℂ^2 × (ℂ^*)^2.
\end{equation*}
This subset $ X'_{v, w} ⊂ ℂ^2 × (ℂ^*)^2 $ does not touch the big diagonal and therefore descends to an object $ X ∈ \relFuk(Y_{O, W}) $. We shall now construct three such objects $ X_{v_i, w_i} $ for $ i = 1, 2, 3 $. We select $ w_1 = e^{πi/6} $, $ w_2 = e^{3πi/6} $ and $ w_3 = e^{5πi/6} $ and $ v_1 = e^{4πi/6} $, $ v_2 = e^{6πi/6} $ and $ v_3 = e^{8πi/6} $. The resulting objects $ X_{v_i, w_i} $ and their projection to the $ (y_1 - y_2)^2 $ coordinate is illustrated in \autoref{fig:nontriviality-triangle-lagrangians} and \ref{fig:nontriviality-triangle-projection}.


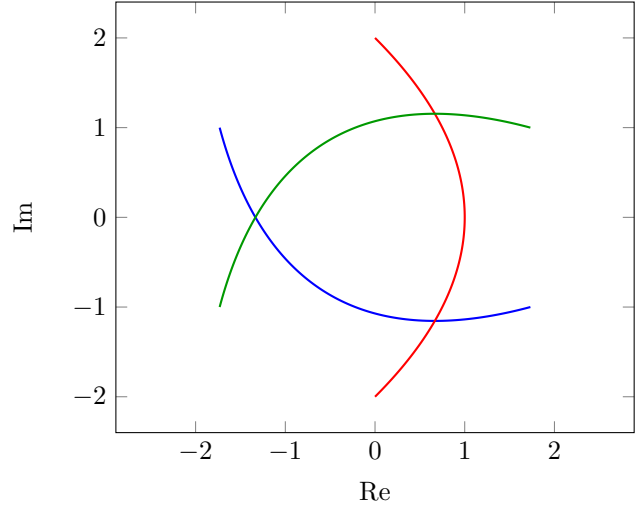
\begin{figure}
\centering
\begin{tikzpicture}
\begin{axis}[
    axis equal,
    xlabel={Re},
    ylabel={Im},
    domain=-1:1,
    samples=100,
]

\addplot [thick, blue] (
    {(( -1/2 + sqrt(3)*x/2 )^2 - ( sqrt(3)/2 + x/2 )^2)},
    {2*( -1/2 + sqrt(3)*x/2 )*( sqrt(3)/2 + x/2 )}
);

\addplot [thick, red] (
    {1 - x^2},
    {-2*x}
);

\addplot [thick, green!60!black] (
    {(( -1/2 - sqrt(3)*x/2 )^2 - ( -sqrt(3)/2 + x/2 )^2)},
    {2*( -1/2 - sqrt(3)*x/2 )*( -sqrt(3)/2 + x/2 )}
);

\end{axis}
\end{tikzpicture}
\caption{This figure depicts the three projections of the three objects $ X_{v_i, w_i} $ for $ i = 1, 2, 3 $ to the coordinate $ (y_1 - y_2)^2 $.}
\label{fig:nontriviality-triangle-projection}
\end{figure}

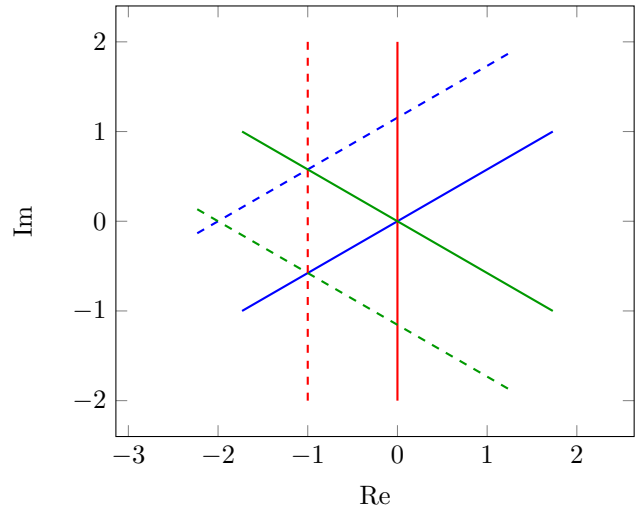
\begin{figure}
\centering
\begin{tikzpicture}
\begin{axis}[
    axis equal,
    xlabel={Re},
    ylabel={Im},
    domain=-2:2,
    samples=2, 
]

\addplot [thick, blue] (
    {sqrt(3)/2 * x},
    {1/2 * x}
);

\addplot [thick, blue, dashed] (
    {-1/2 + sqrt(3)/2 * x},
    {sqrt(3)/2 + 1/2 * x}
);

\addplot [thick, red] (
    {0},
    {x}
);

\addplot [thick, red, dashed] (
    {-1},
    {x}
);

\addplot [thick, green!60!black] (
    {-sqrt(3)/2 * x},
    {1/2 * x}
);

\addplot [thick, green!60!black, dashed] (
    {-1/2 - sqrt(3)/2 * x},
    {-sqrt(3)/2 + 1/2 * x}
);

\end{axis}
\end{tikzpicture}
\caption{This figure depicts the three two-dimensional Lagrangians $ X_{v_i, w_i} $ for $ i = 1, 2, 3 $. They are symmetric products of two individual lines. We have depicted each of the three objects in its own color.}
\label{fig:nontriviality-triangle-lagrangians}
\end{figure}

Let now construct three morphisms $ φ_1 ∈ \Hom(X_{v_1, w_1}, X_{v_2, w_2}) $, $ φ_2 ∈ \Hom(X_{v_2, w_2}, X_{v_3, w_3}) $ and $ φ_3 ∈ \Hom(X_{v_1, w_1}, X_{v_3, w_3}) $. Let us now construct the intersection points $ φ_1, φ_2, φ_3 $. Fix a point $ r = (r_1, r_2) ∈ R $. We have the intersection points
\begin{align*}
φ_1 &= (-1 + \frac{1}{\sqrt{3}} i, 0, r_1, r_2), \\
φ_2 &= (-1 - \frac{1}{\sqrt{3}} i, 0, r_1, r_2), \\
φ_3 &= (-2, 0, r_1, r_2).
\end{align*}

Let us now explain why $ μ^2 (φ_2, φ_1) = ħ φ_3 $. We shall indicate how to construct a disk $ φ: D \to Y_{O, W} $ where $ D $ is the unit disk with three consecutive punctures on the boundary which are mapped to $ φ_1 $, $ φ_2 $ and $ φ_3 $, respectively. We start by letting $ ψ: \mathbb{D} \to \mathbb{C} $ be a biholomorphism onto the area between the curves in the $ (y_1 - y_2)^2 $ picture. Denote by $ D $ be the unit disk with punctures given by the preimages of the three intersection points. We will now assign $ y_1 + y_2 = v_i $ so that one picks the points $ (y_1, y_2) $ given by $ y_1 = v_i - y_2 $ and such that $ (y_1 - y_2)^2 = (v_i - 2 y_2)^2 = ψ(x) $. There are two solutions but they differ exactly by swapping $ y_1 $ and $ y_2 $.

\subsection{Uncurving of twisted complexes}
In this section, we exhibit specific examples of twisted complexes in $ \Tw\Fuk(Y_{O, W})_L $. We start by constructing specific twisted complexes in $ \Tw\Fuk(Y_{O, W})_L $. Using the formalism of curved twisted completions recalled in \autoref{sec:prelim-liou-defo}, we explain their curvature as objects in $ \Tw\relFuk(Y_W)_L $. We determine for specific complexes whether they are uncurvable or not. Instead of viewing these twisted complexes as twisted complexes of objects of the Fukaya category, we construct them as twisted complexes over the single object of the NilHecke category $ \NH_{k, ħ=0} $.

Let us start with generalities on twisted complexes in $ \Tw\NH_{k, ħ=0} $. The single object $ * ∈ \NH_{k, ħ=0} $ has endomorphism algebra $ \End(*) = \NH_{k, ħ=0} $ concentrated in degree zero. Therefore a twisted complex in $ \NH_{k, ħ=0} $ takes the form
\begin{equation*}
C = (\bigoplus_{i ∈ ℤ} *[-i]^{⊕l_i}, δ) = … \overset{δ}{→} *^{⊕l_{-1}} \overset{δ}{→} *^{⊕l_0} \overset{δ}{→} *^{⊕l_1} → … ∈ \Tw\NH_{k, ħ=0}.
\end{equation*}
The condition to be a twisted complex comes down the set $ \{i ∈ ℤ \running l_i ≠ 0\} $ being finite and the square of $ δ $ as an ordinary matrix over the noncommutative ring $ \NH_{k, ħ=0} $ being zero. Indeed, the condition reads $ μ^0_C = 0 $ where
\begin{equation*}
μ^0_C = \sum_{l ≥ 1} μ^l_{\Add \NH_k} (δ, …, δ) = μ^2_{\Add \NH_k} (δ, δ) = - (\sum_{j = 1}^I μ^2_{\NH_{k, ħ=0}} (δ_{ij}, δ_{jl}))_{i, l}.
\end{equation*}
Let us recall from \autoref{sec:prelim-liou-defo} that the category $ \Tw\NH_k $ has exactly the same objects as $ \Tw\NH_{k, ħ=0} $. In particular, every complex $ C ∈ \Tw\NH_{k, ħ=0} $ gives rise to an object $ C ∈ \Tw\NH_k $ with curvature $ -δ^2 $, where the multiplication is evaluated in $ \NH_k $. The category $ \Tw'\NH_k $ is a variant of $ \Tw\NH_k $ which allows infinitesimal entries to stand on and below the diagonal of the $ δ $-matrix as well. Testing whether an object $ C = (\bigoplus_{i ∈ ℤ} *[-i]^{⊕l_i}, δ) ∈ \Tw\NH_{k, ħ=0} $ is uncurvable in $ \Tw\NH_k $ comes down to checking whether there is an infinitesimal deformation $ δ' = δ + ε $ so that $ (\bigoplus_{i ∈ ℤ} *[-i]^{⊕l_i}, δ') ∈ \Tw'\NH_k $ has vanishing curvature, in other words $ δ'^2 = 0 $. Note that the entries of $ ε $ need to be concentrated in the same degrees as those of $ δ $, so that the total degree of $ δ' $ is one.

\begin{example}
If $ A ∈ \NH_{k, ħ=0} $ is any element, then we regard the twisted complex
\begin{equation*}
C_A = \big(*[-1] ⊕ *, \pmat{0 & A \\ 0 & 0}\big) = * \overset{A}{→} * ∈ \Tw \NH_{k, ħ=0}.
\end{equation*}
This twisted complex as an element of $ \Tw\NH_k $ has vanishing curvature. 
\end{example}

\begin{example}
In $ \NH_{k, ħ=0} $ we have the relation $ σ_i X_i - X_{i+1} σ_i = 0 $ for $ i = 1, …, k-1 $. Therefore the following constitutes a twisted complex:
\begin{equation*}
C = \big(*[-2] ⊕ *[-1] ⊕ *[-1] ⊕ *, \pmat{0 & σ_i & -X_{i+1} & 0 \\ 0 & 0 & 0 & X_i \\ 0 & 0 & 0 & σ_i}\big) = * \xrightarrow{\pmat{X_i \\ σ_i}} *^{⊕2} \xrightarrow{\pmat{σ_i & -X_{i+1}}} * ∈ \Tw\NH_{k, ħ=0}.
\end{equation*}
As an object of $ \Tw\NH_k $, this complex has curvature
\begin{equation*}
μ^0_C = \pmat{0 & 0 & 0 & -ħ \\ 0 & 0 & 0 & 0 \\ 0 & 0 & 0 & 0 \\ 0 & 0 & 0 & 0}.
\end{equation*}
Let us show that this twisted complex is not uncurvable. Regard the deformed $ δ $-matrix $ δ' $ given by
\begin{equation*}
* \xrightarrow{\pmat{X_i + ħF \\ σ_i + ħG}} *^{⊕2} \xrightarrow{\pmat{σ_i + ħD & -X_{i+1} + ħE}} *.
\end{equation*}
The condition $ δ'^2 = 0 $ reads
\begin{align*}
0 &= (σ_i + ħD)(X_i + ħF) + (-X_{i+1} + ħE)(σ_i + ħG) \\
&= ħ + ħ(-DX_i + σ_iF - X_{i+1} G + Eσ_i) + \landau(ħ^2).
\end{align*}
In the second row, the multiplications between the individual terms can be carried out in $ \NH_{k, ħ=0} $ instead of $ \NH_k $ since the difference is absorbed into $ \landau(ħ^2) $. Noticing that the polynomial degree of any term in the bracket is at least one, we conclude that the bracket cannot cancel out the first $ ħ $ summand. The complex $ C $ is therefore not uncurvable. It is in fact the most elementary example of such a complex, sourced directly from the deformed relations of $ \NH_k $.
\end{example}

\begin{example}
Any complex with $ δ $ consisting only of either polynomials in the $ X_i $ variables or polynomials in the $ σ_i $ variables is uncurvable in $ \Tw'\NH_k $. Indeed, $ ℂ[X_1, …, X_k] $ embeds into $ \NH_k $ and similarly the subring of $ \NH_{k, ħ=0} $ generated by $ σ_1, …, σ_{k-1} $ embeds into $ \NH_k $.
\end{example}

We shall also regard complexes of the form $ * \xrightarrow{g} * \xrightarrow{f} * $ where $ fg = 0 $.

\begin{example}
\label{ex:nontrivial-uncurving-cases}
Regard the following type of complex in $ \NH_3 $:
\begin{align*}
C &= * \xrightarrow{g} * \xrightarrow{f} *, \\
f &= f_1 + f_2 σ_1 + f_3 σ_2 + f_4 σ_1 σ_2 + f_5 σ_2 σ_1 + f_6 σ_1 σ_2 σ_1, \\
g &= g_1 + σ_1 g_2 + σ_2 g_3 + σ_1 σ_2 g_4 + σ_2 σ_1 g_5 + σ_1 σ_2 σ_1 g_6.
\end{align*}
The complex $ C $ lies in $ \Tw \NH_{k, ħ=0} $ if and only if $ fg = 0 $, which is equivalent to the following set of conditions:
\begin{align*}
f_1 g_1 &= 0, \\
f_2 σ_1 g_1 + f_1 σ_1 g_2 &= 0, \\
f_2 σ_1 σ_2 g_3 &= 0, \\
f_3 σ_2 σ_1 g_2 &= 0, \\
f_2 σ_1 σ_2 σ_1 g_5 + f_3 σ_2 σ_1 σ_2 g_4 + f_4 σ_1 σ_2 σ_1 g_2 + f_5 σ_2 σ_1 σ_2 g_3 &= 0.
\end{align*}
From the first two conditions we observe that $ f_1 = g_1 = 0 $. The third and fourth condition give a distiction into four cases. In each of these cases, the fifth condition takes a particular form. We shall recast the fifth condition in each of these cases by shifting the product $ σ_1 σ_2 σ_1 = σ_2 σ_1 σ_2 $ to the right of the expression:
\begin{itemize}
\item \textbf{Case 1:} $ f_2 = f_3 = 0 $. The fifth condition reads
\begin{equation*}
f_4 g_2^{(σ_1 σ_2 σ_1)} + f_5 g_3^{(σ_1 σ_2 σ_1)} = 0.
\end{equation*}
\item \textbf{Case 2:} $ f_2 = g_2 = 0 $. The fifth condition reads
\begin{equation*}
f_3 g_4^{(σ_1 σ_2 σ_1)} + f_5 g_3^{(σ_1 σ_2 σ_1)} = 0.
\end{equation*}
\item \textbf{Case 3:} $ g_3 = f_3 = 0 $. The fifth condition reads
\begin{equation*}
f_2 g_5^{(σ_1 σ_2 σ_1)} + f_4 g_2^{(σ_1 σ_2 σ_1)} = 0.
\end{equation*}
\item \textbf{Case 4:} $ g_3 = g_2 = 0 $. The fifth condition reads
\begin{equation*}
f_2 g_5^{(σ_1 σ_2 σ_1)} + f_3 g_4^{(σ_1 σ_2 σ_1)} = 0.
\end{equation*}
\end{itemize}
It seems that all complexes with $ fg = 0 $ are uncurvable in $ \Tw'\NH_k $.
\end{example}

\begin{example}
We shall provide a few concrete examples of the complexes from Case 1 in \autoref{ex:nontrivial-uncurving-cases} together with their uncurvings:
\begin{align*}
(σ_1 σ_2 - σ_2 σ_1 + f_6 σ_1 σ_2 σ_1)(σ_1 + σ_2) &= 0.
\end{align*}
\begin{align*}
\big(X_2^2 σ_1 σ_2 - X_3^2 σ_2 σ_1 + ħ (- X_2 σ_1 + X_2 σ_2 - X_3 σ_1 - X_3 σ_2 - ħ)\big) \quad & \\
× \big(σ_1 X_1^2 + σ_2 X_2^2 + ħ (- X_1 - X_2)\big) &= 0.
\end{align*}
\begin{align*}
(X_2 σ_1 σ_2 - X_3 σ_2 σ_1 - ħ σ_1) (-ħ + σ_1 X_1 + σ_2 X_2) &= 0.
\end{align*}
\begin{align*}
\big(X_2^2 σ_1 σ_2 - X_3^2 σ_2 σ_1 + ħ (- X_2 σ_1 + X_2 σ_2 - X_3 σ_1 - X_3 σ_2 - ħ)\big) \quad & \\
× \big(σ_1 X_1^2 + σ_2 X_2^2 + ħ (-X_1-X_2)\big) &= 0.
\end{align*}
\begin{align*}
\bigg(X_2^3 σ_1 σ_2 - X_3^3 σ_2 σ_1 + ħ \big((- X_2^2 - X_2 X_3 - X_3^2) σ_1 \quad & \\
+ (X_2^2 - X_3^2 + X_1 X_2 - X_1 X_3) σ_2 - ħ (X_1 + X_2 + X_3)\big)\bigg) \quad & \\
× \bigg(σ_1 X_1^3 + σ_2 X_2^3 + ħ (- X_1^2 - X_1 X_2 - X_2^2)\bigg) &= 0.
\end{align*}
\begin{align*}
\bigg(X_2^4 σ_1 σ_2 - X_3^4 σ_2 σ_1 + ħ \big((- X_2^3 - X_2^2 X_3 - X_2 X_3^2 - X_3^3) σ_1 + (X_2^3 - X_3^3) σ_2\big) \quad & \\
+ ħ \big(ħ (- X_1^2 - X_1 X_2 - X_1 X_3 - X_2^2 - X_2 X_3 - X_3^2)\big) \quad & \\
+ ħ \big((X_1^2 X_2 - X_1^2 X_3 + X_1 X_2^2 - X_3 X_3 X_1) σ_2\big)\bigg) \quad & \\
× \bigg(σ_1 X_1^4 + σ_2 X_2^4 + ħ (- X_1^3 - X_1^2 X_2 - X_1 X_2^2 - X_2^3)\bigg) &= 0.
\end{align*}
\end{example}

\printbibliography

\end{document}